%% file: Lecture_notes_Otto_2.tex
\documentclass[a4paper,10pt]{amsart}
\usepackage[utf8]{inputenc}
\usepackage[english]{babel}  
\usepackage[T1]{fontenc}     

\usepackage{amsfonts}
\usepackage{amsmath}
\usepackage{amssymb}
\usepackage{amsthm}
\usepackage{bm} 			
\usepackage{dsfont} 			
\usepackage{esint}			

\usepackage[pdftex]{graphicx}
\usepackage{graphics}
\usepackage{float}
\usepackage{tikz}

\usepackage{url}

\def\parag#1{\paragraph{\textit{#1}}}

\usepackage{color} 			
\usepackage[sort,nocompress]{cite} 	

\input{Commandes.tex}
\input{CommandesANG.tex}

\input{Commandes_adhoc.tex}

\newtheorem{theorem}{Theorem}[section]
\newtheorem{corollary}[theorem]{Corollary}
\newtheorem{lemma}[theorem]{Lemma}

\newtheorem{proposition}[theorem]{Proposition}
\newtheorem*{proposition*}{Proposition}

\newtheorem*{theorem*}{Theorem}
\newtheorem*{lemma*}{Lemma}
\newtheorem*{corollary*}{Corollary}
\newtheoremstyle{TheoremNum}
    {\topsep}{\topsep}              		
    {\itshape}                      		
    {}                              		
    {\bfseries}                     		
    {.}                             		
    { }                             		
    {\thmname{#1}\thmnote{ \bfseries #3}}	
\theoremstyle{TheoremNum}

\theoremstyle{remark}

\title[The annealed Calder\'on-Zygmund estimate in homogenization]{The annealed Calder\'on-Zygmund estimate as convenient tool in quantitative stochastic homogenization}
\author{Marc Josien and Felix Otto}
\address{Max-Planck-Institut f\"ur Mathematik in den Naturwissenschaften, Inselstrasse 22, 04103 Leipzig, Germany}

\begin{document}
  
\maketitle
\footnotetext[1]{Max-Planck-Institut f\"ur Mathematik in den Naturwissenschaften, Inselstrasse 22, 04103 Leipzig, Germany}

\section*{Abstract}

This article is about the quantitative homogenization theory
of linear elliptic equations in divergence form with random coefficients. 
We derive gradient estimates on the homogenization error, 
i.~e.~on the difference between the actual solution and the two-scale
expansion of the homogenized solution, both in terms of strong norms
(\textit{oscillation}) and weak norms (\textit{fluctuation}). These estimates
are optimal in terms of scaling in the ratio between the microscopic 
and the macroscopic scale.


The purpose of this article is to highlight the usage of the recently
introduced \textit{annealed} Calder\'on-Zygmund (CZ) estimates in obtaining the above, 
previously known, error estimates. 
Moreover, the article provides a novel proof of these annealed CZ estimate
that completely avoids quenched regularity theory, but rather relies on functional analysis.
It is based on the observation that even on the level of operator norms, 
the Helmholtz projection is close to the one for the homogenized coefficient (for which annealed
CZ estimates are easily obtained).


In this article, we strive for simple proofs, and thus restrict ourselves to
ensembles of coefficient fields that are local transformations of Gaussian random fields
with integrable correlations and H\"older continuous realizations. 
As in earlier work, we use the natural objects from the general theory of homogenization,
like the (potential and flux) correctors and the homogenization commutator. 
Both oscillation and fluctuation estimates rely on a sensitivity calculus, 
\textit{i.e.} on estimating how sensitively the quantity of interest does depend on an 
infinitesimal change in the coefficient field, which is fed into the Spectral Gap inequality.
In this article, the annealed CZ estimate is the only form in which elliptic
regularity theory enters.

\parag{Keywords:} Stochastic homogenization, Calder\'on-Zygmund estimates, corrector estimates, oscillations, fluctuations

\parag{AMS classification:}
35B27, 
35R60, 
35J15, 
35B45, 
42B15 

\newpage

{
\setlength{\parskip}{0\baselineskip}
\tableofcontents
}

\newpage

\section{Introduction}

\subsection{The role of annealed CZ estimates$^{\ref{1}}$}
\footnotetext[1]{\label{1} The jargon of quenched and annealed is common in models of statistical physics, like the Ising model with random interaction strength or, more relevant for us, a random walk in a random environment.
We use it freely here. 
Ultimately, the words come from metallurgy.}
We start with a preamble on language. Homogenization capitalizes 
on a separation between two length scales: a micro-scale set by the coefficient field $a$
and a macro-scale set by the other data (\textit{e.g.} domain, right-hand side). 
In the case of stochastic homogenization, 
the micro-scale is set by a correlation length of the ensemble $\langle\cdot\rangle$;
in this article, it is of the same order as the (potentially much smaller) 
scale up to which the coefficient field $a$ varies little.
In this article, the domain is always the whole $d$-dimensional space $\mathbb{R}^d$, 
so the macro-scale is just set by the scale intrinsic to the right-hand side.
Throughout the entire article, the micro-scale is set to be unity 
-- the reader will (almost) nowhere find the $\epsilon$ dear to the mathematical
homogenization literature!


There is an intimate connection between elliptic regularity theory 
and quantitative stochastic homogenization. This became already apparent 
in the first work on error estimates \cite{Yurinskii}, where Nash's heat kernel decay
for general uniformly elliptic coefficient fields $a$ is used. 
Like for much of the last decade's contributions, the present article
follows the strategy introduced in \cite{NaddafSpencer_1998} of using a Spectral Gap inequality, 
which is assumed to hold for the ensemble $\langle\cdot\rangle$ under consideration 
and quantifies its ergodicity, to capture stochastic cancellations. 
It is in the same article \cite{NaddafSpencer_1998} 
where CZ estimates for $-\nabla\cdot a\nabla$
in form of the $\LL^4_{\mathbb{R}^d}$-boundedness of the $a$-Helmholtz projection 
$-\nabla(-\nabla\cdot a\nabla)^{-1}\nabla\cdot$ entered the stage
as a part of the sensitivity estimate.
In that work, this boundedness was guaranteed by assuming a small ellipticity contrast, 
appealing to Meyers' perturbative argument. 
In \cite{GloriaOtto_2011}, the assumption of small ellipticity was removed by 
using de Giorgi's theory in order to obtain optimal decay estimates on the Green function $G(a;x,y)$, upgraded to optimal decay of its mixed derivative with help of Caccioppoli's estimate. 


Those were \textit{quenched} estimates on the \textit{quenched} Green function $G(a;x,y)$, that is, estimates valid for almost every realization $a$ of the coefficient field,
but they required some macroscopic spatial averaging in case of the derivatives. 
Optimal {\it pointwise} decay estimates on the mixed 
derivative of the Green function were first obtained in \cite{Marahrens_Otto};
these estimates can only hold with high probability, and were formulated as \textit{annealed}
estimates
in the sense that all algebraic moments with respect to the ensemble
$\langle\cdot\rangle$ have the desired behavior.
Note that the mixed derivative of the Green function is the kernel of the $a$-Helmholtz
projection. Hence, loosely speaking, the estimates in \cite{Marahrens_Otto} 
yield ``up to a logarithm'' the boundedness of the $a$-Helmholtz projection 
from $\LL^p_{\mathbb{R}^d}(\LL^r_{\langle\cdot\rangle})$ to 
$\LL^p_{\mathbb{R}^d}(\LL^{r'}_{\langle\cdot\rangle})$, with an arbitrarily small loss $r'<r$
in the stochastic integrability. We call this family of estimates 
\textit{annealed CZ estimates}. 


The result in \cite{Marahrens_Otto} showed that with overwhelming probability,
the random Green function has no worse decay than the constant-coefficient one.
This emergence of large-scale regularity is well-known in periodic homogenization 
since the work \cite{AvellanedaLin} (see the recent monograph \cite{Shen}), where like in classical regularity theory, 
a Campanato iteration propagates H\"older regularity over dyadic scales, 
but unlike classical theory, the regularity of the homogenized operator 
is propagated {\it downward} in scales, 
appealing to qualitative homogenization itself. In \cite{Armstrong},
this approach was adapted to the random case, by substituting the compactness
argument by a more quantitative one, establishing the pivotal large-scale Lipschitz estimate. 
This was refined to large-scale quenched $\CC^{\alpha}$ (Schauder) estimates for the $a$-Helmholtz
projection in \cite{Gloria_Neukamm_Otto_2019}. Here, by a large scale quenched regularity theory we mean
that there exists a stationary random radius $\rstar\sim 1$ starting
from which the theory holds almost surely (which is reminiscent of the 
random time scale starting from which Euclidean heat kernel decay holds in percolation theory). 
In \cite{Armstrong_CZ_2016}, the large-scale Lipschitz estimate was used to establish large-scale quenched CZ estimates, which however had a small loss in terms of spatial integrability. The optimal large-scale quenched CZ theory, that is, boundedness of the $a$-Helmholtz projection in $\LL^p_{\mathbb{R}_x^d}(\LL^2_{\Boule_{\rstar(x)}(x)})$ (loosely speaking,
$\LL^2$ up to scales $\rstar$, $\LL^p$-from scales $\rstar$ onwards) was first obtained
in \cite{Gloria_Neukamm_Otto_2019}. A simpler proof, based on modern kernel-free approaches to 
classical CZ theory, may be found in \cite{Armstrong_book_2018} and \cite[Proposition 6.4]{DuerinckxOtto_2019}. 


While in many error estimates of stochastic homogenization, weighted $\LL^{2+\epsilon}$-estimates
are an alternative to $\LL^p$-estimates,
the usage of CZ estimates seems crucial for the theory of fluctuations
as conceived in \cite{DuerinckxGO_2016}, where again $\LL^4_{\mathbb{R}^d}$ comes
up naturally. As worked out in \cite{DuerinckxOtto_2019},
in particular in the elaborated case of a higher-order theory of fluctuations considered there,
the annealed version of CZ estimates is very convenient. 
In that paper, see \cite[Section 6]{DuerinckxOtto_2019}, the annealed estimates were derived from the quenched ones, appealing to the good stochastic integrability of $\rstar$ established in \cite{Gloria_Neukamm_Otto_2019}, and appealing to the same tool that was used to derive the quenched estimates from the Lipschitz estimates.


In this paper, we completely avoid the large-scale
quenched regularity theory, be it Lipschitz, Schauder, or CZ.
We derive the annealed CZ theory in a single step by, loosely speaking,
establishing that the $a$-Helmholtz projection is close in the corresponding {\it operator norm}
to the Helmholtz projection for the homogenized coefficient $\abar$. This amounts 
to a particularly strong version of an estimate of the homogenization error.
Indeed, on the 
level of the Helmholtz projection, which itself has no regularizing effect,
the homogenization error has a chance of being small
only if the r.~h.~s.~lives on scales $\gg 1$. Estimating the homogenization error
in an operator norm thus requires splitting the r.~h.~s.~into a large-scale part
and its remainder. This is best done \textit{intrinsically} and here is carried out
by considering the family of elliptic operators with massive term 
$\frac{1}{T}-\nabla\cdot a\nabla$, which provides the flexibility of choosing 
a length scale $\sqrt{T}$.
The main challenge arises from the fact that the estimate on the homogenization error
relies on an estimate of the corrector which because of its stochastic nature 
reduces the exponent $r$ in $\LL^p_{\mathbb{R}^d}(\LL^r_{\langle\cdot\rangle})$
and thus prevents a straightforward buckling. The key idea is to trade-in smallness in terms of
$\sqrt{T}$ to compensate the loss of integrability in terms of $r$ in form of a (real)
interpolation argument. We hope that this functional-analytic argument finds applications
in situations where the more PDE-based argument does not apply.


\subsection{The developments in quantitative stochastic homogenization}
  Qualitative stochastic homogenization has been studied for forty years \cite{Varadhan_1979,JKO}.
  Nevertheless, as opposed to the periodic framework, where optimal estimates of the homogenization error were obtained quite early (see \textit{e.g.} \cite{AvellanedaLin,Tartar} and the monograph \cite{Shen}), quantitative stochastic homogenization remained long out of reach.
  On a qualitative level, for instance when it comes to the construction
  of the corrector, the periodic case and the random ergodic case can be
  treated in parallel, by lifting the corrector equation to the torus and
  the probability space, respectively. For a quantitative treatment,
  it is crucial that the torus admits a Poincar\'e estimate, while there
  is no analogue in the general ergodic case, so that the qualitative sublinearity
  cannot be easily improved upon.
  
  There has been substantial progress in quantitative stochastic homogenization in the last decade: A suitable substitute for the missing Poincar\'e estimate on probability space has been identified in the mathematical physics literature \cite{ConlonNaddaf,NaddafSpencer_1998} in form of a spectral gap estimate, which quantifies ergodicity. This was taken up and combined with more elliptic regularity theory in \cite{GloriaOtto_2011} to treat, for the first time, the standard non-perturbative situation; \cite{Gloria_Otto,Gloria_Neukamm_Otto_2019} are key papers that pushed this approach further. An alternative approach based on quantified ergodicity in terms of a finite range assumption has been introduced in \cite{Armstrong}, extending the regularity theory \cite{AvellanedaLin} from periodic to random, and culminating in the monograph \cite{Armstrong_book_2018}.

  One merit of homogenization is the reduction in numerical complexity it entails, reducing the need to resolve the micro-scale to a couple of model calculations. Numerics is indeed a major motivation behind a quantitative theory for stochastic homogenization, see the review \cite{Anantharaman_2012}.
  In particular, devising efficient methods for computing the homogenized coefficient from the knowledge of the heterogeneous media has attracted much attention.
  See in particular \cite{Gloria_Otto,Mourrat_2019} for general methods, and the attempts \cite{Minvielle_2015,Fischer_ARMA_2019,CloJoOtXu_2020} to reduce different sources of error (namely the random and the systematic one) in specific settings.

\subsection{Organization of the paper}

Next to providing a novel, and self-contained, proof for the annealed CZ estimates,
see Proposition \ref{PropCZ},
the second purpose of this paper is to display their benefit in error estimates.
To this aim, we display how the annealed CZ estimate is crucial in obtaining two natural estimates of the homogenization error on the level of the gradient, namely in strong norms (\textit{oscillations}), \textit{cf.}~Corollary \ref{Cor1}, and in weak norms (\textit{fluctuations}), \textit{cf.}~Proposition \ref{Propfluct}.
These results are not new, see \cite{Gloria_Neukamm_Otto_2019} for oscillations, and \cite{DuerinckxGO_2016} for fluctuations. The usage of the
annealed CZ estimates to derive fluctuation estimates 
was recently introduced in \cite{DuerinckxOtto_2019}. 


However, the statements given here are more transparent 
and the proofs are simpler because we allow ourselves to work in the simplest class
of ensembles $\langle\cdot\rangle$, 
namely those that can be obtained by a pointwise Lip\-schitz transformation 
of stationary Gaussian ensembles with integrable correlations (which is the crucial large-scale
assumption) and H\"older continuous realizations (which is a convenient small-scale assumption).
Also, presenting the corrector estimates, \textit{cf.}~Proposition \ref{Propcorr}, alongside
with the oscillation and fluctuation estimates allows us to highlight the many similarities:
the use of an annealed CZ estimate in the perturbative regime 
in Proposition \ref{Propcorr}, the explicit representation of the Malliavin
derivatives in Proposition \ref{Propcorr} and \ref{Propfluct} via auxiliary dual problems
on which the annealed CZ estimates are applied, 
and the use of the $\LL^p$-version of the Spectral Gap estimate.


This presentation, together with the new proof of Proposition \ref{PropCZ}, 
stresses the role of functional analysis techniques for quantitative stochastic 
homogenization, and minimizes the use of finer PDE-ingredients, 
which hopefully resonates well with some audiences.
In this sense, the article has also the nature of a review article
-- in fact, it grew out of a mini-course the second author gave in Toulouse in the spring of 2019.
We now address the content of the individual sections.


In Section \ref{SecHcv}, we introduce the deterministic part of the setting: 
Next to the uniformly elliptic coefficient fields $a$ on $\mathbb{R}^d$, we recall the notions of the (extended) correctors $(\phi_i,\sigma_i)_{i=1,\cdots,d}$, of the homogenization commutator, and of the two-scale expansion. In order to stress the importance of a few key identities involving these quantities, which will be crucially used in later sections, and for the convenience of the reader not familiar with homogenization, we derive two elementary statements in the spirit of Tartar's $H$-convergence in Proposition \ref{PropHcon}.


In Section \ref{SecCorrEstim}, we get to the stochastic part of the setting by introducing the class of ensembles $\langle\cdot\rangle$ of coefficient fields $a$ considered here, and derive the two main features that we actually need: 
1) An $\LL^p_{\langle\cdot\rangle}$-version of the Spectral Gap estimate with the carr\'e-du-champs $\LL^2_{\mathbb{R}_x^d}(\LL^1_{\Boule_1(x)})$ (i.e.~$\LL^1$ on scales up to $1$ and $\LL^2$ on scales larger than $1$) 
and 2) classical Schauder and CZ regularity theory for the elliptic operator on scales up to $1$ with overwhelming probability, see Lemma~\ref{LemSG}.

In Section \ref{SecCorrEstim2}, we state and prove the corrector estimates, which take a particularly simple form thanks to Lemma \ref{LemSG}, see Proposition \ref{Propcorr}. 
In particular, this allows for a simple passage from the estimate of spatial averages of the gradient $\nabla \phi_i$ to the estimates of increments of $\phi_i$, and for a simple passage from the flux $q_i:=a(e_i+\nabla\phi_i)$ to $\nabla\sigma_i$. 
Modulo this easy post-processing, Proposition \ref{Propcorr} relies on the estimate of spatial averages of the field/flux pair $(\nabla\phi_i,q_i)$, which requires the buckling of a stochastic and a deterministic estimate. 
The stochastic estimate relies on a representation of the Malliavin derivative of the spatial average with help of the solution $\nabla v$ of a dual problem, the estimate of $\nabla v$ via the perturbative version of the annealed CZ estimate, and the $\LL^p$-version of the Spectral Gap estimate.
The deterministic estimate relies on an elementary but subtle regularity property of $a$-harmonic functions, here applied to the harmonic coordinate $x_i+\phi_i$, which is reminiscent of \cite[Lem.\ 4]{Bella_Giunti_Otto_2017}. 
While none of the elements is novel, the combination here is slightly more efficient than in the existing literature.

In Section \ref{Sec:Osc} , as a simple application of Proposition \ref{Propcorr}, we derive the oscillation estimates, \textit{cf.}~Corollary \ref{Cor1}.
In the case of $\LL^p_{\mathbb{R}^d}$ with $p\not=2$, we use Proposition \ref{PropCZ} in a straightforward way.


In Section \ref{SecFluctu}, we state and prove the fluctuation estimates, \textit{cf.}~Proposition \ref{Propfluct}. 
In the proof, we stress the similarity with the proof of Proposition \ref{Propcorr}. 
Proposition \ref{PropCZ}, in its non-perturbative part, now enters in a substantial way.


In Section \ref{SecAnnealed}, we finally establish the annealed CZ estimate,  \textit{cf.}~Proposition \ref{PropCZ}.
In the perturbative regime, this is done via Meyers' argument and does not involve a loss of stochastic integrability. 
As discussed in Section \ref{SS_AnnealedCZ}, we derive this non-perturbative part of Proposition \ref{PropCZ} from Proposition \ref{Propcorr} -- this is not circular, since in the latter proposition, only the perturbative part of Proposition \ref{PropCZ} enters. 
We will need a version of Proposition \ref{PropCZ} for $-\nabla\cdot a \nabla$ replaced by the massive operator $\frac{1}{T}-\nabla\cdot a\nabla$ which will be established along the same lines (see Proposition \ref{PropmassCZ}).
The strategy for the proof of Proposition \ref{PropCZ} is discussed in all detail there.
In Section \ref{SecLocal}, we briefly sketch the alternative strategy relying on quenched CZ estimates.

The appendix contains several auxiliary statements, in particular those that have
to do with the usage of local regularity on scales less than $1$.
  
  \subsection{Notations}
  
    Throughout this article, a few conventions will be used without further notice.
    
    We denote by $e_i$ the canonical basis of $\R^d$.
    We use Einstein's convention for summation.
    The symbol $\cdot$ denotes the scalar product. 
    In particular, $\nabla \cdot$ stands for the divergence operator.
    The divergence of a matrix is taken with respect to the second coordinate, that is,
    \begin{align*}
      e_j \cdot (\nabla \cdot \sigma_i)=\partial_k \sigma_{ijk}.
    \end{align*}
    For simplicity, we omit the parentheses when considering operators of the type $\nabla \cdot a\nabla$: We write $\nabla \cdot a \nabla u$ for $\nabla \cdot ( a \nabla u)$.
    If $f$ and $g$ are vector fields with $m$ and $n$ coordinates in the canonical basis, respectively, we denote by $(f,g)$ the $(m+n)$-dimensional vector field obtained by concatenating the coordinates of $f$ and $g$.
    For simplicity, we omit to specify explicitly in the functional space the dimension of the vector field it refers to. For example, we may abusively write $\LL^2(\R^d)$ instead of $\LL^2(\R^d,\R^d)$.
    
    The symbol ``$\lesssim_{\alpha_1,\cdots,\alpha_n}$'' reads ``$\leq C$ for a constant $C$ depending only on the tuple $(\alpha_1,\cdots,\alpha_n)$ of previously defined parameters''.
    Similarly, the symbol ``$\ll_{\alpha_1,\cdots,\alpha_n} 1$'' (``$\gg_{\alpha_1,\cdots,\alpha_n} 1$'') in an assumption reads ``sufficiently small'' (``sufficiently large'', respectively) in the sense of being below a threshold that depends only on the parameters $\alpha_1,\cdots,\alpha_n$.
    For simplicity, in the course of the proofs, the subscript might be omitted.
  
    We denote by $\Boule_R(x)$ the ball of center $x$ and radius $R$.
    When it is clear from the context, the variables $x$ and $R$ may be omitted (by default, we have $\Boule_R:=\Boule_R(0)$).
    Also, if $\Boule$ is the ball $\Boule_R(x)$ we abusively denote by $\frac{1}{2}\Boule$ the ball $\Boule_{R/2}(x)$.
    The function $\mathds{1}_{\Omega}$ is the characteristic function of a domain $\Omega$.
    When the set of integration is omitted, we use the convention that the integral is taken on the whole space $\R^d$; namely $\int_{\R^d} f:=\int f$.
    Also, unless otherwise stated, the equations we consider are satisfied on $\R^d$.
    We denote the homogeneous H\"older semi-norm and the H\"older norm:
    \begin{align*}
    [ h ]_{\CC^{\beta}(\Omega)}:= \sup_{x,x' \in \Omega} \frac{|h(x)-h(x')|}{|x-x'|^\beta} \et \|h\|_{\CC^{\beta}(\Omega)} = [ h ]_{\CC^{\beta}(\Omega)} + \|h\|_{\LL^\infty(\Omega)}.
    \end{align*}
    
    If $p \in [1,\infty]$ is an exponent, its conjugated exponent is denoted $p^\star$ and is defined by $1/p+1/p^\star=1$.
    
\section{General notions in homogenization}\label{SecHcv}

  \subsection{Correctors, homogenization commutator, and two-scale expansion}

    This article is concerned with the following elliptic equation in divergence form
    \begin{align*}
      \nabla \cdot ( a \nabla u + f )=0 \dans \R^d,
    \end{align*}
    where $a$ is a $\lambda$-uniformly elliptic coefficient field, \textit{i.e.} for some $\lambda>0$, there holds
    \begin{align}
      \label{Ellipticite}
      \xi \cdot a(x) \xi \geq \lambda |\xi|^2  \quad\et \quad \xi \cdot a^{-1}(x) \xi \geq |\xi|^2.
    \end{align}
    While we adopt scalar notation and language, everything remains unchanged if $u$ has values in some finite-dimensional  vector space.
    
    In most general terms, homogenization means relating the heterogeneous coefficient field $a$ to a homogeneous coefficient $\abar$. In homogenization, this does not rely on the smallness of $a-\abar$ in a strong or (usual) weak sense. Almost tautologically, the suitable topology is provided by the notion of $H$-convergence \cite[Def.\ 1.2.15, p.\ 25]{Allaire}. On a purely \textit{algebraic} level, this topology is captured on the level of an $a$-Helmholtz decomposition, giving rise to a scalar potential (inducing the definition of the corrector $\phi_i$) and a vector potential (inducing the definition of the flux corrector $\sigma_i$); for $d\not=3$, the vector potential is replaced by a skew symmetric tensor field (i.~e.~an alternating $(d-2)$-form).
    
    We decompose $(a-\abar)e_i$ into two parts as follows:
    \begin{align}
      \label{Id02}
      (a-\abar) e_i&=- a \nabla \phi_i + \nabla \cdot \sigma_i,
    \end{align}
    in the sense of $a_{ji} - \abar_{ji} = - a_{jk} \partial_k \phi_i + \partial_k \sigma_{ijk}$.
    Here, $\phi_i$ is a scalar field and $\sigma_i$ is a skew-symmetric tensor field:
    \begin{align}
      \label{Id01}
      \sigma_{ijk}&=-\sigma_{ikj}.
    \end{align}
    Equation \eqref{Id02} is an $a$-Helmholtz decomposition of $(a-\abar) e_i$ in the following sense: The first r.\ h.\ s.\ term of \eqref{Id02} is the product of $a$ and a curl-free vector field (which, in case of the trivial topology of $\R^d$, amounts to a gradient field), whereas the second one is a divergence-free vector field.
    This is an $a$-Helmholtz decomposition since the identification of the $1$-form $\nabla \phi_i$ with a $(d-1)$-form is done via the ``metric'' $a$.
    We introduce the analogous objects $(\phi_i^\star,\sigma_i^\star)$ on the level of the transposed coefficient field $a^\star$:
    \begin{align}
      \label{Id02prime}
      (a^\star-\abar^\star) e_i&=- a^\star \nabla \phi^\star_i + \nabla \cdot \sigma^\star_i.
    \end{align}
    
    Decomposition \eqref{Id02} immediately implies the classical equation for $\phi_i$:
    \begin{align}
      \label{Id04}
      -\nabla \cdot a \nabla (x_i+ \phi_i)&=
      \nabla \cdot( \abar e_i + \nabla \cdot \sigma_i )\overset{\eqref{Id01}}{=}0.
    \end{align}
    Thus, the functions $\phi_i$ are called \textit{correctors}, for they correct the $\abar$-harmonic coordinates $x_i$ to $a$-harmonic functions $x_i+\phi_i$. (By \textit{$a$-harmonic} functions, we mean functions $u$ such that $\nabla \cdot a \nabla u=0$.) It is clear that this notion is only useful when the correctors grow sublinearly.
    The corrector $\phi_i$ induces a natural quantity, namely the flux (or current) $q_i$ defined by
    \begin{align}\label{Defqi}
      q_i:= a (e_i+\nabla \phi_i ).
    \end{align}
    
    We call the functions $\sigma_i$ \textit{flux correctors}, for they correct the flux as follows:
    \begin{align}\label{Num:1}
      q_i=a(e_i+\nabla \phi_i)=\abar e_i + \nabla \cdot \sigma_i.
    \end{align}
    Note that \eqref{Id02} does not determine $\sigma$ (not even up to the addition of a constant as in the case of $\phi$).
    In fact, $\sigma_i$, which should be considered as an alternating $(d-2)$-form only determined up to a $(d-3)$-form. A natural choice of gauge is given by $-\Delta \sigma_i = \nabla \times q_i$, where we use the $(d=3)$-notation as an abbreviation for
    \begin{align}
      \label{Id03}
      -\Delta \sigma_{ijk}&=\partial_j q_{ik} - \partial_k q_{ij}.
    \end{align}
    Note that \eqref{Id03} is indeed consistent with \eqref{Id02} in the form of \eqref{Num:1} because in view of $\nabla \cdot q_i=0$ it yields $-\Delta ( \nabla \cdot \sigma_i - q_i )=0$.
    We will henceforth call the pair $(\phi,\sigma)$ \textit{extended correctors}.

    The extended correctors appear in two fundamental identities, which are proved at the end of the section. 
    The first one is at the core of homogenization in form of $H$-convergence: The effective tensor $\abar$ is supposed to relate macroscopic averages of fields to macroscopic averages of fluxes. 
    It is easiest to make this precise on the level of $a$-harmonic functions $u$, for which $\nabla u$ is (up to the sign) the field and $a\nabla u$ the flux, so that the \textit{homogenization commutator} $a\nabla u-\abar\nabla u$    should have small macroscopic averages. This is embodied by the following formula
    \begin{align}\label{1a}
      e_j\cdot(a\nabla u-\abar\nabla u)=-\nabla\cdot((\phi_j^\star a+\sigma_j^\star)\nabla u)
      \quad\mbox{provided}\;\nabla\cdot a\nabla u=0.
    \end{align}
    Indeed, due to its divergence form, macroscopic averages of the r.~h.~s.~are smaller than those of each term on the l.~h.~s., provided of course that the extended correctors grow sublinearly.
    
    The second fundamental identity intertwines the variable-coefficient operator $-\nabla\cdot a\nabla$ with the constant-coefficient operator $-\nabla\cdot \abar\nabla$ via what is called the two-scale expansion, namely
    \begin{align}\label{2scale}
      \tilde u:=(1+\phi_i\partial_i)\bar u,
    \end{align}
    where here and in the sequel we use Einstein's convention of implicit summation over repeated indexes. This intertwining is embodied by the following formula
    \begin{align}\label{1b}
      \nabla\cdot a\nabla(1+\phi_i\partial_i)\bar u=\nabla\cdot\abar\nabla \bar u
      +\nabla\cdot((\phi_ia-\sigma_i)\nabla\partial_i\bar u).
    \end{align}
    Note that for affine $\bar u$, \eqref{1b} turns into the familiar \eqref{Id04}. For general macroscopically varying $\bar u$, the two-scale expansion \eqref{2scale} provides a microscopic modulation adapted to the variable-coefficient operator. Indeed, the second r.~h.~s.~term of \eqref{1b} will be much smaller than the first r.~h.~s.~term because of the additional derivative on the macroscopically varying $\bar u$, again provided the extended correctors grow sublinearly.

    Typically, (\ref{1b}) is used in the following way: Given a square integrable $f$, let the square integrable $\nabla u$ and $\nabla\bar u$ solve
    \begin{align}
      \nabla\cdot(a\nabla u+f)&=0,\label{Complexe}\\
      \nabla\cdot(\abar\nabla \bar u+f)&=0,\label{Simple}
    \end{align}
    where here and in the sequel we think of these equations as being posed on the entire $\mathbb{R}^d$ if not stated otherwise. Then we have
    \begin{align}\label{12}
      -\nabla\cdot a\nabla\big(u-(1+\phi_i\partial_i)\bar u\big)=
      \nabla\cdot((\phi_ia-\sigma_i)\nabla\partial_i\bar u).
    \end{align}
    Summing up, we learn from \eqref{1a} that $a\nabla u-\abar\nabla u$ is small in the weak topology, while we learn from \eqref{12} that $\nabla u-\nabla(1+\phi_i\partial_i)\bar u$ is small in the strong topology. We will make this precise in Section \ref{SecHcv_2}.
    
      Let us close with a numerical consideration.
      Homogenization in general and the two-scale expansion \eqref{2scale} in particular are useful from a numerical point of view.
      Indeed, the problem \eqref{Simple} defining $\ubar$ is far simpler than \eqref{Complexe}, since it involves a homogeneous coefficient instead of a heterogeneous one.
      Also, the homogenized coefficient $\abar$ and the correctors $\phi_i$ are independent of $f$ and may be recovered from \eqref{Id04}.
      If $a$ is periodic, solving the latter equation is simple, since it reduces to a problem set on a periodic cell.      
      In the stochastic context, recovering $\abar$ and the correctors $\phi_i$ is more difficult, since \eqref{Id04} is posed on the entire $\R^d$ \cite{Anantharaman_2012}.
      However, it is definitely an efficient strategy when one needs to solve \eqref{Complexe} for a large number of right-hand sides $\nabla \cdot f$, as may be the case inside an optimization procedure.
      In this case, $\bar{a}$ and $\phi_i$ can be computed --at least approximately-- and stored during an ``offline'' stage while \eqref{Complexe} is solved in an ``online stage''.
    
    \begin{proof}[Argument for \eqref{1a} and \eqref{1b}]
      First note that for any function $v$ there holds
      \begin{align}\label{Id1}
	\nabla \cdot (v \nabla \cdot \sigma_i)=(\nabla \cdot \sigma_i) \nabla v=-\nabla \cdot ( \sigma_i \nabla v).
      \end{align}
      Indeed, written componentwise, \eqref{Id1} takes the form
      \begin{align*}
	\partial_j \lt(v \partial_k \sigma_{ijk} \rt)\overset{\eqref{Id01}}{=}
	\partial_k\sigma_{ijk} \partial_j v 
	=\partial_k (\sigma_{ijk} \partial_j v  ) - \sigma_{ijk} \partial_k\partial_j v 
	\overset{\eqref{Id01}}{=} 
	-\partial_k (\sigma_{ikj} \partial_j v  ).
      \end{align*}
      
      We now prove \eqref{1a} by a straightforward computation:
      \begin{align*}
	e_i \cdot ( a \nabla u - \abar \nabla u)
	=~&((a^\star-\abar^\star)e_i ) \cdot \nabla u
	\\
	\overset{\eqref{Id02prime}}{=}~&(-a^\star \nabla \phi_i^\star + \nabla \cdot \sigma_i^\star )\cdot \nabla u
	\\
	=~&-\nabla \phi_i^\star \cdot a \nabla u +(\nabla \cdot \sigma_i^\star) \nabla u
	\\
	\overset{\eqref{Id1}}{=}&\phi_i^\star \nabla \cdot a\nabla u -\nabla \cdot ( \phi_i^\star a \nabla u) -\nabla \cdot (\sigma^\star_i \nabla u ).
      \end{align*}
      
      Here comes the proof of \eqref{1b}.
      We note
      \begin{align*}
	a \nabla ( (1+\phi_i\partial_i) \ubar )
	=~&\partial_i \ubar a ( e_i + \nabla \phi_i) + \phi_i a \nabla \partial_i \ubar
	\\
	\overset{\eqref{Id02}}{=}\;&\partial_i \ubar ( \abar e_i + \nabla \cdot\sigma_i ) + \phi_i a \nabla \partial_i \ubar
	\\
	=~&\abar \nabla \ubar + \partial_i \ubar\nabla \cdot\sigma_i + \phi_i a \nabla \partial_i \ubar.
      \end{align*}
      Applying the divergence, we obtain \eqref{1b} from \eqref{Id1}.
    \end{proof}

  \subsection{$H$-convergence}\label{SecHcv_2}

  The main goal of this section is to familiarize the reader with the notions introduced in the previous section; later parts of the paper do not rely on it.
  Proposition \ref{PropHcon} below (see \cite[Prop.\ 1]{Gloria_Neukamm_Otto_2019}) relates the relative error in the two-scale expansion to the sublinearity of the extended corrector. 
  The qualitative arguments establishing Proposition 2.1 below are at the core of the notion of $H$-convergence.
  \begin{proposition}\label{PropHcon}
    Assume that the coefficient field $a:\R^d \rightarrow \R^{d\times d}$ satisfies \eqref{Ellipticite} and admits $\abar \in \R^{d \times d}$ as a constant homogenized coefficient and $(\phi,\sigma)$ and $(\phi^\star,\sigma^\star)$ as extended correctors in the sense of \eqref{Id02} and \eqref{Id02prime}.
    Let $\epsilon>0$ be given. Then, provided $\delta\ll_{d,\lambda,\epsilon} 1$, the following properties hold:
    
    (i) For any ball $\Boule \subset \R^d$ of radius $R>0$, if the extended correctors satisfy the following estimate:
    \begin{align}
      \label{CorrSub}
      \frac{1}{R} \Big(\fint_{\Boule} |(\phi,\phi^\star,\sigma,\sigma^\star)|^2\Big)^{\frac{1}{2}} \leq \delta,
    \end{align}
    then, for any $a$-harmonic function $u$ in $\Boule$ and for any $\abar$-harmonic function $\ubar$ in $\Boule$ so that $u=\ubar$ on the boundary $\partial B$, there holds:
    \begin{align}
      \label{Prop1Estim1}
      \frac{1}{R^2} \fint_{\Boule} (u-\ubar)^2 \leq \epsilon \fint_{\Boule} |\nabla u|^2.
    \end{align}
    
    (ii) Moreover, we have the following interior bound on the error of the two-scale expansion:
    \begin{align}
      \label{Prop1Estim2}
      \fint_{\frac{1}{4}\Boule} |\nabla u -  \partial_i \ubar (e_i + \nabla \phi_i) |^2 \leq \epsilon \fint_{\Boule}|\nabla u|^2.
    \end{align}
  \end{proposition}
  
  \begin{proof}[Proof of Proposition \ref{PropHcon}]
    The proof of \eqref{Prop1Estim1} and \eqref{Prop1Estim2} rely on the two fundamental algebraical formulas \eqref{1a} and \eqref{12}, respectively.
    The proof is divided into two steps.
    First, we make use of \eqref{1a} in order to obtain that $a\nabla u - a \nabla \ubar$ weakly vanishes when $\delta \downarrow 0$.
    Then, by a compactness argument that is reminiscent of the $H$-convergence \cite[Def.\ 1.2.15]{Allaire}, we deduce from this fact  that \eqref{Prop1Estim1} holds.
    In Step 2, thanks to \eqref{12}, we upgrade the latter estimate to \eqref{Prop1Estim2}.
    In this proof, the symbols $\lesssim$ will only depend on $d$ and $\lambda$ (but \textit{not} on $\epsilon$).
    
    \parag{Step 1: Argument for \eqref{Prop1Estim1}}
      We start with a couple of simplifications.
      W.~l.~o.~g. by scaling, we may restrict to $R=1$ and by translation invariance to $x=0$, so that $\Boule=\Boule_1(0)$.
      Also, we assume that $\fint_{\Boule} u=0$, so that, by the Poincaré-Wirtinger inequality,
      \begin{align*}
	\fint_{\Boule} u^2 \lesssim \fint_{\Boule} |\nabla u|^2.
      \end{align*}
      By the triangle inequality combined with the Poincaré inequality, once more the triangle inequality, and the energy estimate, we deduce that
      \begin{equation*}
        \begin{aligned}
          \fint_{\Boule} \ubar^2 \lesssim~& \fint_{\Boule} (\ubar-u)^2 + \fint_{\Boule} u^2
	  \lesssim  \fint_{\Boule} |\nabla (\ubar-u)|^2 + \fint_{\Boule} |\nabla u|^2
	  \\
	  \lesssim~& \fint_{\Boule} |\nabla \ubar|^2 + \fint_{\Boule} |\nabla u|^2 \lesssim \fint_{\Boule} |\nabla u|^2.
        \end{aligned}
      \end{equation*}
      As a consequence, by homogeneity, we may w.~l.~o.~g.\ suppose that
      \begin{align}\label{Assum1}
	\int_{\Boule} (  u^2 +  |\nabla u|^2 + \ubar^2 + |\nabla \ubar |^2 ) \leq 1.
      \end{align}
      Also, up to choosing $\delta$ sufficiently small, we may assume that:
      \begin{align}\label{Ellipticiteabar}
	\xi \cdot \abar \xi \geq \frac{1}{2}\lambda |\xi|^2 \quad \et \quad \xi \cdot \abar^{-1} \xi \geq \frac{1}{2}|\xi|^2.
      \end{align}
      (The easy argument of how this follows from \eqref{Ellipticite} and \eqref{CorrSub} via \eqref{Id02} is skipped since \eqref{Ellipticiteabar} automatically holds, even without the factor $\frac{1}{2}$, in stochastic homogenization \cite[(4)]{Gloria_Neukamm_Otto_2019}.)
      
      We now give an indirect argument for \eqref{Prop1Estim1}.
      We consider a sequence $a_n$ with associated homogenized matrices $\abar_n$, extended correctors $(\phi_n,\sigma_n)$ and $(\phi^\star_n,\sigma^\star_n)$, as well as $a_n$-harmonic functions $u_n$ and $\abar_n$-harmonic functions $\ubar_n$, satisfying
      \begin{align}
	\lim_{n \uparrow \infty} \fint_{\Boule} |(\phi_n,\phi^\star_n,\sigma_n,\sigma^\star_n)|^2 =0 \qquad \text{and} \qquad \liminf_{n \uparrow \infty} \fint_{\Boule} (u_n-\ubar_n)^2 >0
	\label{Contradict}
      \end{align}
      next to \eqref{Assum1}.
      We argue that such a sequence cannot exist, thus establishing \eqref{Prop1Estim1}.
      Up to a subsequence, the following convergences hold: $\abar_n \rightarrow \abar$ and, by the Rellich theorem,
      \begin{align*}
	(u_n,\ubar_n) \rightarrow (u,\ubar) \quad\et \quad\quad (\nabla u_n,\nabla \ubar_n) \rightharpoonup (\nabla u,\nabla \ubar)  \dans \LL^2(\Boule).
      \end{align*}
      We will obtain a contradiction to \eqref{Contradict} by arguing that $u=\ubar$.
      
      By assumption, $u=\ubar$ on $\partial \Boule$ (in the sense of $u-\ubar \in \HH^1_0(\Boule)$), so that it remains to argue that
      \begin{align}\label{ToEstab}
	-\nabla \cdot \abar \nabla u=0 \dans \HH^{-1}(\Boule).
      \end{align}
      Indeed, by weak convergence of $\nabla \ubar_n$ and convergence of $\abar_n$, we know that $\ubar$ is $\abar$-harmonic.
      Identity \eqref{ToEstab} will follows from
      \begin{align*}
	a_n \nabla u_n \rightharpoonup \abar \nabla u \dans \LL^2(\Boule),
      \end{align*}
      which is equivalent to
      \begin{align}\label{ToEstab2}
	a_n \nabla u_n- \abar_n \nabla u_n \rightharpoonup 0 \dans \LL^2(\Boule).
      \end{align}
      The latter is a consequence of \eqref{1a}.
      Indeed, for any $v \in \CC_0^\infty(\Boule)$ and $j\in \{1,\cdots,d\}$,
      \begin{align*}
	\int_{\Boule} v e_j \cdot ( a_n \nabla u_n - \abar_n \nabla u_n) \overset{\eqref{1a}}{=}\int_{\Boule} \nabla v \cdot (\phi^{\star}_{n,j} a_n + \sigma^{\star}_{n,j} ) \nabla u_n \rightarrow 0
      \end{align*}
      since $\nabla v \in \LL^\infty(\Boule)$, $\phi^{\star}_{n,j} a + \sigma^{\star}_{n,j} \rightarrow 0$ in $\LL^2(\Boule)$ and $\nabla u_n$ is bounded in $\LL^2(\Boule)$.
      This establishes \eqref{ToEstab2} and therefore \eqref{ToEstab}.
      
    \parag{Step 2: Argument for \eqref{Prop1Estim2}}
      The proof of \eqref{Prop1Estim2} relies on the previous step and on two ingredients from elliptic regularity.
      The first one is the well-known Caccioppoli estimate: If $w$ solves $\nabla \cdot ( a \nabla w + f )=0$ in $\frac{1}{2}\Boule$, then
      \begin{align}\label{Caccioppoli}
	\fint_{\frac{1}{4}\Boule} |\nabla w |^2 \lesssim \fint_{\frac{1}{2}\Boule} (w^2 + |f|^2 ).
      \end{align}
      (Recall that $\Boule=\Boule_1(0)$.)
      We reproduce in Appendix \ref{ProofCaccioppoli} a proof of this estimate in a generalized context.
      The second ingredient is an interior regularity result for elliptic equations with constant coefficients. More precisely, if $\ubar$ is $\abar$-harmonic in $\Boule$, then there holds
      \begin{align}\label{Regubar}
	\sup_{\frac{1}{2} \Boule} ( |\nabla^2 \ubar |^2 + |\nabla \ubar|^2 ) \lesssim \fint_{\Boule} |\nabla \ubar |^2.
      \end{align}
      (It is an easy consequence of the Sobolev embedding and an iterated application of the Caccioppoli estimate to derivatives of the constant-coefficient elliptic equation.)
      
      Equipped with \eqref{Caccioppoli} and \eqref{Regubar}, we proceed with the proof of \eqref{Prop1Estim2}.
      Appealing to \eqref{12}, we apply \eqref{Caccioppoli} to $w:=u-(1+\phi_i\partial_i)\ubar$ and $f:=(\phi_i a-\sigma_i)\nabla \partial_i \ubar$:
      \begin{align*}
	\fint_{\frac{1}{4}\Boule} | \nabla u - \partial_i \ubar (e_i+\nabla \phi_i)|^2 
	&~\lesssim\fint_{\frac{1}{4}\Boule} | \nabla w |^2  
	+\fint_{\frac{1}{4}\Boule} | \phi_i \nabla \partial_i \ubar|^2
	\\
	&\overset{\eqref{Caccioppoli}}{\lesssim}
	\fint_{\frac{1}{2}\Boule} w^2 + \fint_{\frac{1}{2}\Boule}
	|(\phi,\sigma)|^2 |\nabla^2  \ubar |^2
	\\
	&~\lesssim
	\fint_{\frac{1}{2}\Boule} ( u -\ubar )^2 + \fint_{\frac{1}{2}\Boule} |(\phi,\sigma)|^2 ( |\nabla^2 \ubar |^2 + |\nabla \ubar|^2 ).
      \end{align*}
      Making use of \eqref{Regubar}, this yields
      \begin{align*}
	&\fint_{\frac{1}{4}\Boule} | \nabla u + \partial_i \ubar (e_i+\nabla \phi_i)|^2 
	\lesssim \fint_{\frac{1}{2}\Boule} ( u -\ubar )^2 + \fint_{\Boule} |(\phi,\sigma)|^2 \fint_{\Boule}|\nabla u|^2.
      \end{align*}
      Appealing to \eqref{Prop1Estim1}, this finally shows \eqref{Prop1Estim2} for $\delta \ll_{d,\lambda,\epsilon} 1$.
  \end{proof}
  
\section{Our framework for stochastic homogenization}\label{SecCorrEstim}

  \subsection{General assumptions}\label{SecAssumpGen}
    In stochastic homogenization, one is given a probability measure on the space of coefficient fields $a$ that are $\lambda$-uniformly elliptic in the sense of \eqref{Ellipticite} (with the topology induced by the $H$-convergence, which makes it into a compact space). Using physics jargon, we call this probability measure an \textit{ensemble} and denote the expectation by $\langl\cdot\rangl$.
    This ensemble is assumed to be:
    \begin{itemize}
      \item{stationary, that is, for all shift vectors $z\in \R^d$, $a(\cdot+z)$ and $a(\cdot)$ have the same law under $\langl \cdot \rangl$;}
      \item{ergodic (see \cite[Chap.\ 7 pp.\ 222--225]{JKO}). Since this assumption will be strengthened below in Section \ref{SecAssumpGauss}, we do not state it precisely. Let us just mention that this property is required to use the Birkhoff ergodic theorem \cite[Th.\ 7.2, p.\ 225]{JKO}, which relates the expectation to the spatial average.}
    \end{itemize}
    
    Under these assumptions, by \cite[Lem.\ 1]{Gloria_Neukamm_Otto_2019}, there exists a random scalar field $\phi_i$ (the corrector), a random tensor field $\sigma_i$ (the flux corrector) for $i \in \{1,\cdots,d\}$, and a deterministic coefficient $\abar$ (the homogenized coefficient) that are related to $a$ in the sense of \eqref{Id02}, \eqref{Id01}, and \eqref{Id03}, respectively.
    Moreover, the following properties hold:
    \begin{itemize}
      \item{$\nabla \phi_i$ and $\nabla \sigma_i$ are stationary. We say that a random field $f(a,x)$ is stationary if it is shift-covariant in the sense of
      \begin{align*}
	&&f(a(\cdot + z),x)=f(a(\cdot),x+z) \qquad \pourtout x, z \in \R^d \text{ and for } \langl\cdot\rangle\text{-a.e. } a.
      \end{align*}
      It implies that $f(a,\cdot)$ and $f(a,\cdot+z)$ have the same law under $\langl \cdot \rangl$.
      }
      \item{$\nabla \phi_i$ and $\nabla \sigma_i$ have a finite second moment and vanishing expectation, that is,
      \begin{align}\label{NablaCorrBdd}
        &\langl |\nabla \phi_i(x)|^2 \rangl +\langl |\nabla \sigma_i(x)|^2 \rangl  \lesssim 1,
	\\
	\label{Vanish_expectation}
        &\langl \nabla \phi_i(x)\rangl =0 \et \langl \nabla \sigma_i(x) \rangl=0.
      \end{align}
      (By stationarity, this estimate and these identities are independent of $x \in \R^d$).
      }
      \item{The coefficient $\abar$ can be expressed as $\abar e_i = \langl q_i \rangl$ for $q_i$ defined by \eqref{Defqi}, and it satisfies \eqref{Ellipticite}.
      }
    \end{itemize}
    
    These results are only qualitative. In order to obtain a quantitative theory of stochastic homogenization, we additionally assume that the ensemble $\langl \cdot \rangl$ satisfies a spectral gap \cite{Gloria_Otto}.
    This condition amounts to a quantification of ergodicity.
    It reads: For any random variable $F$, which is a (measurable) functional of $a$, there holds
    \begin{align}\label{SG}
      \langl (F-\langl F \rangl )^2 \rangl \leq \Big\langl \int  \big| \frac{\partial F}{\partial a}\big|^2 \Big\rangl,
    \end{align}
    where the random tensor field $\frac{\partial F}{\partial a}$ (depending on $(a,x)$) is the functional (or vertical, or Malliavin) derivative of $F$ with respect to $a$ defined by
    \begin{equation}\label{Def:Malliavin}
      \lim_{\epsilon \rightarrow 0} \frac{F(a+\epsilon \delta a) - F(a)}{\epsilon} = \int  \frac{\partial F(a)}{\partial a_{ij}(x)}  (\delta a)_{ij}(x) \dd x.
    \end{equation}
    The spectral gap replaces the Poincaré inequality used in the periodic case.
    It leads to a sensitivity analysis through the study of the functional derivative.
    
    In the next section, we consider a more specific class of ensembles $\langl \cdot \rangl$ satisfying the above assumptions.

  \subsection{A class of Gaussian ensembles}\label{SecAssumpGauss}
    In this section, we introduce the class of ensembles we will be working with.
    Let $\langl\cdot\rangl$ denote an ensemble of stationary and centered Gaussian fields $g$
    (with values in some finite dimensional vector space, but we adopt scalar language and notation). It is thus characterized by its covariance function $c$, the (non-negative) Fourier transform $\FF c$ of which is assumed to satisfy
    \begin{align}\label{wg01}
    \FF c(k)\;\le (1+|k|)^{-d-2\alpha}\pourtout k\in\R^d,
    \end{align}
    for some fixed exponent $\alpha \in (0,1)$.
    Loosely speaking, \eqref{wg01} encodes that the covariance function is integrable at infinity (in a somewhat weakened way) and $2\alpha$-H\"older continuous at the origin.
    These assumptions contain the popular Mat\'ern kernel \cite[p.\ 31 \& (32), p.\ 49]{Stein_1999}, the Fourier transform of which reads $\FF c(k):= C (1+|k|^2)^{-\nu-d/2}$ for $C>0$ and $\nu>0$.
    
    In this framework, we set $a:=A(g)$, where $A$ is a Lipschitz map from $\R$ to the set of $d\times d$ tensors satisfying \eqref{Ellipticite} (see Figure \ref{Fig_Coeff} for an example).
    For conciseness, we henceforth suppress the dependence in the estimates on the $4$-tuple of fixed constants
    \begin{equation}\label{Def_Gamma}
      \gamma:=\big(d,\lambda,\alpha,\|A'\|_{\LL^\infty}\big).
    \end{equation}
    
    \begin{figure}[h]
      \includegraphics[width=\textwidth]{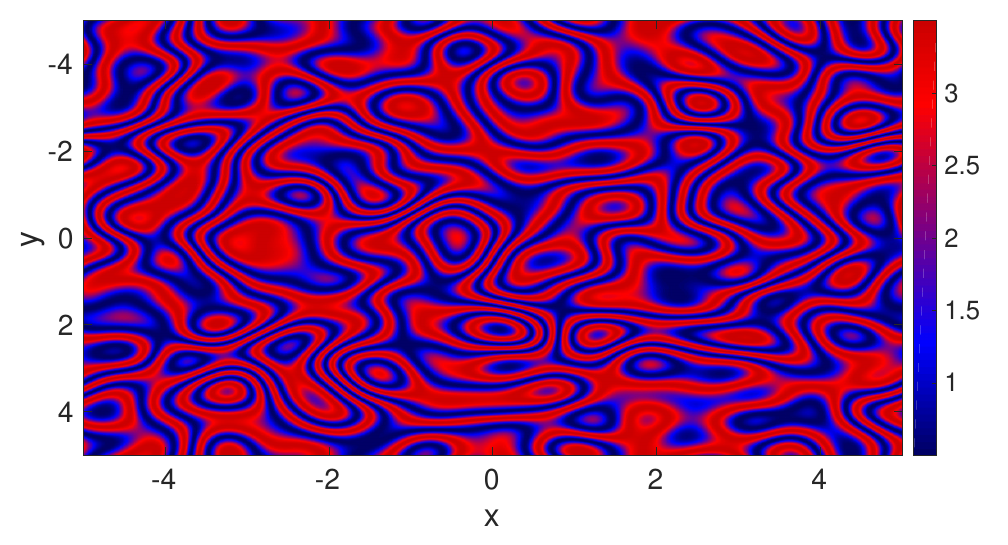}
      \caption{A realization of a random  (scalar) coefficient field $a$ generated by a Gaussian field.}
      \label{Fig_Coeff}
    \end{figure}

    Assumption (\ref{wg01}) on the ensemble $\langl\cdot\rangl$ has two beneficial consequences: On the one hand, we have a spectral gap estimate with a carr\'e-du-champs that behaves like the $\LL^1$-norm on scales up to order one and like the $\LL^2$-norm on scales larger than order $1$ (and thus is smaller than the $\LL^2$-norm, thereby strengthening the spectral gap property w.~r.~t.\ \eqref{SG}).
    On the other hand, the realizations of $g$ and thus $a$ are H\"older continuous.
    More precisely, we show the following result (the proof of \eqref{wg03} is relegated to Appendix \ref{SecProofLemloc}):
    
    \begin{lemma}\label{LemSG}
      Let $g$ and $a$ satisfy the above assumptions.
      Then, for any functional $F=F(a)$ and $r \in [1,\infty)$, there holds
      \begin{align}
	\label{Prop21a}
	\langl (F-\langl F \rangl)^{2r}\rangl^{\frac{1}{r}} 
	\lesssim_{\gamma, r} 
	\bigg\langl \Big( \int \Big( \fint_{\Boule_1(x)} \big| \frac{\partial F}{\partial a(y)}\big| \dd y \Big)^2 \dd x \Big)^{r} \bigg\rangl^{\frac{1}{r}},
      \end{align}
      and, for any $0<\alpha'<\alpha$ and $r \in [1,\infty)$, there holds
      \begin{align}\label{wg03}
	\big\langl\| a \|_{\CC^{\alpha'}(\Boule_1)}^r\big\rangl^\frac{1}{r}
	\lesssim_{\gamma, r, \alpha'} 1.
      \end{align}
    \end{lemma}
    
    \textit{En passant}, notice that estimate \eqref{wg03} allows to appeal to classical Schauder and CZ estimates for scales up to order $1$ since the dependence of these estimates on $\|a\|_{\CC^{\alpha'}(\R^d)}$ scales polynomially, see Lemma \ref{LemAppendHold} with $T=\infty$.
    In particular, if $u$ is $a$-harmonic in $\Boule_1$, then
    \begin{equation}\label{Lipschitz}
      \langl | \nabla u(0) |^{r'}\rangl^\frac{1}{r'}
      \lesssim_{\gamma, r',r} \langl \| \nabla u \|_{\LL^2(\Boule_1)}^{r} \rangl^{\frac{1}{r}} \pourtout 1\leq r'<r< \infty.
    \end{equation}
    
    \begin{proof}[Proof of \eqref{Prop21a}]
      W.~l.~o.~g.\ we assume that $\langl F \rangl=0$.
      We derive \eqref{Prop21a} by a two-step argument.
      First, we show that we have the following strengthening of \eqref{SG}:
      \begin{align}\label{wg02}
	\langl(F-\langl F\rangl)^2\rangl\lesssim
	\Big\langl \int 
	\Big(\fint_{\Boule_1(x)}\big|\frac{\partial F}{\partial g}\big|\Big)^2\dd x\Big\rangl,
      \end{align}
      from which we deduce \eqref{Prop21a} in a second step.
	
      \parag{Step 1: Argument for \eqref{wg02}}
	On the Fourier side, the Helffer-Sj\"ostrand representation  \cite[(2.2)]{Helffer_1998} implies\footnote{Here, the Fourier transform is normalized in the following way: $\FF u (k):= \int  \ee^{-\ii k x} u(x) \dd x$.\label{fref}
	}
	\begin{align*}
	  \langl F^2\rangl 
	  \leq~& \Big\langl \int \int  \frac{\partial F}{\partial g(x)} c(x-y) \frac{\partial F}{\partial g(y)} \dd y \dd x \Big\rangl
	  = \Big\langl\frac{1}{(2\pi)^{d}}\int  \FF c\big|\FF \frac{\partial F}{\partial g}\big|^2 \dd k\Big\rangl.
	\end{align*}
	\newcounter{compteur}
	\setcounter{compteur}{\thepage}
	Therefore, in view of (\ref{wg01}), it is enough to estimate the carr\'e-du-champs 
	for an arbitrary function $\zeta=\zeta(x)$ as follows:
	\begin{align*}
	\int (1+|k|)^{-d-2\alpha}|\FF \zeta|^2\dd k
	\lesssim 
	\int \Big(\fint_{\Boule_1(x)}|\zeta|\Big)^2\dd x,
	\end{align*}
	which by duality is equivalent to
	\begin{align}\label{wg04}
	\int \sup_{\Boule_1(x)} \zeta^2\dd x
	\lesssim
	\int (1+|k|)^{d+2\alpha}|\FF \zeta|^2\dd k.
	\end{align}
	Since the Sobolev inequality \eqref{wg04} is unrelated to homogenization, we relegate its proof to Appendix \ref{ProofE}.
	
      \parag{Step 2: Conclusion}
	We now improve \eqref{wg02} to \eqref{Prop21a}.
	By the chain rule,
	\begin{align*}
	  \frac{\partial}{\partial g(y)} = A'(g(y)) \frac{\partial }{\partial a(y)},
	\end{align*}
	and since $A'$ is bounded (recall  the definition \eqref{Def_Gamma} of $\gamma$), \eqref{wg02} yields
	\begin{align}\label{SGImproved}
	  \langl (F-\langl F\rangl)^2\rangl\lesssim
	  \Big\langl \int \Big(\fint_{\Boule_1(x)}\big|\frac{\partial F}{\partial a}\big|\Big)^2\dd x\Big\rangl.
	\end{align}
    
	We apply \eqref{SGImproved} to $F$ replaced by $|F|^r$ for $r \geq 1$:
	\begin{align}
	  \label{1a0}
	  \langl F^{2r} \rangl \lesssim \langl |F|^r \rangl^2 
	  + \Big\langl \int  \Big( \fint_{\Boule_1(x)} \big| \frac{\partial}{\partial a} |F|^r \big| \Big)^2 \Big\rangl.
	\end{align}
	By the chain rule
	\begin{align*}
	  \frac{\partial}{\partial a} |F|^r = r |F|^{r-2} F \frac{\partial F}{\partial a},
	\end{align*}
	and using the H\"older inequality in probability, we obtain
	\begin{equation}
	  \begin{aligned}
	    \Big\langl \int  \Big( \fint_{\Boule_1(x)} | \frac{\partial}{\partial a} |F|^r | \Big)^2 \dd x \Big\rangl
	    &= r^2
	    \Big\langl F^{2(r-1)}\int  \Big( \fint_{\Boule_1(x)}\big| \frac{\partial F}{\partial a}\big|\Big)^2 \dd x  \Big\rangl
	    \\
	    &\leq r^2
	    \langl F^{2r} \rangl^{\frac{r-1}{r}} \Big\langl \Big(\int  \Big( \fint_{\Boule_1(x)}\big| \frac{\partial F}{\partial a}\big|\Big)^2 \dd x \Big)^r  \Big\rangl^{\frac{1}{r}}.
	  \end{aligned}
	  \label{1a1}
	\end{equation}
	Moreover, the H\"older inequality followed by the spectral gap \eqref{SGImproved} yields
	\begin{align*}
	  \langl |F|^r \rangl^2
	  \leq~&
	  \langl F^{2r} \rangl^{\frac{r-2}{r-1}} \langl F^2 \rangl^{\frac{r}{r-1}} 
	  \lesssim
	  \langl F^{2r} \rangl^{\frac{r-2}{r-1}} \Big\langl \int   \Big( \fint_{\Boule_1(x)}\big| \frac{\partial F}{\partial a}\big|\Big)^2 \dd x \Big\rangl^{\frac{r}{r-1}}.
	\end{align*}
	Applying Jensen's inequality to the last factor, we deduce
	\begin{align}
	  \langl |F|^r \rangl^2 
	  \lesssim 
	  \langl F^{2r} \rangl^{\frac{r-2}{r-1}} 
	  \Big\langl \Big(\int  \Big( \fint_{\Boule_1(x)}\big| \frac{\partial F}{\partial a}\big|\Big)^2 \dd x \Big)^r  \Big\rangl^{\frac{1}{r-1}}.
	  \label{1a2}
	\end{align}
	Now, \eqref{Prop21a} follows from inserting \eqref{1a1} and \eqref{1a2} into \eqref{1a0} and appealing to Young's inequality.
    \end{proof}
    
  \section{Stochastic corrector estimates}\label{SecCorrEstim2}
  
    Estimates on the extended corrector are at the core of quantitative stochastic homogenization. 
    The following proposition collects all the estimates needed in this paper:
    \begin{proposition}\label{Propcorr}
      Under the assumptions of Section \ref{SecAssumpGauss}, for any $r \in [1, \infty)$, there holds:
      \begin{align}
        \label{i}
        \langl | \nabla (\phi,\sigma)|^{2r} \rangl^{\frac{1}{r}} \lesssim_{\gamma, r} 1.
      \end{align}
      Moreover, the spatial averages of the gradient of the extended correctors display cancellations:
      \begin{align}\label{ii}
        \Big\langl \Big| \int  g \cdot  \nabla  (\phi,\sigma) \Big|^{2r} \Big\rangl^{\frac{1}{r}} \lesssim_{\gamma, r} \int  |g|^2 \quad \text{ for all deterministic vector fields }g.
      \end{align}
      Last, the increments of the extended correctors are controlled as follows:
      \begin{align}
        \label{iii}
        \langl |(\phi,\sigma)(x) - (\phi,\sigma)(0)|^{2r} \rangl^{\frac{1}{r}} \lesssim_{\gamma, r} \mu_d^2(|x|) \pourtout x \in \R^d,
      \end{align}
      where the function $\mu_d$ is defined by
      \begin{align}\label{Defmu}
        \mu_d(r):=
	\begin{cases}
          \sqrt{r+1} &\si d=1,\\
          \ln^{\frac{1}{2}}(r+2) &\si d=2,\\
          1 & \si d>2.
        \end{cases}
      \end{align}
    \end{proposition}
    Estimate \eqref{ii} is identical to what one would obtain if $\nabla(\phi,\sigma)$ were replaced by white noise. It is thus not surprising that the large-scale behavior of increments of the extended corrector, \textit{cf.}~\eqref{iii}, is that of the Gaussian free field (on small scales, $(\phi,\sigma)$ is more regular).

\begin{proof}[Proof of Proposition \ref{Propcorr}]
      For abbreviation, we denote
      \begin{align}\label{Def_chiu}
        \chiu_i(x):=x_i + \phi_i(x),
      \end{align}
      and henceforth omit the index $i$.
      
      \parag{Strategy of the proof}
	The proof is divided into two main parts. Part 1 deals with the corrector $\phi$ and Part 2 deals with the flux corrector $\sigma$.	
	
	We describe the first part:
	In Step 1, we derive a representation formula for the Malliavin derivative $\frac{\partial}{\partial a} \int g \cdot \nabla \phi$.
	In Step 2, the $\LL^r$ version \eqref{Prop21a} of the spectral gap is used to estimate the $r$-th moments of the spatial averages of $\nabla \phi$ as
	\begin{align}
	  \label{Num:809}
	  \Big\langl \Big| \int g \cdot \nabla \phi \Big|^{2r} \Big\rangl^{\frac{1}{r}}
	  \lesssim
	  \Big\langl \Big( \fint_{\Boule_1} |\nabla \chiu|^2 \Big)^{r} \Big\rangl^{\frac{1}{r}} \int |g|^2 \pour 1 \ll r < \infty.
	\end{align}
	There, the representation formula for Malliavin derivative derived in Step 1 plays a central role, as well as an annealed CZ estimate of Proposition \ref{PropCZ}(i).
	Step 3 involves a PDE ingredient, namely the Caccioppoli estimate, to control the $r$-th moments of $\fint_{\Boule_1} |\nabla \chiu|^2$ by spatial averages of $\nabla \chiu$
	\begin{align}
	  \label{Claim01}
	  \Big\langl \Big(\fint_{\Boule_1} |\nabla \chiu|^2 \Big)^r \Big\rangl^{\frac{1}{r}} 
	  \lesssim
	  R^{d(1-\frac{1}{r})}\Big\langl  \Big| \fint_{\Boule_{R}} \nabla \chiu \Big|^{2r} \Big\rangl^{\frac{1}{r}} \pour R \gg 1.
	\end{align}
	In Step 4, the two above inequalities \eqref{Num:809} and \eqref{Claim01} are combined in order to buckle, yielding
	\begin{align}
	  \label{Step3}
	  \Big\langl \Big(\fint_{\Boule_1} |\nabla \chiu|^2 \Big)^r \Big\rangl^{\frac{1}{r}}  \lesssim 1.
	\end{align}
	By the local regularity estimate \eqref{Lipschitz}, we deduce \eqref{i} for $\nabla\phi$.
	Inserting this into \eqref{Num:809} yields \eqref{ii} for $\nabla\phi$.
	Finally, in Step 5, potential theory shows that \eqref{iii} is a consequence of \eqref{ii}.
	The second part is divided into three steps:
	We show the analogue of \eqref{Claim01} for the flux $q_i-\abar e_i$, which, using  the estimate \eqref{i} for $\nabla \phi$, entails that $\nabla\sigma$ satisfies \eqref{ii}, from which we deduce that $\sigma$ also satisfies \eqref{iii}.
      
      \parag{Part 1, Step 1: Representation formula for the Malliavin derivative}
	We establish the following representation formula for the Malliavin derivative:
	\begin{align}
	  \label{Prop21b}
	  \frac{\partial}{\partial a} \int g \cdot \nabla \phi= \nabla v \otimes \nabla \chiu,
	\end{align}
	where $\nabla v$ denotes the square-integrable solution of
	\begin{equation}\label{Prop21b_prime}
	  \nabla \cdot ( a^\star \nabla v +g )=0.
	\end{equation}
	
	Appealing to \eqref{Def:Malliavin}, we give ourselves a smooth and compactly supported infinitesimal variation $\delta a$ of $a$ and denote by $\nabla \delta \phi$ the corresponding infinitesimal variation of $\nabla \phi$, which is the square-integrable solution of\footnote{The argument for this identity is formal, since $\nabla \phi$ is only defined for $\langl\cdot\rangl$-a.e. $a$, see Section~\ref{SecAssumpGen}. However, this is fully justified in \cite[p.\ 21]{GNO_2019_Correlatedfields}, and we also refer to \eqref{Diff_phiT} where we work with the massive corrector which is defined for every $a$.}
	\begin{align}\label{DefDeltaphi}
	  -\nabla \cdot a \nabla \delta \phi=\nabla \cdot \delta a (e + \nabla \phi ).
	\end{align}
	Since $g$ is deterministic, this implies
	\begin{equation*}
	  \int \Big(\frac{\partial}{\partial a} \int  g \cdot \nabla \phi\Big) : \delta a
	  = \int  g \cdot \nabla \delta \phi.
	\end{equation*}
	Inserting the variational formulations of \eqref{Prop21b_prime} and \eqref{DefDeltaphi} into this, we get
	\begin{equation*}
	    \int  \Big(\frac{\partial}{\partial a} \int  g \cdot \nabla \phi\Big) : \delta a
	    =-\int \nabla v \cdot a \nabla \delta \phi 
	    =\int   \nabla v \cdot \delta a \nabla \chiu.
	\end{equation*}
	This yields \eqref{Prop21b}.

      \parag{Part 1, Step 2: Proof of \eqref{Num:809}}
	The proof relies on three ingredients. 
	The first one is the $\LL^r$ version \eqref{Prop21a} of spectral gap. 
	The second one is the representation formula for the Malliavin derivative \eqref{Prop21b}.
	The third one is the annealed CZ estimate given by Proposition \ref{PropCZ}(i) below.
	
	Equipped with these ingredients, we establish \eqref{Num:809}.
	From \eqref{Vanish_expectation}, \eqref{Prop21a}, and \eqref{Prop21b}, followed by the Cauchy-Schwarz inequality in $\LL^2(\Boule_1(x))$, we deduce that
	\begin{align*}
	  \Big\langl \Big| \int g \cdot \nabla \phi \Big|^{2r} \Big\rangl^{\frac{1}{r}}
	  \lesssim~& \Big\langl \Big( \int \Big( \fint_{\Boule_1(x)} | \nabla v \otimes \nabla \chiu |\Big)^2 \dd x\Big)^{r} \Big\rangl^{\frac{1}{r}}
	  \\
	  \leq~& \Big\langl \Big( \int  \Big( \fint_{\Boule_1(x)} |\nabla v|^2 \Big) \Big( \fint_{\Boule_1(x)}|\nabla \chiu|^2\Big) \dd x \Big)^r \Big\rangl^{\frac{1}{r}}.
	\end{align*}
	Thus by duality in probability, with $1/r + 1/r^\star=1$, 
	\begin{equation}
	  \label{Dual1}
	  \begin{aligned}
	    &\Big\langl \Big| \int g \cdot \nabla \phi \Big|^{2r} \Big\rangl^{\frac{1}{r}} 
	    \lesssim \sup_{\langl F^{2r^\star} \rangl=1} \Big\langl F^2 \int  \Big( \fint_{\Boule_1(x)} |\nabla v|^2 \Big) \Big( \fint_{\Boule_1(x)}|\nabla \chiu|^2\Big) \dd x \Big\rangl.
	  \end{aligned}
	\end{equation}
	We use the H\"older inequality and the stationarity of $\nabla \chiu$ (\textit{cf.} \eqref{NablaCorrBdd}) to the effect of
	\begin{equation}
	  \label{Dual2}
	  \begin{aligned}
	  &\int \Big\langl F^2  \Big( \fint_{\Boule_1(x)} |\nabla v|^2 \Big) \Big( \fint_{\Boule_1(x)}|\nabla \chiu|^2\Big)  \Big\rangl\dd x
	  \\
	  &\qquad \qquad \leq \int \Big\langl \Big( \fint_{\Boule_1(x)}|F\nabla v|^2\Big)^{r^\star} \Big\rangl^{\frac{1}{r^\star}} \dd x
	  \Big\langl \Big( \fint_{\Boule_1} |\nabla \chiu|^2 \Big)^r \Big\rangl^{\frac{1}{r}}.
	  \end{aligned}
	\end{equation}
	Since $u \mapsto \langl |u|^{r^\star} \rangl^{\frac{1}{r^\star}}$ is a norm, by Jensen's inequality,
	\begin{equation}\label{Num:713}
	  \begin{aligned}
	    \int \Big\langl \Big( \fint_{\Boule_1(x)}|F\nabla v|^2\Big)^{r^\star}  \Big\rangl^{\frac{1}{r^\star}} \dd x
	    &\leq
	    \int \fint_{\Boule_1(x)}\langl (|F\nabla v|^2)^{r^\star}  \rangl^{\frac{1}{r^\star}} \dd x
	    \\
	    &= \int \langl |F\nabla v|^{2r^\star}  \rangl^{\frac{1}{r^\star}},
	  \end{aligned}
	\end{equation}
	where we used in the last equality the identity
	\begin{align}\label{Num:556}
	\int h(x) \dd x
	=\int \fint_{\Boule_1(x)} h(y) \dd y \dd x.
	\end{align}
	Then, since \eqref{Prop21b_prime} may be written as $\nabla \cdot ( a^\star \nabla (Fv) + Fg)=0$, we learn from Proposition \ref{PropCZ}(i) below that
	\begin{equation}\label{Num:714}
	  \int \big\langl|F\nabla v|^{2r^\star} \big\rangl^{\frac{1}{r^\star}}
	  \lesssim
	  \int \langl| F g|^{2r^\star} \rangl^{\frac{1}{r^\star}}
	  =
	  \langl F^{2r^\star} \rangl^{\frac{1}{r^\star}} \int |g|^2.
	\end{equation}
	In the last lines, we have extensively used that $F$ is a random variable independent of $x$, whereas $g$ is deterministic.
	The application of Proposition \ref{PropCZ}(i) requires $|2r^\star-2| \ll 1$, which amounts to $r \gg 1$.
	Inserting \eqref{Num:713} and \eqref{Num:714} into \eqref{Dual2}, which we use in \eqref{Dual1}, entails \eqref{Num:809}.
	
      \parag{Part 1, Step 3: Proof of \eqref{Claim01}}
	By stationarity we have
	\begin{align*}
	  \Big\langl \Big(\fint_{\Boule_1} |\nabla \chiu|^2 \Big)^r \Big\rangl^{\frac{1}{r}} 
	  =\Big\langl \fint_{\Boule_R} \Big(\fint_{\Boule_1(x)} |\nabla \chiu|^2 \Big)^r \dd x\Big\rangl^{\frac{1}{r}}.
	\end{align*}
	We now appeal to the elementary string of inequalities:
	\begin{align*}
	  \fint_{\Boule_R} \Big(\fint_{\Boule_1(x)} |\nabla \chiu|^2 \Big)^r \dd x
	  &= \fint_{\Boule_R} \Big(\fint_{\Boule_1(x)} |\nabla \chiu|^2 \Big)^{r-1} \Big(\fint_{\Boule_1(x)} |\nabla \chiu|^2 \Big) \dd x
	  \\
	  &\leq (2R)^{d(r-1)}\Big(\fint_{\Boule_{2R}} |\nabla \chiu|^2 \Big)^{r-1} \fint_{\Boule_R}\fint_{\Boule_1(x)} |\nabla \chiu|^2 \dd x
	  \\
	  &\leq 2^d(2R)^{d(r-1)}\Big(\fint_{\Boule_{2R}} |\nabla \chiu|^2 \Big)^{r}.
	\end{align*}
	Therefore, we obtain
	\begin{align}\label{2b}
	  \Big\langl \Big(\fint_{\Boule_1} |\nabla \chiu|^2 \Big)^r \Big\rangl^{\frac{1}{r}} 
	  \lesssim~&R^{d(1-\frac{1}{r})} \Big\langl \Big(\fint_{\Boule_{2R}} |\nabla \chiu|^2 \Big)^{r}\Big\rangl^{\frac{1}{r}}.
	\end{align}
	Since $u$ is $a$-harmonic, \textit{c.f.} \eqref{Def_chiu} and \eqref{Id04}, the above right-hand term may be bounded thanks to the Caccioppoli estimate~\eqref{Caccioppoli3} (with $T=\infty$)
	\begin{align}\label{Cacc_1}
	  \fint_{\Boule_{2R}} |\nabla \chiu|^2
	  \lesssim
	  \frac{1}{R^2}\fint_{\Boule_{4R}} (\chiu-\Cmu)^2 \qquad \text{for all constants } \Cmu.
	\end{align}
	We denote by $u_{R'}$  the mollification of $u$ on scale $R' \leq R$ by convolution:
	\begin{equation}\label{Num:719}
	u_{R'}(x):=\fint_{\Boule_{R'}(x)}u.
	\end{equation}
	Choosing $\Cmu:=\fint_{\Boule_{4R}} u_{R'}(x)$, using the convolution estimate and the Poincaré-Wirtinger estimate, that is,
	\begin{equation*}
	  \fint_{\Boule_{4R}}(u-u_{R'})^2\lesssim {R'}^2\fint_{\Boule_{8R}}|\nabla u|^2
	  \et \quad
	  \fint_{\Boule_{4R}}(u_{R'}-c)^2
	  \lesssim
	  R^2\fint_{\Boule_{4R}}|\nabla u_{R'}|^2,
	\end{equation*}
	we obtain by the triangle inequality
	\begin{equation}\label{Num:739}
	  \begin{aligned}
	    \fint_{\Boule_{4R}} (\chiu-\Cmu)^2
	    \lesssim~& \fint_{\Boule_{4R}} \big( \chiu-\chiu_{R'}\big)^2
	    +\fint_{\Boule_{4R}} \big(\chiu_{R'}-\Cmu\big)^2
	    \\
	    \lesssim~& {R'}^2 \fint_{\Boule_{8R}} |\nabla \chiu|^2 + R^2 \fint_{\Boule_{4R}} |\nabla\chiu_{R'}|^2.
	  \end{aligned}
	\end{equation}
	Inserting this into \eqref{Cacc_1} entails
	\begin{align}
	  \fint_{\Boule_{2R}} |\nabla \chiu|^2
	  \lesssim~& \big(\frac{R'}{R}\big)^2 \fint_{\Boule_{8R}} |\nabla \chiu|^2 
	  + \fint_{\Boule_{4R}} |\nabla\chiu_{R'}|^2.
	  \label{2c}
	\end{align}
	Since we may cover $\Boule_{8R}$ by balls $\Boule_{2R}(x_n)$ for $n \in \{1,\cdots,N\}$ and $N\lesssim 1$,
	\begin{align*}
	  \fint_{\Boule_{8R}} |\nabla \chiu|^2
	  \lesssim \sum_{n=1}^N \fint_{\Boule_{2R}(x_n)} |\nabla \chiu|^2.
	\end{align*}
	Therefore, by the stationarity of $\nabla \chiu$ and the triangle inequality,
	\begin{align}\label{Num:771}
	  \Big\langl \Big(\fint_{\Boule_{8R}} |\nabla \chiu|^2\Big)^r\Big\rangl^{\frac{1}{r}} 
	  \lesssim 
	  \Big\langl \Big(\fint_{\Boule_{2R}} |\nabla \chiu|^2\Big)^r\Big\rangl^{\frac{1}{r}}.
	\end{align}
	As a consequence, it follows from \eqref{2c} that
	\begin{align*}
	  \Big\langl \Big(\fint_{\Boule_{2R}} |\nabla \chiu|^2\Big)^r \Big\rangl^{\frac{1}{r}}
	  \lesssim~&
	  \big(\frac{R'}{R}\big)^2 \Big\langl \Big(\fint_{\Boule_{2R}} |\nabla \chiu|^2\Big)^r\Big\rangl^{\frac{1}{r}}
	  +\Big\langl \Big(  \fint_{\Boule_{4R}} |\nabla\chiu_{R'}|^2\Big)^r \Big\rangl^{\frac{1}{r}}.
	\end{align*}
	Hence, if $R'=\theta R$ with $\theta \ll 1$, the first right-hand term may be absorbed in the left-hand side.
	Inserting this into \eqref{2b} yields:
	\begin{align*}
	  \Big\langl \Big(\fint_{\Boule_1} |\nabla \chiu|^2 \Big)^r \Big\rangl^{\frac{1}{r}} 
	  \lesssim
	  R^{d(1-\frac{1}{r})}\Big\langl \Big(  \fint_{\Boule_{4R}} |\nabla\chiu_{\theta R}|^2\Big)^r \Big\rangl^{\frac{1}{r}}.
	\end{align*}
	By Jensen's inequality and the stationarity of $\nabla \chiu$, we finally obtain
	\begin{equation*}
	  \Big\langl \Big(\fint_{\Boule_1} |\nabla \chiu|^2 \Big)^r \Big\rangl^{\frac{1}{r}} 
	  \lesssim
	  R^{d(1-\frac{1}{r})}\big\langl  |\nabla\chiu_{\theta R}|^{2r} \big\rangl^{\frac{1}{r}},
	\end{equation*}
	which may be rewritten as \eqref{Claim01} (by replacing $\theta R \rightsquigarrow R$).

      \parag{Part 1, Step 4: Buckling procedure, \textit{i.e.} proof of \eqref{Step3}}
	Note that, in light of \eqref{Step3}, \eqref{Num:809} reduces to
	\begin{align}
	  \label{1prime}
	  \Big\langl \Big| \int g \cdot \nabla \phi \Big|^{2r} \Big\rangl^{\frac{1}{r}}
	  \lesssim
	  \int |g|^2
	  \pourtout 1 \ll r < \infty.
	\end{align}
	Also, taking advantage of the local regularity estimate \eqref{Lipschitz}, estimate \eqref{Step3} may be upgraded to
	\begin{align}\label{EstimNablaLp}
	  \langl |\nabla \chiu |^{2r} \rangl^{\frac{1}{r}}  \lesssim 1.
	\end{align}
	Therefore, we obtain \eqref{i} and \eqref{ii} for $\nabla\phi$.
	
	Here comes the argument for \eqref{Step3}.
	W.~l.~o.~g.\ we may assume that $r\gg 1$ (Jensen's inequality will provide the full range $[1,\infty)$).
	With the abbreviation $g_R=\frac{1}{|\Boule_{R}|}\mathds{1}_{\Boule_{R}}$, we obtain
	\begin{align*}
	  \Big\langl  \Big| \fint_{\Boule_{R}} \nabla \chiu\Big|^{2r} \Big\rangl^{\frac{1}{r}}
	  \overset{\eqref{Def_chiu}}{\leq}&1 + \Big\langl  \Big| \int g_R \nabla \phi\Big|^{2r} \Big\rangl^{\frac{1}{r}}
	  \\
	  \overset{\eqref{Num:809}}{\lesssim}& 1 + \Big\langl\Big(\fint_{\Boule_1} |\nabla \chiu|^{2}\Big)^{r} \Big\rangl^{\frac{1}{r}} \int g_R^2 
	  \\
	  \lesssim~& 1 + R^{-d}\Big\langl\Big(\fint_{\Boule_1} |\nabla \chiu|^{2}\Big)^{r} \Big\rangl^{\frac{1}{r}}.
	\end{align*}
	Inserting this estimate into \eqref{Claim01} yields
	\begin{align*}
	  \Big\langl \Big(\fint_{\Boule_1} |\nabla \chiu|^2 \Big)^r \Big\rangl^{\frac{1}{r}} 
	  \lesssim
	  R^{d(1-\frac{1}{r})} +R^{-\frac{d}{r}}\Big\langl\Big(\fint_{\Boule_1} |\nabla \chiu|^{2}\Big)^{r} \Big\rangl^{\frac{1}{r}}.
	\end{align*}
	Choosing $R\gg 1$, the above r.~h.~s.\ may be absorbed into the above left-hand side.
	This establishes \eqref{Step3}.
	
      \parag{Part 1, Step 5: Proof of \eqref{iii} for $\phi$} This step is dedicated to proving
	\begin{align}
	  \label{iiia}
	  \langl |\phi(x) - \phi(0)|^{2r} \rangl^{\frac{1}{r}} \lesssim \mu_d^2(|x|).
	\end{align}
	By the triangle inequality and the stationarity of $x \mapsto \phi(x)-\fint_{\Boule_1(x)} \phi$, it is sufficient to show the following estimates:
	\begin{align}
	  \Big\langl \Big| \fint_{\Boule_1(x)} \phi - \fint_{\Boule_1(0)} \phi\Big|^{2r} \Big\rangl^{\frac{1}{r}} &\lesssim \mu_d^2(|x|),
	  \label{4a}
	  \\
	  \Big\langl \Big| \phi(0) - \fint_{\Boule_1(0)} \phi\Big|^{2r} \Big\rangl^{\frac{1}{r}} &\lesssim 1.
	  \label{4b}
	\end{align}
	
	Here comes the argument for \eqref{4a}.
	We have the representation formula
	\begin{align}\label{RepFormula}
	  \fint_{\Boule_1(x)} \phi - \fint_{\Boule_1(0)} \phi
	  &=\int  \nabla \ubar \cdot \nabla \phi,
	\end{align}
	where $\ubar$ denotes the decaying solution of
	\begin{align*}
	  -\Delta \ubar =\frac{1}{|\Boule_1|} ( \mathds{1}_{\Boule_1(x)} - \mathds{1}_{\Boule_1(0)} ).
	\end{align*}
	By classical potential theory\footnote{Indeed, for $d >3$, this follows from the energy identity $\int  |\nabla \ubar|^2 = \fint_{\Boule_1(x)} \ubar - \fint_{\Boule_1(0)} \ubar$ via the Sobolev inequality $(\int  \ubar^{\frac{2d}{d-2}} )^{\frac{d-2}{2d}} \lesssim ( \int  |\nabla \ubar|^2 )^{\frac{1}{2}}$. In dimension $d=2$, we may appeal to the explicit representation $\nabla \ubar(y)=\fint_{\Boule_1}  ( \nabla G(y-x-z)-\nabla G(y-z) ) \dd z$ via the fundamental solution $\nabla G(y):=\frac{1}{2\pi} \frac{y}{|y|^2}$.},
	\begin{align}\label{IneqPot}
	  \int  |\nabla \ubar|^2 \lesssim \mu^2_d(|x|).
	\end{align}
	Since $\ubar$ is deterministic, as a consequence of \eqref{1prime} and \eqref{RepFormula}, we obtain
	\begin{align*}
	  \Big\langl \Big| \fint_{\Boule_1(x)} \phi - \fint_{\Boule_1(0)} \phi\Big|^{2r} \Big\rangl^{\frac{1}{r}} 
	  &\lesssim \int  |\nabla \ubar|^2,
	\end{align*}
	which entails \eqref{4a} by \eqref{IneqPot}.
	
	We now argue for \eqref{4b}.
	By a Sobolev embedding, provided $2r >d$, there holds:
	\begin{align*}
	  \Big| \phi(0) - \fint_{\Boule_1(0)} \phi\Big|^{2r} 
	  \lesssim \fint_{\Boule_1(0)} |\nabla \phi|^{2r}.
	\end{align*}
	Hence, taking the expectation, and recalling that $\nabla \phi$ is stationary and satisfies \eqref{i}, we deduce that
	\begin{align*}
	  \Big\langl \Big| \phi(0) - \fint_{\Boule_1(0)} \phi\Big|^{2r}\Big\rangl 
	  \lesssim \Big\langl \fint_{\Boule_1(0)} |\nabla \phi|^{2r} \Big\rangl \lesssim 1.
	\end{align*}
	This shows \eqref{4b} and concludes the proof of \eqref{iiia}.
      
      \parag{Part 2, Step 1: Proof of \eqref{ii} for $\nabla\sigma$}
	We first show the analogue of \eqref{1prime} for the flux $q_i$ (\textit{cf.} \eqref{Defqi}); that is, if $g$ is a deterministic vector field, then
	\begin{align}
	  \label{1flux}
	  \Big\langl \Big| \int  g \cdot ( q_i-\abar e_i) \Big|^{2r} \Big\rangl^{\frac{1}{r}}
	  \lesssim
	  \int |g|^2,
	\end{align}
	from which we shall deduce \eqref{ii} for $\nabla\sigma$.
	The only change is in the representation formula:
	\begin{align}\label{1fluxprime}
	  \frac{\partial}{\partial a} \int  g \cdot ( q_i-\abar e_i) =( \nabla v + g ) \otimes \nabla \chiu_i,
	\end{align}
	where the square-integrable vector field $\nabla v$ satisfies $\nabla \cdot a^\star ( \nabla v + g)=0$.
	Indeed, we have
	\begin{align*}
	  \int   \Big(\frac{\partial}{\partial a}\int  g \cdot ( q_i-\abar e_i) \Big): \delta a
	  \overset{\eqref{Defqi}}=~&
	  \int  g \cdot ( \delta a ( e_i + \nabla \phi_i)+ a \nabla \delta \phi_i),
	\end{align*}
	where $\nabla \delta\phi_i$ satisfies \eqref{DefDeltaphi}.
	The rightmost term in the integrand is dealt with as in \eqref{Prop21b}, so that \eqref{1fluxprime} is proved.
	
	Next, we show that $\nabla\sigma$ satisfies \eqref{ii}.
	Denoting by $\nabla \vbar$ the square-integrable solution of $\Delta \vbar+\nabla \cdot g=0$, we obtain
	\begin{align*}
	  \int  g \cdot \nabla \sigma_{ijk}
	  =~& -\int  ( \vbar\partial_j q_{ik} - \vbar \partial_k q_{ij})
	  \overset{\eqref{Id02}}{=}\int   
	  \big( \partial_j \vbar(q_{ik}- \abar_{ki}) - \partial_k\vbar (q_{ij}-\abar_{ji})\big),
	\end{align*}
	where \eqref{1flux} justifies this formal integration by parts.
	Therefore, we deduce from \eqref{1flux} that
	\begin{align*}
	  \Big\langl \Big|\int  g \cdot \nabla \sigma_{ijk}\Big|^{2r} \Big\rangl^{\frac{1}{r}}
	  \lesssim
	  \int  |\nabla \vbar |^2
	  \leq
	  \int  |g |^2.
	\end{align*}

      \parag{Part 2, Step 2: Proof of \eqref{i} for $\nabla \sigma$} 
      By the CZ estimates applied to the constant-coefficient equation \eqref{Id03} (by a post-processing of \cite[Th.\ 9.11, p.\ 235]{GT}), there holds
      \begin{align*}
        \fint_{\Boule_R} |\nabla \sigma |^{2r}
        \lesssim
        \Big(\fint_{\Boule_{2R}} |\nabla \sigma |^{2} \Big)^{r}
        +
        \fint_{\Boule_{2R}} |q |^{2r}.
      \end{align*}
      By ergodicity and stationarity, when $R \uparrow \infty$, each of these spatial averages converges almost surely to the associated expectation (this is a consequence of the Birkhoff theorem). Therefore, we get
      \begin{align*}
        \langl |\nabla \sigma|^{2r} \rangl^{\frac{1}{r}} 
        \lesssim 
        \langl|\nabla \sigma |^{2} \rangl
        +\langl |q |^{2r}\rangl^{\frac{1}{r}},
      \end{align*}
      from which we deduce that $\nabla\sigma$ satisfies \eqref{i}, by \eqref{NablaCorrBdd} and \eqref{EstimNablaLp}.
      
      \parag{Part 2, Step 3: Proof of \eqref{iii} for $\sigma$}
	By arguments identical to Part 1, Step 5, we deduce that $\sigma$ satisfies \eqref{iii}.
    \end{proof}

  \section{Oscillations: Estimate of homogenization error in strong norms}\label{Sec:Osc}
    Equipped with the stochastic corrector estimates, we now may tackle the homogenization error. 
    In this section, we address the error in the two-scale expansion on the level of the gradient in strong $\LL^p$-norms.
    In fact, in view of \eqref{12}, the error estimate is a corollary of the corrector estimates in Proposition \ref{Propcorr} and the annealed CZ estimate of Proposition \ref{PropCZ}(ii). 
    
    \begin{corollary}\label{Cor1}
      Suppose that $f$  a deterministic, smooth, and compactly supported vector field. 
      Under the assumptions of Section \ref{SecAssumpGauss}, let the square-integrable vector fields $\nabla u$, $\nabla \ubar$ be related to $f$ by
      \begin{align*}
	\nabla \cdot ( a \nabla u + f ) =0 =\nabla \cdot( \abar \nabla \ubar + f).
      \end{align*}
      Then, for any $p \in (1,\infty)$ and $r\in[1,\infty)$, there holds:
      \begin{align}
	\begin{aligned}
	  \Big\langl \Big( \int  | \nabla ( u - (1+\phi_i\partial_i) \ubar ) |^p \Big)^{r} \Big\rangl^{\frac{1}{pr}}
	  \lesssim_{\gamma,p, r} \Big( \int  \big|\mu_d(|\cdot| ) \nabla f\big|^p \Big)^{\frac{1}{p}}.
	\end{aligned}
	\label{Estim02}
      \end{align}
    \end{corollary}

    It is customary to rescale the space variable according to macroscopic coordinates $\hat{x}=\epsilon x$, where $\epsilon \ll 1$ corresponds to the ratio of the microscale coming from the coefficient field (implicitly contained in our assumption \eqref{wg01}) and the macroscale coming from the r.~h.~s.\ $f$.
    By scaling, we have 
    \begin{equation*}
      \dd \hat{x}=\epsilon^d \dd x,\quad \hat{\nabla} = \epsilon^{-1}\nabla, \quad \hat{u}(\hat{x})=\epsilon u(x) \quad \et  \quad \hat{\phi}(\hat{x})=\epsilon\phi(x).
    \end{equation*}
    In this new notation, Corollary \ref{Cor1} reads:
    \begin{corollary}\label{Cor2}
      Suppose that $\hat{f}$  a deterministic, smooth, and compactly supported vector field. 
      Under the assumptions of Section \ref{SecAssumpGauss}, let the square-integrable vector fields $\nabla \hat{u}$, $\nabla \ubar$ be related to $\hat{f}$ by
      \begin{align*}
        \hat{\nabla} \cdot \big( a\big(\frac{\cdot}{\epsilon}\big) \hat{\nabla} \hat{u} + \hat{f} \big) 
        =0 
        =\hat{\nabla} \cdot( \abar \hat{\nabla} \ubar + \hat{f})
      \end{align*}
      for some $\epsilon\leq 1$.
      Then, for any $p \in (1,\infty)$ and $r\in [1,\infty)$, there holds:
      \begin{align}
	\begin{aligned}
	\Big\langl \Big( \int  \big| \hat{\nabla} \big(\hat{u} - \big(1+\hat{\phi}_i \hat{\partial}_i\big) \ubar \big) \big|^p  \dd \hat{x} \Big)^{r} \Big\rangl^{\frac{1}{pr}}
        &\lesssim_{\gamma,p, r, \hat{f}} \epsilon \mu_d(1/\epsilon)
        \\
        &=
	\begin{cases}
          \epsilon^{\frac{1}{2}}
          & \si \; d=1,\\
	  \epsilon \ln^{\frac{1}{2}}(1/\epsilon +2)
	  & \si \; d=2,\\
	  \epsilon
	  & \si \; d>2.
        \end{cases}
	\end{aligned}
	\label{Estim03}
      \end{align}
    \end{corollary}
    Hence, for dimensions $d>2$, stochastic homogenization has linear order of convergence in $\epsilon$, as in periodic homogenization.
    In fact, higher-order correctors provide the dimension-dependent order of convergence $d/2$  (with a logarithmic correction in even dimensions), \textit{c.f.} \cite[Prop.\ 2.7]{DuerinckxOtto_2019}.

    \begin{proof}[Proof of Corollary \ref{Cor1}]
    We first note that is it sufficient to establish
    \begin{align}
      \label{Estim01}
      \begin{aligned}
	\Big\langl \Big( \int  | \nabla ( u - (1+\phi_i\partial_i) \ubar ) |^p \Big)^{r} \Big\rangl^{\frac{1}{r}}
	\lesssim_{\gamma,p, r} \int   \big|\mu_d(|\cdot|) \nabla^2 \ubar\big|^p.
      \end{aligned}
    \end{align}
    Indeed, observe that $\mu_d^p(|\cdot|)$ is a Muckenhoupt weight of class $A_p$ \cite[Def.\ 7.1.3, p.\ 503]{Grafakos_book} if $d\geq 2$.
    Thus, by weighted CZ estimates \cite[Th.\ 7.4.6, p.\ 540]{Grafakos_book} for the operator $\nabla (-\nabla \cdot \abar \nabla)^{-1} \nabla \cdot$, there holds
    \begin{equation*}
      \int  \big|\mu_d (|\cdot|) \nabla^2 \ubar\big|^p \lesssim \int \big|\mu_d (|\cdot|) \nabla f\big|^p.
    \end{equation*}
    (This estimate is trivially true if $d=1$.)
    As a consequence, \eqref{Estim01} yields \eqref{Estim02}.
  
    Here comes the argument for \eqref{Estim01}.
    We multiply \eqref{12} by an arbitrary random variable $F=F(a) \geq 0$, to the effect of
    \begin{align*}
      -\nabla\cdot a \nabla(F ( u - (1+\phi_i\partial_i) \ubar )) = \nabla \cdot F (\phi_i a - \sigma_i ) \nabla \partial_i \ubar.
    \end{align*}
    By Proposition \ref{PropCZ} (ii), this yields
    \begin{align*}
      \int  \langl |F \nabla ( u - (1+\phi_i\partial_i) \ubar)|^p \rangl
      \lesssim
      \int  \langl F^{s}|(\phi_i a - \sigma_i ) \nabla \partial_i \ubar|^{s} \rangl^{p/s}
    \end{align*}
    for some $s >p$ to be fixed later.
    The integrand of the above right-hand side may be estimated by means of the H\"older inequality with $t, t^\star>1$ satisfying $1/t+1/t^\star=1$ (to be fixed later) as
    \begin{align*}
      \langl F^{s}|(\phi_i a - \sigma_i ) \nabla \partial_i \ubar|^{s} \rangl^{p/s}
      =~&\langl F^{s}|\phi_i a - \sigma_i|^{{s}} \rangl^{p/{s}} |\nabla \partial_i \ubar|^p
      \\
      \leq~& \langl F^{{s}t^\star}\rangl^{\frac{p}{{s}t^\star}}
      \langl|\phi_i a - \sigma_i|^{{s}t} \rangl^{\frac{p}{{s}t}} |\nabla \partial_i \ubar|^p.
    \end{align*}
    W.~l.~o.~g.\ we may assume the anchoring $(\phi,\sigma)(0)=0$ so that by Proposition \ref{Propcorr} we have
    \begin{align*}
      \langl|(\phi_i a - \sigma_i )(x)|^{{s}t} \rangl^{\frac{p}{{s}t}} \lesssim \mu^p_d(|x|).
    \end{align*}
    Therefore, we deduce from the above estimates that
    \begin{align}
      &\Big\langl F^p\int  |\nabla ( u - (1+\phi_i\partial_i) \ubar)|^p \Big\rangl
      \lesssim
      \langl (F^p)^{\frac{{s}t^\star}{p}}\rangl^{\frac{p}{{s}t^\star}}
      \int \big| \mu_d(|\cdot|) \nabla^2 \ubar \big|^p.
      \label{Abo1}
    \end{align}
    We now pick an exponent $t^\star < r^\star$ and specify
    \begin{align*}
      s:=\frac{pr^\star}{t^\star} >p.
    \end{align*}
    Hence, \eqref{Abo1} reads
    \begin{align*}
      \Big\langl F^p\int  |\nabla ( u - (1+\phi_i\partial_i) \ubar)|^p \Big\rangl
      \lesssim
      \langl F^{pr^\star}\rangl^{\frac{1}{r^\star}}
      \int \big| \mu_d(|\cdot|) \nabla^2 \ubar\big|^p,
    \end{align*}
    so that a duality argument yields \eqref{Estim01}.
  \end{proof}

    \begin{proof}[Proof of Corollary \ref{Cor2}]
      Rescaling \eqref{Estim02}, we obtain:
      \begin{align*}
        \Big\langl \Big( \int  \big| \hat{\nabla} \big(\hat{u} - \big(1+\hat{\phi}_i \hat{\partial}_i\big) \ubar \big) \big|^p  \dd \hat{x} \Big)^{r} \Big\rangl^{\frac{1}{pr}}
        \lesssim
        \Big( \int   \big|\epsilon \mu_d\big(\frac{|\hat{x}|}{\epsilon}\big) \hat{\nabla} \hat{f}\big|^p  \dd \hat{x}\Big)^{\frac{1}{p}}.
      \end{align*}
      Recalling the expression \eqref{Defmu} of $\mu_d$ directly yields \eqref{Estim03} for $d\neq 2$. If $d=2$, we additionally make use of the inequality $\ln(|x|/\epsilon+2) \lesssim \ln(|x|+2) \ln( 1/\epsilon+2)$ to get \eqref{Estim03}.
    \end{proof}

\section{Fluctuations: Estimate of homogenization error in weak norms}\label{SecFluctu}

  This section is devoted to studying the fluctuations of an observable
  \begin{align}\label{MacroF}
    G:=\int  g \cdot \nabla u,
  \end{align}
  where the square-integrable vector fields $\nabla u$ and $f$ are related through \eqref{Complexe} for deterministic vector fields $g$ and $f$.
  More precisely, we are interested in macroscopic observables, that is
  \begin{align}
    \label{Macrohf}
    g(x):=\epsilon^{d} \hat{g}(\epsilon x) \et f(x):=\hat{f}(\epsilon x),
  \end{align}
  where $\hat{g}$ and $\hat{f}$ are fixed vector fields of scale $1$ and $\epsilon \ll 1$.
  In \cite{Marahrens_Otto}, it has been observed that the variance of $G$ has a central-limit theorem scaling in $\epsilon^{-1}$. Moreover, in \cite{GuMourrat_2016}, it is shown that the leading order (in $\epsilon \ll 1$) of this variance may be explicitly characterized in terms of a quartic tensor $\QU$ introduced in \cite{Mourrat_Otto_2016}.
  
  Surprisingly enough, the naive (but natural) idea of replacing $\nabla u$ by its two-scale expansion in \eqref{MacroF} gives the wrong leading order, as was discovered by \cite{GuMourrat_2016}. 
  However, the two-scale expansion may be used in a more subtle way, as was discovered in \cite{DuerinckxGO_2016}: We define $\nabla\vbar $ as the square-integrable vector field related to $g$ through
  \begin{align}\label{Defh}
    \nabla \cdot ( \abar^\star \nabla \vbar + g)=0.
  \end{align}
  Then, $G$ may be written in terms of the homogenization commutator (see \eqref{1a}) as
  \begin{align}\label{Gread}
    G=~&-\int  \nabla \vbar \cdot \abar \nabla  u=\int  \nabla \vbar \cdot (a-\abar) \nabla  u + F,
  \end{align}
  where $F=-\int  \nabla \vbar \cdot a \nabla u=\int  \nabla \vbar \cdot f$ is deterministic.
  It turns out that in the homogenization commutator $(a-\abar)\nabla u$, it is legitimate to approximate $\nabla u$ by its two-scale expansion. 
  This leads to the \textit{standard homogenization commutator}, which is the stationary random field $\Xi$ defined by
  \begin{align}\label{Def_Xi}
    \Xi e_i := ( a - \abar ) ( e_i+\nabla \phi_i)
  \end{align}
  (see \cite{DuerinckxGO_2016}, and the discussion of the literature there).
  This observation is made rigorous in the result below.  
  This result is identical to \cite[Prop.\ 3.2]{DuerinckxOtto_2019}, which itself is a continuum version of \cite[Th.\ 2]{DuerinckxGO_2016}.
  We reproduce the streamlined proof in order to highlight the similarity to the one of Proposition \ref{Propcorr} and the use of the annealed CZ estimates, namely Proposition \ref{PropCZ}(ii).
  
  \begin{proposition}\label{Propfluct}
    Let the ensemble $\langl \cdot \rangl$ be defined as in Section \ref{SecAssumpGauss}.
    Assume that $h$ and $f$ are square-integrable deterministic vector fields.
    Let the square-integrable vector fields $\nabla u$ and $\nabla \ubar$ be related to $f$ through \eqref{Complexe} and \eqref{Simple}.
    Then, the random variable
    \begin{align}\label{DefH}
      H:=\int  h \cdot ( a - \abar ) \big(\nabla u -\partial_i \ubar (e_i+\nabla \phi_i )\big)
    \end{align}
    satisfies, for $\mu_d$ defined by \eqref{Defmu},
    \begin{align}\label{Num:004}
      \langl ( H - \langl H \rangl)^{2 r} \rangl^{\frac{1}{r}} 
      \lesssim_{\gamma, r} \Big( \int  | h |^4 \int  \big|\mu_d(|\cdot|)\nabla f\big|^4 
      + \int  |f|^4\int \big|\mu_d(|\cdot|) \nabla h \big|^4 \Big)^{\frac{1}{2}}.
    \end{align}
  \end{proposition}
  
  Let us briefly compare Proposition \ref{Propfluct} and  Corollary \ref{Cor1}.
  Both give an estimate on the error in the two-scale expansion on the level of the gradient $\nabla(u-(1+\phi_i\partial_i)\ubar) \simeq \nabla u - \partial_i \ubar (e_i+\nabla \phi_i)$.
  While Corollary \ref{Cor1} gives an estimate in a strong norm (\textit{oscillations}), Proposition \ref{Propfluct}, modulo the homogenization commutator, gives an estimate in a weak norm (\textit{fluctuations}), \textit{i.e.} when tested against a macroscopic $h$.
  
  Proposition \ref{Propfluct} allows to approximate the random variable $G$ defined by \eqref{MacroF} by its \textit{two-scale expansion} $\tilde G$ defined by
  \begin{equation}\label{Def_tildeG}
    \tilde{G}:=\int  \nabla \vbar \cdot \Xi \nabla \ubar.
  \end{equation}
  This is valuable because $\tilde G$ only relies on solving the constant-coefficient problem \eqref{Simple}, the dual constant-coefficient problem \eqref{Defh}, and involves randomness only in form of the standard homogenization commutator \eqref{Def_Xi}, which is independent on the r.~h.~s.~$f$ and the averaging function $g$. In fact, on large scales, $\Xi-\langle \Xi\rangle$ behaves like (tensorial) white noise and  is characterized by the above-mentioned quartic tensor $\QU$.
  As for $\abar$, $\QU$ may be approximated by the representative volume element method \cite[Th.\ 2]{DuerinckxGO_2016}.
  
  \begin{corollary}\label{CoroFluctu}
    Let the ensemble $\langl \cdot \rangl$ be defined as in Section \ref{SecAssumpGauss}.
    Assume that $\hat{f}$ and $\hat{g}$ are smooth and compactly supported, and let $f$ and $g$ be defined by \eqref{Macrohf}.
    Let $\nabla u, \nabla \ubar$, and $\nabla \vbar $ be square-integrable functions satisfying \eqref{Complexe}, \eqref{Simple}, and \eqref{Defh}.
    Let $G$ and $\tilde{G}$ be defined by \eqref{MacroF} and \eqref{Def_tildeG}.
    Then, for $\mu_d$ defined by \eqref{Defmu} and $r \in [1,\infty)$, there holds
    \begin{equation}
      \label{Ineq4}
      \begin{aligned}
        &\epsilon^{-d/2}\langl ( G -\tilde{G} - \langl G-\tilde{G} \rangl  )^{2r} \rangl^{\frac{1}{2r}}
	&\lesssim_{\gamma, r,\hat{f},\hat{g}} \epsilon \mu_d(1/\epsilon).
      \end{aligned}
    \end{equation}
  \end{corollary}
  
  \begin{proof}[Proof of Corollary \ref{CoroFluctu}]
    By \eqref{Gread} and as a consequence of Proposition \ref{Propfluct}, there holds
    \begin{align}
    \begin{aligned}
      &\langl ( G -\tilde{G} - \langl G-\tilde{G} \rangl  )^{2r} \rangl^{\frac{1}{r}}
      \\
      &\qquad \lesssim \Big( \int  | h |^4 \int  \big|\mu_d(|\cdot|)\nabla f\big|^4 + 
      \int  |f|^4\int \big|\mu_d(|\cdot|) \nabla h \big|^4 \Big)^{\frac{1}{2}}
    \end{aligned}
    \label{Ineq3}      
    \end{align}
    for $h:=\nabla \vbar $.
    We recall that, in dimensions $d \geq 2$, $\mu_d^4(|\cdot|)$ are Muckenhoupt weights of class $A_{4}$.
    Therefore, by weighted CZ estimates \cite[Th.\ 7.4.6, p.\ 540]{Grafakos_book}, we also have
    \begin{equation}\label{Num:718}
      \int  | h |^4 \lesssim \int  |g|^4 \et  \int  \big|\mu_d(|\cdot|)\nabla h \big|^4 \lesssim \int  \big|\mu_d(|\cdot|) \nabla g\big|^4.
    \end{equation}
    (A similar estimate also holds in dimension $d=1$.)
    Hence, using the scaling induced by \eqref{Macrohf}, estimate \eqref{Ineq3} yields
    \begin{align*}
      \langl (G -\tilde{G} - \langl G-\tilde{G}\rangl  )^{2r} \rangl^{\frac{1}{r}}
      &\lesssim \Big( \int  | \epsilon^{d} \hat{g}(\epsilon x) |^4 \dd x \int  \big|\mu_d(|x|)\epsilon\hat{\nabla} \hat{f}(\epsilon x)\big|^4 \dd x\Big)^{\frac{1}{2}} 
      \\
      &\quad + \Big( \int  |\hat{f}(\epsilon x)|^4 \dd x \int  \big|\mu_d(|x|) \epsilon^{d+1} \hat{\nabla} \hat{g}(\epsilon x) \big|^4 \dd x\Big)^{\frac{1}{2}},
    \end{align*}
    which implies \eqref{Ineq4} by the change of variables $\hat{x}=\epsilon x$.
  \end{proof}
  
  \begin{proof}[Proof of Proposition \ref{Propfluct}]
    The proof, which has the same ingredients as the proof of Proposition \ref{Propcorr}, is divided into three steps.
    In Step 1, which is reminiscent of Part 1, Step 1 of the proof of Proposition \ref{Propcorr}, we derive the following representation of the Malliavin derivative of $H$:
    \begin{equation}
      \label{DeriveH}
      \begin{aligned}
        \frac{\partial H}{\partial a}
	=~&h_j ( e_j + \nabla \phi^\star_j) \otimes ( \nabla w + \phi_i \nabla \partial_i \ubar)
	\\
	&
	+ ( \nabla w^\star + \phi_j^\star \nabla h_j) \otimes \nabla u
	\\
	&- ( \nabla w_i^\star + \phi^\star_j \nabla ( h_j \partial_i \ubar)) \otimes ( e_i + \nabla \phi_i ),
      \end{aligned}
    \end{equation}
    where $\nabla w$, $\nabla w^\star$, and $\nabla w^\star_j$ are the square-integrable solutions of
    \begin{align}
      \label{w1}
      &\nabla \cdot ( a \nabla w + ( \phi_i a - \sigma_i) \nabla \partial_i \ubar )=0,\\
      \label{w2}
      &\nabla \cdot ( a^\star \nabla w^\star + ( \phi^\star_j a^\star - \sigma_j^\star ) \nabla h_j )=0,\\
      \label{w3}
      &\nabla \cdot ( a^\star \nabla w^\star_j + ( \phi^\star_i a^\star - \sigma^\star_i ) \nabla (h_i \partial_j \ubar ) )=0.
    \end{align}
    (Note that \eqref{w2} and \eqref{w3} define $d+1$ functions $w^\star, w_1^\star, \cdots, w_d^\star$.)
    Step 2 relies on the spectral gap \eqref{Prop21a} and makes use of the annealed estimate in Proposition \ref{PropCZ}(ii) below.
    Finally, in Step 3, we appeal to the correctors estimates to establish \eqref{Num:004}.

    \parag{Step 1:}
      Here comes the argument for \eqref{DeriveH}.
      Given a smooth and compactly supported infinitesimal variation $\delta a$ of $a$, since $h$ is deterministic, we have for the generated infinitesimal variation of $H$ defined by \eqref{DefH}
      \begin{equation}
        \begin{aligned}
	  \delta H
	  =~& \int h \cdot \delta a ( \nabla u - \partial_i \ubar ( e_i + \nabla \phi_i))
	  + \int  h \cdot ( a - \abar) ( \nabla \delta  u -\partial_i \ubar \nabla \delta  \phi_i).
        \end{aligned}
	\label{Diff0}
      \end{equation}
      Since $f$ is deterministic, differentiating the equation \eqref{Complexe} satisfied by $u$ we obtain
      \begin{align}
	\label{Diff1}
	&\nabla \cdot ( a \nabla \delta  u + \delta a \nabla u  )=0.
      \end{align}
      Moreover, we recall that $\nabla \delta \phi_i$ satisfies \eqref{DefDeltaphi}.
      Next, we note, and will prove below, that $w^\star$ and $w^\star_j$ defined by \eqref{w2} and \eqref{w3} satisfy 
      \begin{align}
	\label{w21}
	&\nabla \cdot \big( a^\star \nabla ( w^\star + \phi_j^\star h_j ) 
	+ (a^\star- \abar^\star) h \big)=0,
	\\
	\label{w31}
	&\nabla\cdot \big( a^\star \nabla ( w^\star_j + \phi_i^\star h_i \partial_j \ubar ) + \partial_j \ubar ( a^\star- \abar^\star) h \big)=0.
      \end{align}
      This allows us to rewrite the r.~h.~s.\ of \eqref{Diff0} as
      \begin{equation}
        \begin{aligned}
	  \int  h \cdot ( a - \abar) \nabla \delta  u
	  &\overset{\eqref{w21}}{=}-\int \nabla ( w^\star + \phi_j^\star h_j )\cdot a \nabla \delta  u
	  \\
	  &\overset{\eqref{Diff1}}{=} \int \nabla ( w^\star + \phi_j^\star h_j ) \cdot \delta a \nabla u,
        \end{aligned}
	\label{Diff4}
      \end{equation}
      and, similarly,
      \begin{equation}
        \begin{aligned}
	  -\int \partial_i \ubar h \cdot ( a - \abar) \nabla \delta  \phi_i
	  &\overset{\eqref{w31}}{=}\int \nabla ( w^\star_i + \phi_j^\star h_j \partial_i \ubar ) \cdot a \nabla \delta  \phi_i
	  \\
	  &\overset{\eqref{DefDeltaphi}}{=} -\int \nabla (w_i^\star +  \phi_j^\star h_j \partial_i \ubar )\cdot  \delta a ( e_i + \nabla \phi_i).
        \end{aligned}
	\label{Diff5}
      \end{equation}
      Inserting \eqref{Diff4} and \eqref{Diff5} into \eqref{Diff0}, using Leibniz' rule, and reordering the terms yields
      \begin{equation}
	\label{Diff6}
        \begin{aligned}
          \delta H
	  =~&
	  \int h_j (e_j + \nabla \phi_j^\star )\cdot \delta a ( \nabla u - \partial_i \ubar ( e_i + \nabla \phi_i))
	  \\
	  &+\int ( \nabla w^\star + \phi_j^\star \nabla h_j )  \cdot \delta a \nabla u
	  \\
	  &-\int \big(\nabla w_i^\star +  \phi_j^\star  \nabla (h_j \partial_i \ubar) \big) \cdot \delta a ( e_i + \nabla \phi_i).
        \end{aligned}
      \end{equation}
      Comparing \eqref{w1} with \eqref{12} yields by uniqueness (and Leibniz' rule)
      \begin{equation*}
	\nabla w + \phi_i \nabla \partial_i \ubar=
	\nabla u - \partial_i \ubar ( e_i + \nabla \phi_i).
      \end{equation*}
      Inserting this information into the first right-hand term of \eqref{Diff6} yields the desired identity \eqref{DeriveH}.
      
      We finally argue that $w^\star$ and $w^\star_j$ satisfy \eqref{w21} and \eqref{w31}, respectively.
      We consider the l.~h.~s.\ of \eqref{w21}, in which we replace the first term by means of \eqref{w2}, and then use \eqref{Id1} (for $\sigma^\star_j$):
      \begin{equation*}
        \begin{aligned}
          &\nabla \cdot \big( a^\star \nabla  ( w^\star + \phi_j^\star h_j ) + (a^\star- \abar^\star) h \big)
          \\
	  &\qquad\overset{\eqref{w2}}{=}
	  \nabla \cdot \big( \sigma^\star_j \nabla h_j + h_j a^\star \nabla \phi_j^\star + ( a^\star - \abar^\star) h \big)
	  \\
	  &\qquad\overset{\eqref{Id1}}{=} \nabla \cdot \big( h_j ( a^\star ( e_j + \nabla \phi^\star_j)- \abar^\star e_j -\nabla \cdot \sigma^\star_j ) \big)
	  \overset{\eqref{Id02prime}}{=}0.
        \end{aligned}
      \end{equation*}
      Similarly, we obtain \eqref{w31} by replacing $w^\star$ by $w^\star_j$ and $h$ by $\partial_j \ubar h$ in the above manipulations.
    
    \parag{Step 2:}
      We establish the following estimate
      \begin{equation}
      \begin{aligned}
        \langl (H- \langl H \rangl)^{2r}\rangl^{\frac{1}{r}}
        \lesssim~&
	\Big(\int h_j^4 \langl | e_j + \nabla \phi^\star_j |^{4r} \rangl^{\frac{1}{r}}
	\int \langl|(\phi, \sigma)|^{8 r}\rangl^{\frac{1}{2r}} |\nabla^2 \ubar|^4\Big)^{\frac{1}{2}}
	\\
	&+\Big(\int  \langl|(\phi^\star, \sigma^\star) |^{8r}\rangl^{\frac{1}{2r}} |\nabla h|^4 \int |f|^{4}\Big)^{\frac{1}{2}}
	\\
	&+\int  \langl|(\phi^\star,\sigma^\star)|^{8r}\rangl^{\frac{1}{4r}}  |  \nabla(\partial_i \ubar h) |^2\langl |e_i + \nabla \phi_i|^{4r} \rangl^{\frac{1}{2r}}.
      \end{aligned}
      \label{Num:006}
      \end{equation}
      Indeed, by an application of the spectral gap \eqref{Prop21a} into which we insert \eqref{DeriveH}, and making use of Jensen's inequality, we obtain
      \begin{align*}
        \langl (H- \langl H \rangl)^{2r}\rangl^{\frac{1}{r}}
        \lesssim~& \int  \Big\langl\big| \frac{\partial H}{\partial a(x)}\big|^{2r}\Big\rangl^{\frac{1}{r}} \dd x
	\\
	\lesssim~&
	\int  \big\langl | h_j ( e_j + \nabla \phi^\star_j) \otimes ( \nabla w + \phi_i \nabla \partial_i \ubar)|^{2r} \big\rangl^{\frac{1}{r}}
	\\
	&+\int  \big\langl| ( \nabla w^\star + \phi_j^\star \nabla h_j) \otimes \nabla u|^{2r} \big\rangl^{\frac{1}{r}}
	\\
	&+\int  \big\langl| ( \nabla w_i^\star + \phi^\star_j \nabla ( h_j \partial_i \ubar)) \otimes ( e_i + \nabla \phi_i )|^{2r}\big\rangl^{\frac{1}{r}}.
      \end{align*}
      Let us focus on the first above right-hand term. 
      By the Cauchy-Schwarz inequality in probability and in space, this yields
      \begin{align*}
	&\int  \langl | h_j ( e_j + \nabla \phi^\star_j) \otimes ( \nabla w + \phi_i \nabla \partial_i \ubar)|^{2r} \rangl^{\frac{1}{r}}
	\\
        &\qquad \leq \Big(\int \langl | h_j ( e_j + \nabla \phi^\star_j)|^{4r}\rangl^{\frac{1}{r}} \Big)^{\frac{1}{2}}
        \Big( \int\langl | \nabla w + \phi_i \nabla \partial_i \ubar|^{4r}\rangl^{\frac{1}{r}}\Big)^{\frac{1}{2}}.
      \end{align*}
      By similar manipulations and the stationarity of $\nabla \phi_i$, we get
      \begin{equation}
      \begin{aligned}
        \langl (H- \langl H \rangl)^{2r}\rangl^{\frac{1}{r}}
	& \lesssim
	\Big(\int \langl | h_j ( e_j + \nabla \phi^\star_j) |^{4r} \rangl^{\frac{1}{r}}\Big)^{\frac{1}{2}}
	\Big(\int \langl | \nabla w + \phi_i \nabla \partial_i \ubar|^{4r} \rangl^{\frac{1}{r}}\Big)^{\frac{1}{2}}
	\\
	&~~~~+\Big(\int \langl | \nabla w^\star + \phi_j^\star \nabla h_j |^{4r} \rangl^{\frac{1}{r}}\Big)^{\frac{1}{2}}
	\Big(\int \langl | \nabla u|^{4r} \rangl^{\frac{1}{r}}\Big)^{\frac{1}{2}}
	\\
	&~~~~+\int  \langl |\nabla w_i^\star + \phi^\star_j \nabla ( h_j \partial_i \ubar)|^{4r} \rangl^{\frac{1}{2r}}\langl |e_i + \nabla \phi_i|^{4r} \rangl^{\frac{1}{2r}}.
      \end{aligned}
      \label{Num:005}
      \end{equation}
      We now invoke the annealed estimate \eqref{Estim13_ter} from Proposition \ref{PropCZ}(ii) below with $r' \rightsquigarrow 4r$, $r\rightsquigarrow 8r$, and $p \rightsquigarrow 4$ on $\nabla \cdot ( a \nabla u + f )=0$, obtaining, since $f$ is deterministic,
      \begin{align*}
	\int \langl | \nabla u|^{4r} \rangl^{\frac{1}{r}}
	\lesssim \int  \langl |f|^{8r} \rangl^{\frac{1}{2r}} =\int |f|^{4}.
      \end{align*}
      Similarly, based on \eqref{w1}, \eqref{w2}, and \eqref{w3},  we have that
      \begin{align*}
	\int \langl | \nabla w |^{4r} \rangl^{\frac{1}{r}}
	\lesssim~&
        \int  \langl|(\phi,\sigma)|^{8 r}\rangl^{\frac{1}{2r}} |\nabla^2 \ubar|^4,
        \\
	\int \langl|\nabla w^\star|^{4r}\rangl^{\frac{1}{r}}
	\lesssim~& \int  \langl|(\phi^\star, \sigma^\star) |^{8r}\rangl^{\frac{1}{2r}} |\nabla h|^4,
        \\
	\int \langl |\nabla w^\star_i|^{4r} \rangl^{\frac{1}{2r}} 
	\lesssim~& \int  \langl|(\phi^\star,\sigma^\star)|^{8r}\rangl^{\frac{1}{4r}}  |  \nabla(\partial_i \ubar h) |^2.
      \end{align*}
      Inserting these four estimates into \eqref{Num:005} (and using Jensen's inequality in probability) entails \eqref{Num:006}.
      
    \parag{Step 3: Conclusion}
      By the corrector estimates \eqref{i} and \eqref{iii}, and Leibniz' rule followed by the Cauchy-Schwarz inequality in the last term, \eqref{Num:006} turns into
      \begin{equation}
      \begin{aligned}
        &\langl |H- \langl H \rangl|^{2r}\rangl^{\frac{1}{r}}
	\\
	&\qquad \lesssim
	\Big(\int |h|^4 \int \big|\mu_d(|\cdot|) \nabla^2 \ubar\big|^4\Big)^{\frac{1}{2}}
	+\Big(\int  \big|\mu_d(|\cdot|) \nabla h\big|^4 \int |(f,\nabla \ubar)|^{4}\Big)^{\frac{1}{2}}.
      \end{aligned}
      \label{Num:007}
      \end{equation}
      As for \eqref{Num:718}, we have
      \begin{align*}
        \int   |\nabla \ubar|^4 \lesssim \int  |f|^4 \et \int  \big|\mu_d(|\cdot|) \nabla^2 \ubar\big|^4 \lesssim \int  \big|\mu_d(|\cdot|) \nabla f\big|^4.
      \end{align*}
      Inserting these estimates into \eqref{Num:007} yields \eqref{Num:004}.
    \end{proof}

\section{Our tool: Annealed CZ estimates}\label{SecAnnealed}

  \subsection{General statement, and proof of the perturbative CZ estimates}
  The main contribution of this paper is a novel proof of the following result from \cite[Th.\ 6.1]{DuerinckxOtto_2019}:
  \begin{proposition}\label{PropCZ}
    Let the random fields $\nabla u$ and $f$ be square-integrable and related by
    \begin{align}\label{Prop3Defv}
      \nabla \cdot ( a \nabla u + f ) = 0.
    \end{align}
    Then:
    
    (i) {For any ensemble $\langl \cdot \rangl$ of elliptic coefficient fields $a$ satisfying \eqref{Ellipticite}, there holds
      \begin{align}
        \label{Prop3E1}
        \int  \langl |\nabla u|^r \rangl^{\frac{2}{r}}
        \lesssim_{d,\lambda}
        \int  \langl |f|^r \rangl^{\frac{2}{r}}, \qquad \text{ provided } |r-2| \ll_{d,\lambda} 1.
      \end{align}
      }
      
    (ii) {For any ensemble $\langl \cdot\rangl$ as in Section \ref{SecAssumpGauss}, there holds
      \begin{align}\label{Estim13_ter}
	\int \langl|\nabla u|^{r'}\rangl^\frac{p}{r'}
	\lesssim_{\gamma, p, r',r}
	\int \langl|f|^{r}\rangl^\frac{p}{r} \pourtout 1\leq r'<r\leq \infty \text{ and } 1 < p<\infty.
      \end{align}
	}
  \end{proposition}
  
  \parag{Remark}
    Note that the ensemble occurs in the \textit{inner} part of the norms in Proposition \ref{PropCZ}.
    In other words, the estimates are not about stochastic moments of the constant in a quenched CZ estimate, \textit{i.e.} in norms $\LL^r_{\langl\cdot\rangl}(\LL_{\R^d}^p)$, but rather about \textit{vector-valued} CZ estimates, \textit{i.e.} in norms $\LL_{\R^d}^p(\LL^r_{\langl\cdot\rangl})$ (in the terminology of \cite{McConnell_1984}, where the vector is represented by the values with respect to the ensemble).
    In particular, even in the case when $a$ is a deterministic constant coefficient, estimate \eqref{Prop3E1} is not trivial and requires refined tools (which however are well-established, see the proof of  Proposition \ref{PropCZ}(i) below).
    
    Proposition \ref{PropCZ}(i) holds under very general assumptions and rests only on the perturbative Meyers' approach to lift the above-mentioned corresponding statement for constant coefficients.
    On the contrary, Proposition \ref{PropCZ}(ii) requires homogenization techniques.
    Namely, we follow the basic idea of lifting the constant-coefficient regularity theory for the Helmholtz projection (see Section \ref{SS_AnnealedCZ} for a discussion).
    
    Proposition \ref{PropCZ}(ii) holds at the price of a loss of stochastic integrability.
    This loss is unavoidable for $p$ or $r$ far from $2$.
    This is not only due to the randomness of the local regularity of $a$, but more importantly, it is also a consequence of the randomness on large scales.
    To explain this, we fix  $|p-2| \gg_{d,\lambda} 1$ and consider the simpler framework of $\R^d$ replaced by the lattice $\Z^d$ (see \cite{Gloria_Otto}).
    The random coefficient field $a$ now lives on the edges of the lattice $\Z^d$ and takes two values according to i.~i.~d. Bernoulli variables.
    Then, it is easily seen that any given configuration on a domain of finite size may be found with positive probability.
    Choosing a configuration with an arbitrarily large constant $C$ in the CZ estimate%
    \footnote{Even for $d=2$, such a configuration may be obtained by using the singularities induced by coefficients piecewise constant on sectors.
    In such a case (see \cite{Kellogg_1972}), there exist $a$-harmonic functions $u$ such that $\nabla u \in \LL^2(\R^d)$ but $\nabla u \notin \LL^p(\R^d)$ for $p \gg_\lambda 1$.
    These counterexamples are available for the continuum framework and their large-scale behavior transmits to the discrete setting by interpreting it as a discretization of the continuum one.
    The latter argument is in the spirit of \cite[Prop.\ 22]{Fischer_Otto_2015}.}, that is
    \begin{equation*}
      \int |\nabla u|^p \geq C \int |f|^p
    \end{equation*}
    for some suitable $f$, we may build a random r.~h.~s.\ $f$ such that
    \begin{align*}
      \int \langle  |\nabla u|^p \rangle \geq C \quad \text{ but }\int \langle  |f|^p \rangle = 1,
    \end{align*}
    so that \eqref{Estim13_ter} cannot hold for $r'=r=p$.
    
  \begin{proof}[Argument for Proposition \ref{PropCZ}(i)]
  The main ingredient is the following: If $a=\Id$, then Proposition \ref{PropCZ}(i) is satisfied (here, we denote by $\Id$ the identity matrix). 
  A perturbation argument \textit{\`a la Meyers} yields the desired result for general elliptic coefficient field $a$ (see \textit{e.g.} \cite{Meyers_Estim} and a more recent presentation in \cite[Chap.\ 2, Th.\ 2.6.2, p.\ 122]{MuellerI} for the proof of the Meyers estimate).

    By \cite[Th.\ 1.1.]{McConnell_1984}, if $a=\Id$, then \eqref{Prop3E1} holds for all $r \in (1,\infty)$ since the space $\LL^{r}_{\langl \cdot\rangl}$ is UMD by \cite[Prop.\ 4.2.15, p.\ 291]{Weis_Book}.
    In this special case, we denote by $C_{r}$ the best constant in
    \begin{align*}
      \Big( \int  \langl |\nabla u|^r \rangl^{\frac{2}{r}} \Big)^{\frac{1}{2}} 
      \leq C_{r} \Big( \int  \langl |f|^r \rangl^{\frac{2}{r}} \Big)^{\frac{1}{2}}.
    \end{align*}
    We now rewrite \eqref{Prop3Defv} as
    \begin{align}\label{Truc}
      \nabla u = \nabla \Delta^{-1} \nabla \cdot ( \Id - a ) \nabla u - \nabla \Delta^{-1} \nabla \cdot f.
    \end{align}
    By \eqref{Ellipticite}, we have in the sense of operator norm on $\R^d$
    \begin{align*}
      |\Id-a| \leq \sqrt{1-\lambda}.
    \end{align*}
    Indeed, for all $\xi \in \R^d$, there holds
    \begin{equation*}
      |(\Id-a)\xi|^2 = \xi^2 + |a\xi|^2 -2 \xi \cdot a \xi \overset{\eqref{Ellipticite}}{\leq} \xi^2 - \xi \cdot a \xi \overset{\eqref{Ellipticite}}{\leq} (1-\lambda) |\xi|^2.
    \end{equation*}
    Therefore, the operator $\nabla \Delta^{-1} \nabla \cdot ( \Id - a )$ appearing in \eqref{Truc} is a contraction in $\LL_{\R^d}^{2}\LL^r_{\langl \cdot \rangl}$ provided
    \begin{align}\label{Condition}
      C_{r} \sqrt{1-\lambda} <1.
    \end{align}
    By the energy estimate we already know that $C_{2}=1$. 
    Fix an exponent $\bar{r} \in (2,\infty)$; by complex interpolation between spaces $\LL_{\R^d}^{2}\LL^r_{\langl \cdot \rangl}$, we obtain that for $1/r=(1-\theta)/2 + \theta/\bar{r}$ there holds:
    \begin{align*}
      C_{r} \leq C_{2}^{1-\theta} C_{\bar{r}}^{\theta} = C_{\bar{r}}^{\theta}, \quad \text{so that}\; \limsup_{r \downarrow 2} C_{r} \leq 1.
    \end{align*}
    Thus, if $|r-2| \ll_{d,\lambda} 1$, then \eqref{Condition} is satisfied, so that the operator $\nabla \Delta^{-1} \nabla \cdot ( \Id - a )$ appearing in \eqref{Truc} is indeed a contraction on $\LL_{\R^d}^{2}\LL^r_{\langl \cdot \rangl}$.
    This shows that the operator $\nabla ( \nabla \cdot a \nabla )^{-1} \nabla \cdot$ is bounded in $\LL_{\R^d}^{2}\LL^r_{\langl \cdot \rangl}$, which amounts to \eqref{Prop3E1}.
  \end{proof}

  \subsection{Strategy for proving the non-perturbative CZ estimates: massive equation and correctors}\label{SS_AnnealedCZ}
  
  The proof of Proposition 7.1 (ii) follows the philosophy of \cite{AvellanedaLin} in the sense that we appeal to homogenization to deduce the boundedness of the $a$-Helmholtz projection from that of the $\abar$-Helmholtz projection. Smallness of the homogenization error necessarily requires some amount of regularity of the r.~h.~s.\ $f$, see for instance Corollary \ref{Cor1}. Since for the operator norm, we have to consider arbitrary $f$, this suggests to decompose $f$ into a fairly smooth or large-scale part $f_<$ (we use the language of \textit{low-pass}) and an oscillatory small scale part $f_>$ (\textit{high-pass}); the corresponding $\nabla u_>$ then requires an independent argument. The latter argument is based on locality, meaning that $\nabla u_>(x)$  depends on $f$ only in a (large) ball around $x$.

  Given a length scale, such a decomposition into a smooth part and a local part is provided by the semi-group, which here we substitute by the resolvent ${\frac{1}{\tau}-\nabla\cdot a\nabla}$, which is called the \textit{massive} version of the elliptic operator. The notation is motivated by the interpretation of $\frac{1}{\tau}-\nabla\cdot a\nabla$ as the generator of a diffusion coupled to desorption at exponential rate $\frac{1}{\tau}$, and the language comes from quantum field theory where this term is well-known to provide an infra-red cut-off beyond length scales $\sqrt{\tau}$. More precisely, we consider $u_>:=(\frac{1}{\tau}-\nabla\cdot a\nabla)^{-1}\nabla\cdot f$, which up to exponentially small tails has the desired locality properties on scale $\sqrt{\tau}$ (as can be guessed from the form of the fundamental solution of $\frac{1}{\tau}-\triangle$). 

  As mentioned, we shall apply homogenization to $u_<:=u-u_>$. We depart from the general strategy in \cite{AvellanedaLin} by using a whole-space argument (as opposed to arguing on dyadically increasing balls typical for a Campanato iteration). Since we will appeal to the divergence-form representation of the residuum in the error of the two-scale expansion, \textit{cf.}~\eqref{12}, this whole-space argument would require the correctors $(\phi,\sigma)$ to be stationary, which would impose the restriction $d>2$. (This easier argument will be carried out in \cite{JosienOtto_2020}.) 
  Here, we avoid this restriction by equivalently characterizing the low-pass part as the solution of the \textit{massive} equation $u_<=(\frac{1}{\tau}-\nabla\cdot a\nabla)^{-1}\frac{1}{\tau}u$, so that when comparing $u_<$ to its homogenized counterpart $\bar u$  $:=(\frac{1}{\tau}-\nabla\cdot \abar_\tau\nabla)^{-1}\frac{1}{\tau}u$, the analogue of \eqref{12} will involve the \textit{massive} correctors (and the corresponding homogenized coefficient $\abar_\tau$ defined by \eqref{Def_qT} below), see Section \ref{Sec_MassCorr}. 
  The massive correctors are trivially stationary and satisfy estimates analogous to the  standard correctors, see Lemma~\ref{LemMassCorr}. 

  A ``collateral damage'' of carrying out the homogenization on the level of $u_<=(\frac{1}{\tau}-\nabla\cdot a\nabla)^{-1}\frac{1}{\tau}u$ (instead of $u_<=(-\nabla\cdot a\nabla)^{-1}\frac{1}{\tau}u_>$), leading to $\bar u$ $=(\frac{1}{\tau}-\nabla\cdot \abar_\tau\nabla)^{-1}\frac{1}{\tau}u$ is the following: We need to split also $\bar u$ into  $\bar u_>$ $:=(-\nabla\cdot \abar_\tau\nabla)^{-1}\frac{1}{\tau}u_>$ coming from the local part, and the remainder $\bar u_<:=\bar u-\bar u_>$, which again is expected to be small by homogenization, \textit{cf.} \eqref{Ann_ae10} and \eqref{Ann_ae11}.

  In view of this crucial role of the massive operator, it is convenient to unfold our task \eqref{Estim13_ter} by establishing the boundedness of the massive $a$-Helmholtz projection $\nabla(\frac{1}{\tau}-\nabla\cdot a\nabla)^{-1}\nabla\cdot$, and to include the zero-order term on the l.~h.~s.\ into the estimate, as well as to allow for a non-divergence form r.~h.~s., \textit{cf.}~\eqref{Ann_ae1}. This comes with the notational disadvantage that we need to monitor \textit{two} cut-off parameters $T\ge\tau$.
  We recover the massless CZ estimate \eqref{Estim13_ter} as the limit of its massive counterpart \eqref{Ann_ee5} (along with \eqref{Borne_C}).

  \subsubsection{The massive equation}
    We consider a massive version of \eqref{Prop3Defv}
      \begin{align}\label{Ann_ae1}
	\frac{1}{T}u-\nabla\cdot a\nabla u=\frac{1}{T}g + \nabla\cdot f
      \end{align}
      for a parameter $T\geq 1$, which we think of as being large.
      As mentioned above, our strategy is to derive annealed estimates for \eqref{Ann_ae1} (see Proposition \ref{PropmassCZ}).
      As explained above, we split $u$ as follows:
      \begin{align}
	\label{Ann_ae3}
	u=u_>+u_<,
      \end{align}
      where, for some $\tau \leq T$, $u_>$ is the high-pass part defined by
      \begin{equation}
        \label{Def_u_tau}
        \frac{1}{\tau}u_>-\nabla\cdot a\nabla u_>=\frac{1}{T}g + \nabla\cdot f,
      \end{equation} 
      and where the low-pass part $u_<$ satisfies
      \begin{align}
	\label{Ann_ae4}
	\frac{1}{\tau} u_< - \nabla \cdot a\nabla u_<= \big(\frac{1}{\tau}-\frac{1}{T}\big)u.
      \end{align}
      We approximate the low-pass part $u_<$ by the solution of the homogenized equation
      \begin{align}
	\label{Ann_ae5}
	\frac{1}{\tau} \ubar  - \nabla \cdot \abar_\tau \nabla \ubar = \big(\frac{1}{\tau}-\frac{1}{T}\big)u
      \end{align}
      for $\abar_\tau$ defined by \eqref{Def_qT} below.
      However, as mentioned above, on the homogenized level of $\ubar$, it turns out that we need to split once more into the high-pass and the low-pass parts:
      \begin{equation}
        \label{Ann_ae3_bis}
	\ubar
	=\ubar_>+\ubar_<,
      \end{equation} 
      where $\ubar_>$ is the high-pass part defined by
      \begin{equation}
        \label{Ann_ae10_bis}
	\frac{1}{T} \ubar_> - \nabla \cdot \abar_\tau \nabla \ubar_>
	=\big(\frac{1}{\tau}-\frac{1}{T}\big)u_>,
      \end{equation} 
      and where the low-pass part $\ubar_<$ satisfies
      \begin{equation}
        \label{Ann_ae11_bis}
	\frac{1}{T} \ubar_< - \nabla \cdot \abar_\tau \nabla \ubar_<
	=\big(\frac{1}{\tau}-\frac{1}{T}\big)(u_<-\ubar ).
      \end{equation} 
  
  \subsubsection{The massive correctors}\label{Sec_MassCorr}
    Expressing the homogenization error when passing from \eqref{Ann_ae4} to \eqref{Ann_ae5} requires the use of the \textit{massive} extended correctors $(\phi_\tau,\sigma_\tau,\psi_\tau)$ (see \cite[(48), (50) \& (51)]{GloriaOtto_2015} and also \cite{GloriaOtto_2017_Corr}).
      These objects are stationary solutions to the upcoming equations:
      Generalizing \eqref{Id04} and \eqref{Id03}, the massive correctors $\phi_{\tau,i}$ are defined through
      \begin{align}\label{Ann_as12}
	\frac{1}{\tau}\phi_{\tau,i}-\nabla\cdot a(\nabla\phi_{\tau,i}+e_i)=0,
      \end{align}
      and the massive flux correctors $\sigma_{\tau,i}=\{\sigma_{\tau,ijk}\}_{j, k=1,\cdots,d}$ are skew symmetric tensor fields given by
      \begin{equation}\label{Ann_as12s}
	\big(\frac{1}{\tau}-\Delta\big)\sigma_{\tau,ijk}=\partial_j q_{\tau,ik}-\partial_k q_{\tau,ij},
      \end{equation}
      where
      \begin{equation}\label{Def_qT}
        q_{\tau,ij}:=e_j\cdot a(e_i+\nabla\phi_{\tau,i}) \et \quad \abar_\tau e_i:=\langl q_{\tau,i}\rangl.
      \end{equation}
      The vector field $\psi_{\tau,i}$, which has no analogue in the case of $\tau=\infty$, is defined through
      \begin{align}\label{Ann_as11}
	\big(\frac{1}{\tau}-\Delta\big)\psi_{\tau,i}=q_{\tau,i}-\abar_\tau e_i-\nabla\phi_{\tau,i}.
      \end{align}
      The merit of $\psi_{\tau,i}$ is that the massive extended correctors $(\phi_\tau,\sigma_\tau,\psi_\tau)$ together satisfy the following generalization of \eqref{Num:1}:
      \begin{align}\label{Ann_as10}
	a(e_i+\nabla\phi_{\tau,i})=q_{\tau,i}=\abar_\tau e_i+\nabla\cdot\sigma_{\tau,i}+\frac{1}{\tau}\psi_{\tau,i}.
      \end{align}
      Note that each of the equations \eqref{Ann_as12}, \eqref{Ann_as12s} and \eqref{Ann_as11} has a unique solution in the class of bounded fields, which is thus stationary (see \textit{e.g.} \cite[Lem.\ 2.7]{GloriaOtto_2017_Corr}).
      
      The massive correctors enjoy properties similar to their massless counterparts (see Proposition \ref{Propcorr}):
      \begin{lemma}\label{LemMassCorr}
	Let $T \geq 1$.
	Under the assumptions of Section \ref{SecAssumpGauss}, for any $r \in [1, \infty)$, there holds:
	\begin{align}\label{Ann_ap56}
	  \Big\langl\big|\big(\nabla\phi_T,\nabla\sigma_T,\frac{\nabla\psi_T}{\sqrt{T}}\big)\big|^{2r}\Big\rangl^\frac{1}{r}
	  &\lesssim_{\gamma,r} 1,
	\\
	\label{Ann_ap57}
	\Big\langl\big|\big(\phi_T,\sigma_T,\frac{\psi_T}{\sqrt{T}}\big)\big|^{2r}\Big\rangl^\frac{1}{r}
	&\lesssim_{\gamma,r}
	\mu^2_d(\sqrt{T}),
	\end{align}
	where $\mu_d$ is defined by \eqref{Defmu}.
      \end{lemma}
      We postpone the proof of Lemma \ref{LemMassCorr} until Section \ref{Sec_LemMassCorr}.

      \begin{proof}[Argument for (\ref{Ann_as10})]
	By the uniqueness result \cite[Lem.\ 2.7]{GloriaOtto_2017_Corr}, it is enough to check (\ref{Ann_as10}) after applying the operator $\big(\frac{1}{\tau}-\Delta\big)$ to it, that is, in form of
	\begin{align*}
	  \big(\frac{1}{\tau}-\Delta\big)q_{\tau,i}=\frac{1}{\tau}\abar_\tau e_i+\nabla\cdot\big(\frac{1}{\tau}-\Delta\big)\sigma_{\tau,i} +\frac{1}{\tau}\big(\frac{1}{\tau}-\Delta\big)\psi_{\tau,i}.
	\end{align*}
	We first eliminate $\psi_\tau$ via its definition (\ref{Ann_as11}), leading to
	\begin{align*}
	-\Delta q_\tau=\nabla\cdot\big(\frac{1}{\tau}-\Delta\big)\sigma_\tau-\frac{1}{\tau}\nabla\phi_\tau.
	\end{align*}
	We then eliminate $\sigma_\tau$ by inserting $\nabla\cdot\big(\frac{1}{\tau}-\Delta\big)\sigma=\nabla\nabla\cdot q_\tau-
	\Delta q_\tau$, which is obtained by applying $\partial_k$ to equation (\ref{Ann_as12s}), to the effect of
	\begin{align*}
	0=\nabla\nabla\cdot q_\tau-\frac{1}{\tau}\nabla\phi_\tau.
	\end{align*}
	This is nothing else than the operator $-\nabla$ applied to equation \eqref{Ann_as12}, appealing to definition \eqref{Def_qT}.
      \end{proof}
  
  \subsubsection{Splitting}\label{Sec:0_Splitting}
    As announced, the massive extended correctors allow us to characterize the error in the massive two-scale expansion
    \begin{align}
      \label{Ann_ae6}
      w:=u_<-(1+\phi_{\tau,i} \partial_i)\ubar 
    \end{align}
    via
    \begin{equation}
      \begin{aligned}
	\frac{1}{\tau} w - \nabla \cdot a \nabla w
	=~& \nabla \cdot \big( ( \phi_{\tau,i} a  -\sigma_{\tau,i}) \nabla \partial_i \ubar  + \frac{1}{\tau} \partial_i \ubar \psi_{\tau,i}  \big)
	- \frac{1}{\tau}\phi_{\tau,i} \partial_i \ubar .
      \end{aligned}
      \label{Ann_ae7}
    \end{equation}
    This induces the following splitting $(u,\nabla u)=(v_>,h_>)+(v_<,h_<)$ into high-pass and low-pass parts via
      \begin{align}
	\left\{
	\begin{aligned}
	  v_>&:=u_> + \ubar_>,
	  \\
	  v_<&:=w+\ubar_< + \phi_{\tau,i} \partial_i \ubar ,
	  \\
	  h_>&:=\nabla u_> +  \partial_i \ubar_>(e_i + \nabla \phi_{\tau,i}),
	  \\
	  h_<&:=\nabla w + \partial_i \ubar_< (e_i + \nabla \phi_{\tau,i})  + \phi_{\tau,i} \nabla \partial_i \ubar,
	\end{aligned}
	\right.
	\label{Ann_ae12}
      \end{align}
      \textit{cf.} \eqref{Ann_ae3}, \eqref{Ann_ae3_bis}, and \eqref{Ann_ae6}, which will be used in Parts 1 and 2 of the proof of Proposition~\ref{PropCZ}.
      
      \parag{Argument for \eqref{Ann_ae7}} Equation \eqref{Ann_ae7} is derived by taking the difference between \eqref{Ann_ae4} and \eqref{Ann_ae5}, and by appealing to the intertwining relation (which generalizes \eqref{1b})
      \begin{align*}
      \nabla\cdot a\nabla(1+\phi_{\tau,i}\partial_i)\ubar =\nabla\cdot\abar_\tau\nabla \ubar 
      +\nabla\cdot\big((\phi_{\tau,i} a-\sigma_{\tau,i}) \nabla \partial_i\ubar +\frac{1}{\tau} \partial_i\ubar \psi_{\tau,i} \big).
      \end{align*}
      This intertwining relation itself follows via $\nabla(1+\phi_{\tau,i}\partial_i)\ubar $ $=\partial_i\ubar (e_i+\nabla\phi_{\tau,i})+\phi_{\tau,i}\nabla \partial_i\ubar $ from multiplying (\ref{Ann_as10}) by $\partial_i\ubar $ and using the identity \eqref{Id1} in form of $\nabla\cdot(\partial_i\ubar \nabla\cdot\sigma_{\tau,i})=-\nabla\cdot(\sigma_{\tau,i}\nabla \partial_i\ubar )$, which relies on the skew symmetry of $\sigma_{\tau,i}$.

  \subsubsection{Result for the massive equation}
      For further discussion, we denote the norm appearing in \eqref{Estim13_ter} by
      \begin{equation}\label{Num:701}
	\|h\|_{p,r}:=\Big(\int \big\langl|h|^r\big\rangl^\frac{p}{r}\Big)^\frac{1}{p}.
      \end{equation}
      As announced, we generalize Proposition \ref{PropCZ}(ii) to:      
      \begin{proposition}\label{PropmassCZ}
        Fix the exponents $p\in (1,\infty)$ and $1 \leq r'<r < \infty$.
        Let $\barC_{p,r',r}(T) \geq 1$ denote the smallest constant such that for all square-integrable random fields $u$, $g$, and $f$ related through the massive equation \eqref{Ann_ae1}, we have
	\begin{align}
	  \label{Ann_ee5}
	  \big\|\big(\frac{u}{\sqrt{T}},\nabla u\big)\big\|_{p,r'} \leq 
	  \barC_{p,r',r}(T) \big\|\big(\frac{g}{\sqrt{T}},f\big)\big\|_{p,r}.
	\end{align}
	Under the assumptions of Section \ref{SecAssumpGauss}, this constant satisfies
	\begin{equation}\label{Borne_C}
	  \barC_{p,r',r}(T) \lesssim_{\gamma,p,r',r} 1 \quad \pourtout T \geq 1.
	\end{equation}
      \end{proposition}
      
      The proof of Proposition \ref{PropmassCZ} is done in Section \ref{Sec:ProofPropMassCZ} and relies on only two ingredients:
      \begin{itemize}
        \item{the homogenized (constant-coefficient) operator $\frac{1}{T}-\nabla\cdot\abar_\tau\nabla$ satisfies \eqref{Ann_ee5}, with a constant $\bar{C}_{p,r',r}$ independent of $T$, and even for $r'=r$ (see Section \ref{Ann_ssconst});}
        \item{the massive correctors are strictly sublinear in $\sqrt{T}$ (see Lemma \ref{LemMassCorr}).}
      \end{itemize}
      Philosophically speaking, we substitute any regularity theory for the operator $\frac{1}{T}-\nabla\cdot a\nabla$ by regularity theory for the operator $\frac{1}{T}-\nabla\cdot\abar_\tau\nabla$.
      This substitution relies on functional analysis, mostly interpolation.
      Literally, the only large-scale regularity theory ingredient for the massive operator $\frac{1}{T}-\nabla\cdot a\nabla$ is the whole-space energy estimate (see Section \ref{Sec:Lem_SsOpti}).
      This approach quite different from the one based on quenched CZ estimates (briefly described in Section \ref{SecLocal}).
  
  \subsubsection{Strategy for dealing with the loss in the stochastic integrability}\label{Sec:Strat}
      The strategy is to start from the standard annealed CZ estimates for the solution $\ubar$ of the constant-coefficient equation \eqref{Ann_ae5} (\textit{cf.}~Section \ref{Ann_ssconst}) and to buckle on the level of the optimal constant $\barC_{p,r',r}(T)$ in \eqref{Ann_ee5} by applying CZ estimates also to the equation \eqref{Ann_ae7} for $w$. 
      The main challenge in this strategy is the loss in stochastic integrability coming from the need to feed in the estimates on the massive correctors (\ref{Ann_ap57}). 
      To overcome this challenge, we use a real interpolation argument. 
      The loss in stochastic integrability of the low-pass part $u_<$ is compensated by its smallness when $\tau$ is large and by the gain in stochastic integrability of the high-pass part (see Parts 1 and 2 of the proof of Proposition \ref{PropmassCZ}).
      
      To carry out this real interpolation argument, we need an independent estimate of the high-pass part $u_>$. 
      It is provided in Section \ref{Sec:Lem_SsOpti} by Lemma \ref{Lem_SsOpti}, which is based on the locality of $u_>$. 
      The estimate \eqref{Ann_ee2*} therein is highly suboptimal in its scaling in $\sqrt{T}$ (which plays the role of $\sqrt{\tau}$); this however is balanced by the smallness of $u_<$ in our real interpolation argument.
      Finally, by complex interpolation starting from the pivotal energy estimate \eqref{Ann_ap25}, we iteratively enlarge the zone in which we have $\barC_{p,r',r}(T) \lesssim 1$, reaching any admissible $3$-tuple of exponents $(p,r',r)$ in a finite number of steps (see Parts 3 and 4 of the proof of Proposition \ref{PropmassCZ}).
  

    \subsubsection{The constant-coefficient estimates}\label{Ann_ssconst}

      We only need one result for the constant-coefficient equation, namely:
      \begin{lemma}\label{LemCstMassCZ}
        Assume that $\abar$ is a constant coefficient satisfying  \eqref{Ellipticiteabar}. Then, for any $p, r \in (1,\infty)$, the solution $\ubar$ to the massive equation
	\begin{align}
	  \label{Ann_masseq}
	  \frac{1}{T} \ubar - \nabla \cdot \abar \nabla \ubar =\frac{1}{T} g + \nabla \cdot f
	\end{align}
	satisfies
	\begin{align}
	  \label{Ann_eee23}
	  \big\|\big( \frac{\ubar}{\sqrt{T}}, \nabla \ubar \big) \big\|_{p,r} \lesssim_{d,\lambda,p,r} \big\|\big( \frac{g}{\sqrt{T}}, f \big)\big\|_{p,r}.
	\end{align}
      \end{lemma}
      
      The proof is based on a version of the Mikhlin theorem, which involves Fourier multipliers \cite[Th.\ 1.1]{McConnell_1984}.
      By scaling, we may assume $T=1$.
      The functional space $\LL^{r}_{\langl\cdot\rangl}$ is a UMD space (see \cite[Def.\ 4.2.1, p.\ 281]{Weis_Book} for a definition).
      Moreover, the Fourier multiplier $m$ corresponding to the solution operator $\mathfrak{M}:(g,f) \mapsto (\ubar,\nabla \ubar )$ of \eqref{Ann_masseq} has the block structure
      \begin{align*}
	m(k)=\frac{1}{1 + k \cdot \abar k} \left( \begin{array}{c|c} 1 & \ii  k^\star \\ \hline \ii k & -  k \otimes k \end{array} \right),
      \end{align*}
      where $k^\star$ is the vector $k$ transposed.
      Obviously, the symbol $m$ belongs to $\CC^{d+1}(\R^d \backslash \{0\})$, and satisfies
      \begin{align*}
	\sup_{k \in \R^d\backslash\{0\}}\sum_{j=0}^{d+1} |k|^j |\nabla^j m(k)| \lesssim_{d,\lambda} 1.
      \end{align*}
      Therefore, by \cite[Th.\ 1.1]{McConnell_1984}, the operator $\mathfrak{M}$ extends from $\LL^p(\R^d)$ to $\LL^{p}(\R^d,\LL^r_{\langl\cdot\rangl})$.
      This establishes \eqref{Ann_eee23}.

    \subsection{Suboptimal CZ estimates for the massive equation}\label{Sec:Lem_SsOpti}
    As mentioned above, we need a robust but suboptimal estimate of $\barC_{p,r',r}(T)$:
    \begin{lemma}\label{Lem_SsOpti}
	Let $T \geq 1$.
	Under the assumptions of Section \ref{SecAssumpGauss}, the operator norm denoted by $\barC_{p,r',r}(T)$ in \eqref{Ann_ee5} satisfies
	\begin{align}
	\label{Ann_ap25}
	\barC_{2,2,2}(T)&\lesssim_{d,\lambda} 1,
	\\
	\label{Ann_ee2*}
	\barC_{p,r',r}(T) &\lesssim_{\gamma,p,r',r}
	\sqrt{T}^d \qquad \text{provided } p \in (1,\infty) \text{ and } 1\leq r'<r<\infty.
      \end{align}
    \end{lemma}
    Lemma \ref{Lem_SsOpti} relies on the locality on scale $\sqrt{T}$ of solutions to the massive equation, and on the regularity on scale $1$ of the coefficient field (see Lemma \ref{Lem_AnnealedCZ_2}).
    These two properties allow for estimates between different spatial $\LL^p$-norms, where locality provides the large-scale cut-off, and regularity the small-scale cut-off. 
    In particular, we will jump between the $\LL^2_{\R^d}$-norm, on which scale we have energy estimates, and the $\LL^\infty_{\R^d}$-norm, where we may handle the $\LL^r_{\langle\cdot\rangle}$-norm.

    \begin{proof}
    Throughout the proof, the square-integrable random fields $u$, $g$, and $f$ are related by \eqref{Ann_ae1}.
    By a duality argument and Jensen's inequality, we may restrict to the case of $2\le r'<r$. As mentioned above, the local regularity of $a$ allows us to use CZ estimates on scales $\le 1$. We capitalize on this in form of Lemma \ref{Lem_AnnealedCZ_2}(i), via a family of norms that treat scales $\le 1$ separately, namely
    \begin{align}\label{Ann_ap20}
    \|h\|_{p,r,q}:=\bigg(\int \Big\langl\Big(\fint_{\Boule_1(x)}|h|^q\Big)^\frac{r}{q}\Big\rangl^\frac{p}{r}\dd x\bigg)^\frac{1}{p},
    \end{align}
    where we think of the exponent $q \in(1,\infty)$ of the innermost spatial norm as being close to $1$ (of course, the precise value $1$ of the radius in \eqref{Ann_ap20} is not important).
    In Step 1, we will derive the crucial property of these norms, namely that they decrease with increasing spatial exponent $p$, see \eqref{Ann_ap39}.
    As mentioned above, the massive term provides an approximate locality on scale $\sqrt{T}$, which in Step 2 we capture through the energy estimate with exponential weight on that scale, see \eqref{Ann_ap34}. 
    In Step 3, we use local regularity to express this weighted energy estimate on the level of the norms \eqref{Ann_ap20}. 
    In Step 4, we derive the statement of this lemma on the scale of the norms \eqref{Ann_ap20}, see \eqref{Num:703}. 
    In Step 5, finally, we use once more local regularity to return to the original norms \eqref{Num:701}.

    \parag{Step 1: Nestedness properties of the norms}
      We claim the following properties of the norms \eqref{Ann_ap20}:
      \begin{align}
	\label{Ann_ap39}
	&\|h\|_{\bar{p},r,q}\lesssim\|h\|_{p,r,q} &&\mbox{provided}\;p\le \bar{p},
	\\
	\label{Num:03_bis}
        &\|h\|_{p,r'}\leq\|h\|_{p,r} \quad \et \quad \|h\|_{p,r',q}\leq\|h\|_{p,r,q} &&\mbox{provided}\;r'\le r,
        \\
	\label{Num:706}
        &\|h\|_{p,r,q}\leq \|h\|_{p,r} &&\mbox{provided}\; q \leq \min\{p,r\}.
      \end{align}
      
      We start with the argument for (\ref{Ann_ap39}). 
      The core is the following discrete $\ell_{\bar{p}}-\ell_p$ estimate,
      where we introduce the abbreviation $\bar{h}_R(y):=\langl(\fint_{\Boule_R(y)}|h|^q)^\frac{r}{q}\rangl^\frac{1}{r}$:
      \begin{align*}
      \big(\sum_{z\in\mathbb{Z}^d}|\bar h_1(x+z)|^{\bar{p}}\big)^\frac{1}{{\bar{p}}}
      \le
      \big(\sum_{z\in\mathbb{Z}^d}|\bar h_1(x+z)|^p\big)^\frac{1}{p}.
      \end{align*}
      Using that $0\le\bar {h}_1(y)\lesssim\bar {h}_2(y')$ for $|y'-y|<1$ and thus $|\bar h_1(y)|^p\lesssim\fint_{\Boule_1(y)}|\bar h_2|^p$, the above inequality may be upgraded to
      \begin{align*}
	\big(\sum_{z\in\mathbb{Z}^d}|\bar h_1(x+z)|^{\bar{p}}\big)^\frac{1}{{\bar{p}}}
	\lesssim
	\Big(\int |\bar h_2|^p\Big)^\frac{1}{p}.
      \end{align*}
      Taking the $\LL^{\bar{p}}$-norm in $x\in[0,1)^d$ of this estimate, we obtain
      \begin{align}\label{Ann_ee21}
      \Big(\int |\bar h_1|^{\bar{p}}\Big)^\frac{1}{{\bar{p}}}
      \lesssim\Big(\int |\bar h_2|^p\Big)^\frac{1}{p}.
      \end{align}
      Using now the elementary geometric fact that there exist $N\lesssim 1$ shift vectors $z_1,\cdots,z_N\in\R^d$ such that $\Boule_2(x)$ $\subset\bigcup_{n=1}^N \Boule_1(x-z_n)$ and therefore $(\fint_{\Boule_2(x)}|h|^q)^\frac{1}{q}$  ${\leq\sum_{n=1}^N(\fint_{\Boule_1(x-z_n)}|h|^q)^\frac{1}{q}}$, we obtain by the triangle inequality and the shift invariance of the norm $(\int \langl|h|^r\rangl^\frac{{\bar{p}}}{r})^\frac{1}{{\bar{p}}}$ that
      \begin{align}\label{Ann_ap23}
	\bigg(\int \Big\langl\Big(\fint_{\Boule_2(x)}|h|^q\Big)^\frac{r}{q}\Big\rangl^\frac{p}{r}\dd x\bigg)^\frac{1}{p}
	\le N\|h\|_{p,r,q}\lesssim \|h\|_{p,r,q}.
      \end{align}
      By definitions of $\bar h_1$ and of $\bar h_2$, the latter in conjunction with \eqref{Ann_ap23},
      \eqref{Ann_ee21} turns into (\ref{Ann_ap39}).

      Inequalities \eqref{Num:03_bis} are obvious from Jensen's inequality in probability.
      
      We finally turn to \eqref{Num:706}, which is a consequence of Jensen's inequality:
      Rewriting definition \eqref{Ann_ap20} in form of
      \begin{align*}
      \|h\|_{p,r,q}&=\Big(\int\big\langle\big(\fint_{\Boule_1}|h(x+z)|^qdz\big)^\frac{r}{q}
      \big\rangle^\frac{p}{r}dx\Big)^\frac{1}{p}
      =\Big\|\big\|\fint_{\Boule_1}|h(\cdot+z)|^qdz\big\|_{\LL^\frac{r}{q}_{\langle\cdot\rangle}}
      \Big\|_{\LL^\frac{p}{q}_{\mathbb{R}^d}}^\frac{1}{q},
      \end{align*}
      we learn from the convexity and translation invariance of the involved norms  (here we use $q\le\min\{p,r\}$) that
      \begin{align*}
      \|h\|_{p,r,q}&\le
      \Big(\fint_{\Boule_1}\big\|\||h(\cdot+z)|^q\|_{\LL^\frac{r}{q}_{\langle\cdot\rangle}}
      \big\|_{\LL^\frac{p}{q}_{\mathbb{R}^d}}dz\Big)^\frac{1}{q}
      \stackrel{\eqref{Num:701}}{=}\|h\|_{p,r}.
      \end{align*}
      
    \parag{Step 2: The pivotal estimate \eqref{Ann_ap25}}
      In fact, in the next step, we need its local version
      \begin{equation}\label{Ann_ap34}
	\big\|\omega_T \big(\frac{u}{\sqrt{T}},\nabla u\big)\big\|_{2,2}
	\lesssim_{d,\lambda} \big\|\omega_T \big(\frac{g}{\sqrt{T}},f\big)\big\|_{2,2},
      \end{equation}
      where the weight is of exponential form 
      \begin{equation}\label{Ann_ap21}
      \omega_T(x):=\exp(-\frac{|x|}{C\sqrt{T}})
      \end{equation}
      for a constant $C=C(d,\lambda)$ fixed below.

      Starting with \eqref{Ann_ap25}, we test \eqref{Ann_ae1} with $u$, use the uniform $\lambda$-ellipticity of $a$ and the Cauchy-Schwarz inequality, which gives
      \begin{equation}\label{Energ}
      \int  \big|\big(\frac{u}{\sqrt{T}},\nabla u\big)\big|^2
      \lesssim
      \int \big|\big(\frac{g}{\sqrt{T}},f\big)\big|^2.
      \end{equation}
      We now take the expectation, which we exchange with the spatial integral, yielding \eqref{Ann_ap25} in the form of
      \begin{align}\label{Ann_ap97}
      \big\|\big(\frac{u}{\sqrt{T}},\nabla u\big)\big\|_{2,2}
      \lesssim\big\|\big(\frac{g}{\sqrt{T}},f\big)\big\|_{2,2}.
      \end{align}


      We now upgrade \eqref{Ann_ap97} to \eqref{Ann_ap34}:
      Multiplying (\ref{Ann_ae1}) with a cut-off function $\omega$ we obtain
      by Leibniz' rule
      \begin{align*}
	\frac{1}{T}\omega u-\nabla\cdot a\nabla (\omega u)=\frac{1}{T}\omega g-\nabla\omega\cdot(a\nabla u+f)+\nabla\cdot(\omega f-au\nabla\omega).
      \end{align*}
      Appealing to \eqref{Ann_ap97} yields
      \begin{align*}
      \big\|\omega\big(\frac{u}{\sqrt{T}},\nabla u\big)\big\|_{2,2}
      \lesssim\big\|\omega\big(\frac{g}{\sqrt{T}},f\big)\big\|_{2,2}
      +\sqrt{T}\|\nabla\omega\cdot(a\nabla u+f)\|_{2,2}
      +\|u\nabla\omega\|_{2,2}
      .
      \end{align*}
      Specifying $\omega$ to be of the form of \eqref{Ann_ap21}, for which $|\nabla\omega_T|\le\frac{1}{C\sqrt{T}}\omega_T$, this entails
      \begin{align*}
      &\big\|\omega_T\big(\frac{u}{\sqrt{T}},\nabla u\big)\big\|_{2,2}
      \\
      &\qquad \lesssim \frac{1+C}{C}\|\omega_T f\|_{2,2} +\frac{1}{\sqrt{T}}\|\omega_T g\|_{2,2}
      +\frac{1}{C\sqrt{T}}\|\omega_T u\|_{2,2}+\frac{1}{C}\|\omega_T\nabla u\|_{2,2}.
      \end{align*}
      For $C \gg_{d,\lambda} 1$, the two last r.~h.~s.\ terms may be absorbed, giving rise to \eqref{Ann_ap34}.
      
    \parag{Step 3: Pivotal estimate on the scale of norms \eqref{Ann_ap20}}
      We claim that, for any $s'>2$ and $q \in (1,2]$, there exists a constant $C \gg_{d,\lambda,s',q} 1$ such that, defining $\omega_T$ by \eqref{Ann_ap21}, there holds
      \begin{equation}
        \label{Num:702}
        \big\|\omega_T \big(\frac{u}{\sqrt{T}},\nabla u\big)\big\|_{2,2,q}
	\lesssim_{\gamma,s',q} \big\|\omega_T \big(\frac{g}{\sqrt{T}},f\big)\big\|_{2,s',q}.
      \end{equation}
      
      The argument relies on complex interpolation.
      We choose a $q' \in (1,q)$ and then define $\theta \in (0,1)$ and $s \in (s',\infty)$ through
      \begin{equation*}
        \frac{1}{q}= \frac{\theta}{2}+\frac{1-\theta}{q'} \quad \et \quad \frac{1}{s'}=\frac{\theta}{2}+\frac{1-\theta}{s}.
      \end{equation*}
      We make use of the estimate \eqref{Num:116} of Lemma \ref{Lem_AnnealedCZ_2}(ii) (replacing $(r,q) \rightsquigarrow (s,q')$), which we copy here:
      \begin{equation}\label{Num:116_bis}
        \big\|\big(\frac{u}{\sqrt{T}},\nabla u\big)\big\|_{2,2,q'} \lesssim \big\|\big(\frac{u}{\sqrt{T}},\nabla u\big)\big\|_{2,s,q'}.
      \end{equation}
      Appealing to \eqref{Num:556}, we may identify the norms $\|\cdot\|_{2,2}$ and $\|\cdot\|_{2,2,2}$ so that \eqref{Ann_ap34} reads
      \begin{equation}\label{Ann_ap34_bis}
        \big\|\omega_T \big(\frac{u}{\sqrt{T}},\nabla u\big)\big\|_{2,2,2}
	\lesssim \big\|\omega_T \big(\frac{g}{\sqrt{T}},f\big)\big\|_{2,2,2},
      \end{equation}
      By complex interpolation between \eqref{Num:116_bis} and \eqref{Ann_ap34_bis} (using the Stein-Weiss theorem \cite[Th.\ 5.4.1 p.\ 115]{BerghLofstrom}), we get
      \begin{equation*}
        \big\|\omega_T^\theta \big(\frac{u}{\sqrt{T}},\nabla u\big)\big\|_{2,2,q}
	\lesssim_{\gamma,s',q} \big\|\omega_T^\theta \big(\frac{g}{\sqrt{T}},f\big)\big\|_{2,s',q}.
      \end{equation*}
      Since $\omega_T^\theta(x)=\exp(-\frac{\theta|x|}{C\sqrt{T}})$ by \eqref{Ann_ap21}, we obtain \eqref{Num:702} by adapting the definition of~$C$.
      
    \parag{Step 4: Suboptimal estimates on the scale of norms \eqref{Ann_ap20}}
      We now are given $p \in (1,\infty)$, $2< r'<r$, and $q \leq \min\{p,2\} \in (1,2]$, and establish
      \begin{equation}\label{Num:703}
	\big\|\big(\frac{u}{\sqrt{T}},\nabla u\big)\big\|_{p,r',q}
	\lesssim_{\gamma,p,r',r,q} \sqrt{T}^d \big\|\big(\frac{g}{\sqrt{T}},f\big)\big\|_{p,r,q}.
      \end{equation}
      
      The two ingredients for \eqref{Num:703} are the following norm relations, to be established below:
      \begin{align}
      &\sup_{x} \Big\langl \Big(\fint_{\Boule_1(x)} |h|^q \Big)^{\frac{2}{q}} \Big\rangl^{\frac{1}{2}} =: \|h\|_{\infty,2,q}
      \lesssim\|h\|_{2,2,q} \quad\mbox{and}
      \label{Ann_ap35}
      \\
      &
      \left\{
      \begin{aligned}
        &\|\omega_T h\|_{2,r,q}\lesssim\sqrt{T}^{\max\{0,d(\frac{1}{2}-\frac{1}{p})\}}\|\omega_{4T} h\|_{p,r,q},
	\\
	&\|\omega_{\frac{T}{4}}h\|_{p,r',q}\lesssim\sqrt{T}^\frac{d}{p}\|\omega_{T} h\|_{\infty,r',q}.
      \end{aligned}
      \right.
      \label{Ann_ap38}
      \end{align}
      The merit of passing to the spatial exponent $p=\infty$ is that by duality,
      \begin{align}
	\|h\|_{\infty,r',q}
	=~&\sup_{\langl|F|^\frac{2r'}{r'-2}\rangl\le 1}\|Fh\|_{\infty,2,q}.
	\label{Ann_ap36}
      \end{align}
      Estimate \eqref{Ann_ap35} is immediate from \eqref{Ann_ap39}. 
      In the case of $p\ge 2$, the first estimate in \eqref{Ann_ap38} follows from appealing to the relation $\omega_T=\omega_{4T}^2$, \textit{cf.} \eqref{Ann_ap21}, which allows to use the H\"older inequality, so that it reduces to the obvious $\|\omega_{4T}\|_{\frac{2p}{p-2},\infty,\infty}$ ${\lesssim\sqrt{T}^{d(\frac{1}{2}-\frac{1}{p})}}$ (recall $T \geq 1$). 
      In the case of $p\le 2$, we appeal to \eqref{Ann_ap39} and the obvious $\omega_T$ $\le\omega_{4T}$. For the second estimate in \eqref{Ann_ap38}, we start from $\omega_{T/4}=\omega_T^2$, use the H\"older inequality, and $\|\omega_{T}\|_{p,\infty,\infty} \lesssim\sqrt{T}^{\frac{d}{p}}$.

      Turning now to \eqref{Num:703}, we define the exponent $s'>2$ by
      \begin{equation}\label{Num:751}
        \frac{1}{s'}:=\frac{1}{2}-\frac{1}{r'}+\frac{1}{r}=\frac{r'-2}{2r'}+\frac{1}{r}.
      \end{equation}
      We specify the constant $C$ in \eqref{Ann_ap21} to be the one of Step 3 belonging to $q$ and~$s'$.
      Let $F$ be an auxiliary random variable.
      Since it does not depend on space, \eqref{Ann_ae1} is preserved by multiplication with $F$. 
      Hence by \eqref{Num:702} followed by the H\"older inequality in probability (based on the definition \eqref{Num:751} of $s'$) we obtain
      \begin{align*}
      \big\|\omega_{T}F \big(\frac{u}{T},\nabla u\big)\big\|_{2,2,q}
      &\lesssim
      \|\omega_T F \big(\frac{g}{\sqrt{T}},f\big)\|_{2,s',q}
      \le
      \big\|\omega_T \big(\frac{g}{\sqrt{T}},f\big)\big\|_{2,r,q}
      \langl|F|^{\frac{2r'}{r'-2}}\rangl^\frac{r'-2}{2r'}.
      \end{align*}
      In combination with \eqref{Ann_ap35}, this yields 
      \begin{align*}
      \big\|\omega_{T}F \big(\frac{u}{\sqrt{T}},\nabla u\big)\big\|_{\infty,2,q}
      \lesssim
      \big\|\omega_T \big(\frac{g}{\sqrt{T}},f\big)\big\|_{2,r,q}
      \langl|F|^{\frac{2r'}{r'-2}}\rangl^\frac{r'-2}{2r'},
      \end{align*}
      so that by \eqref{Ann_ap36}
      \begin{align*}
      \big\|\omega_{T}\big(\frac{u}{\sqrt{T}},\nabla u\big)\big\|_{\infty,r',q}\lesssim
      \big\|\omega_T \big(\frac{g}{\sqrt{T}},f\big)\big\|_{2,r,q},
      \end{align*}
      which by \eqref{Ann_ap38} implies, using $\max\{0,d(\frac{1}{2}-\frac{1}{p})\}+\frac{d}{p} \leq d$,
      \begin{align*}
      \|\omega_{\frac{T}{4}}\big(\frac{u}{\sqrt{T}},\nabla u\big)\|_{p,r',q}
      \lesssim\sqrt{T}^{d}
      \|\omega_{4T} \big(\frac{g}{\sqrt{T}},f\big)\|_{p,r,q}.
      \end{align*}
      Since by the definition \eqref{Ann_ap20} and the properties of $\omega_T$ we have
      \begin{align*}
      \|\omega_T(\cdot-z)h\|_{p,r,q}\sim
      \Big(\int\Big(\omega_T(x-z)\big\langle\big(\fint_{\Boule_1(x)}|h|^q\big)^\frac{r}{q}\big\rangle^\frac{1}{r}\Big)^pdx
      \Big)^\frac{1}{p},
      \end{align*}
      the desired \eqref{Num:703} follows from taking the $L^p$-norm in the shift $z$.
      
    \parag{Step 5: Conclusion}
      Let $p \in (1,\infty)$.
      By duality and Jensen's inequality, it suffices to establish \eqref{Ann_ee2*} for $2 \leq r' < r < \infty$.
      We set $q:=\min\{p,r\}$ and select an $r'<s<r$. Appealing to Lemma \ref{Lem_AnnealedCZ_2}(i) we have
      \begin{align*}
      \|(\frac{u}{\sqrt{T}},\nabla u)\|_{p,r'}\lesssim\|(\frac{u}{\sqrt{T}},\nabla u)\|_{p,s,q}
      +\|(\frac{g}{\sqrt{T}},f)\|_{p,s}.
      \end{align*}
      Estimating the first and the second r.~h.~s.~term by \eqref{Num:703} (with $r'$ replaced by $s$) and by \eqref{Num:03_bis}, respectively, we obtain
      \begin{align*}
      \|(\frac{u}{\sqrt{T}},\nabla u)\|_{p,r'}\lesssim\sqrt{T}^d\|(\frac{g}{\sqrt{T}},f)\|_{p,r,q}
      +\|(\frac{g}{\sqrt{T}},f)\|_{p,r}.
      \end{align*}
      By definition of $q$, we may use \eqref{Num:706} on the first r.~h.~s.~term and so obtain \eqref{Ann_ee2*}, recalling that $T\ge 1$.
    \end{proof}

    \subsection{Proof of the non-perturbative CZ estimates}\label{Sec:ProofPropMassCZ}
    This section contains the proofs of Propositions \ref{PropmassCZ} and \ref{PropCZ}(ii).

    \begin{proof}[Proof of Proposition \ref{PropmassCZ}]
      As announced in Section \ref{Sec:Strat}, the proof is divided into four parts.
      
      \parag{Part 1: Splitting of $(u,\nabla u)$}
      This part is at the core of our argument.
      Let  $1 \leq \tau \leq T$.
      We recall the splitting $(u,\nabla u)=(v_>,h_>)+(v_<,h_<)$ from \eqref{Ann_ae12}.
      Then, for exponents $p, r \in (1,\infty)$,
      \begin{align}\label{Ann_orderExp}
	    &1 \leq s'''' < s''' < s''  < s' <r, \et \quad 1\leq r''<r'<r,
      \end{align}
      we claim that the following estimates hold:
      \begin{align}
	\label{Ann_ae141}
	\big\|\big(\frac{v_>}{\sqrt{T}}, h_>\big)\big\|_{p,r''} 
	&\lesssim \barC_{p,r',r}(\tau) 
	\big\|\big(\frac{g}{\sqrt{T}} ,f\big)\big\|_{p,r},
	\\
	\big\|\big( \frac{v_<}{\sqrt{T}}, h_<\big)\big\|_{p,s''''} 
	&\lesssim \barC_{p,s''',s''}(\tau) \sqrt{\tau}^{-\frac{1}{2}} \barC_{p,s',r}(T) \big\|\big(\frac{g}{\sqrt{T}} ,f\big)\big\|_{p,r}.
	\label{Ann_ae17}
      \end{align}
      
      We note that if all the stochastic exponents in \eqref{Ann_ae141} and \eqref{Ann_ae17} were equal to $r$, then we could deduce from \eqref{Ann_ae141} and \eqref{Ann_ae17} that
      \begin{align*}
	\big\|\big( \frac{u}{\sqrt{T}},\nabla u\big)\big\|_{p,r} \lesssim \big(\barC_{p,r,r}(\tau) +\barC_{p,r,r}(\tau) \sqrt{\tau}^{-\frac{1}{2}} \barC_{p,r,r}(T)\big) \big\|\big(\frac{g}{\sqrt{T}} ,f\big)\big\|_{p,r},
      \end{align*}
      which amounts to
      \begin{align}\label{Ann_ee17}
	\barC_{p,r,r}(T) \lesssim \barC_{p,r,r}(\tau) +\barC_{p,r,r}(\tau) \sqrt{\tau}^{-\frac{1}{2}} \barC_{p,r,r}(T).
      \end{align}
      Thus, if we would know that $\barC_{p,r,r}(\tau) \lesssim \sqrt{\tau}^{\frac{1}{2} \theta}$ for some $\theta<1$, then we would obtain from the above estimate that $\barC_{p,r,r}(T)$ is uniformly bounded in $T \geq 1$.
      (This motivates assumption \eqref{Ann_ee7} of Part 2.)
      However, the stochastic exponents have to be strictly ordered as in \eqref{Ann_orderExp}.
      Therefore, in Part 2, we resort to real interpolation to increase the stochastic exponent of $\nabla u$, and we buckle with an estimate similar to \eqref{Ann_ee17}.

      \parag{Part 1, Step 1: Argument for \eqref{Ann_ae141}}
      We first estimate the high-pass part $(v_>,h_>)$. 
      The right-hand sides of \eqref{Ann_ae10_bis} and  \eqref{Ann_ae11_bis} may be expressed in a more convenient form thanks to \eqref{Def_u_tau}, \eqref{Ann_ae6}, and \eqref{Ann_ae7}, namely
      \begin{align}
	\frac{1}{T} \ubar_> - \nabla \cdot \abar_\tau \nabla \ubar_>
	&=\big(1-\frac{\tau}{T} \big) \big( \nabla \cdot ( a \nabla u_> + f) + \frac{1}{T}g\big),
	\label{Ann_ae10}
	\\
	\label{Ann_ae11}
	\frac{1}{T} \ubar_< - \nabla \cdot \abar_\tau \nabla \ubar_<
	&=\big(1-\frac{\tau}{T} \big)
	\nabla \cdot\big(a \nabla w +( \phi_{\tau,i} a  -\sigma_{\tau,i}) \nabla \partial_i \ubar  +\frac{1}{\tau} \partial_i \ubar \psi_{\tau,i}  \big).
      \end{align}
      By equation \eqref{Def_u_tau} (recall that $\tau \leq T$), we obtain
      \begin{equation}
        \label{Num:902}
        \big\|\big(\frac{u_>}{\sqrt{\tau}},\nabla u_>\big) \big\|_{p,r'} 
	  \overset{\eqref{Ann_ee5}}{\leq} 
	  \barC_{p,r',r}(\tau) \|(\frac{\sqrt{\tau} g}{T} ,f)\|_{p,r}
	  \leq 
	  \barC_{p,r',r}(\tau) \big\|\big(\frac{g}{\sqrt{T}} ,f\big)\big\|_{p,r},
      \end{equation} 
      which by \eqref{Num:03_bis} implies
      \begin{equation}
	  \big\|\big(\frac{u_>}{\sqrt{T}},\nabla u_>\big)\big\|_{p,r''}
	  \lesssim
	  \barC_{p,r',r}(\tau) \big\|\big(\frac{g}{\sqrt{T}} ,f\big)\big\|_{p,r}.
      \label{Ann_ae13}
      \end{equation}
      By estimate \eqref{Ann_ap56} on the massive correctors and recalling \eqref{Ann_ae10}, we have
      \begin{equation}
        \begin{aligned}
          \big\|\big(\frac{\ubar_>}{\sqrt{T}}, \partial_i \ubar_> (e_i + \nabla \phi_{\tau,i})\big)\big\|_{p,r''} 
	  &\overset{\eqref{Ann_ap56}}{\lesssim}
	  \big\|\big(\frac{\ubar_>}{\sqrt{T}},\nabla \ubar_>\big)\big\|_{p,r'}
	  \\
	  &\overset{\eqref{Ann_eee23}}{\lesssim}
	  \big\|\big(\frac{g}{\sqrt{T}},f,\nabla u_>\big)\big\|_{p,r'}
	  \\
	  &\overset{\eqref{Num:902}}{\lesssim} \barC_{p,r',r}(\tau) \big\|\big(\frac{g}{\sqrt{T}} ,f\big)\big\|_{p,r}.
        \end{aligned}
	\label{Ann_ae14}        
      \end{equation}
      In view of \eqref{Ann_ae12}, estimates \eqref{Ann_ae13} and \eqref{Ann_ae14} combine to \eqref{Ann_ae141}.

      \parag{Part 1, Step 2: Argument for \eqref{Ann_ae17}}
      We now turn to the low-pass contribution $(v_<,h_<)$. 
      We start with its first constituent $w$ (\textit{cf.} \eqref{Ann_ae12}). Appealing to equation \eqref{Ann_ae7} and the correctors estimate \eqref{Ann_ap57}, we obtain (recall that $\tau \leq T$)
      \begin{equation}
      \begin{aligned}
	\big\|\big(\frac{w}{\sqrt{T}},\nabla w\big)\big\|_{p,s'''}
	\leq~~& \|(\frac{w}{\sqrt{\tau}},\nabla w)\|_{p,s'''} 
	\\
	\overset{\eqref{Ann_ee5}}{\lesssim}~&
	\barC_{p,s''',s''}(\tau) 
	\Big\|\Big( \big|(\frac{\psi_{\tau}}{\sqrt{\tau}},\phi_\tau)\big| \frac{\nabla \ubar }{\sqrt{\tau}},|(\phi_{\tau},\sigma_{\tau})|\nabla^2 \ubar \Big)\Big\|_{p,s''}
	\\
	\overset{\eqref{Ann_ap57}}{\lesssim}~& \barC_{p,s''',s''}(\tau) \mu_d(\sqrt{\tau})
	\big\|\big(\frac{\nabla \ubar }{\sqrt{\tau}},\nabla^2 \ubar \big)\big\|_{p,s'}.
      \end{aligned}
      \label{Ann_ae15}
      \end{equation} 
      Turning to the second constituent $\ubar_<$ of $(v_<,h_<)$, we obtain once more by \eqref{Ann_ap56}:
      \begin{equation*}
	\| \partial_i \ubar_<(e_i +\nabla \phi_{\tau,i}) \|_{p,s''''} 
	\lesssim \| \nabla \ubar_<\|_{p,s'''}.
      \end{equation*}
      Therefore, recalling equation \eqref{Ann_ae11} and the correctors estimate \eqref{Ann_ap57}, we thus have
      \begin{align}
	\begin{aligned}
	\big\|\big(\frac{\ubar_<}{\sqrt{T}}, \partial_i \ubar_< (e_i + \nabla \phi_{\tau,i})\big)\big\|_{p,s''''} 
	&\overset{\eqref{Ann_eee23}}{\lesssim} \Big\| \big(\nabla w ,\frac{|\psi_{\tau}|}{\sqrt{\tau}}\frac{\nabla \ubar }{\sqrt{\tau}},|( \phi_{\tau},\sigma_{\tau})| \nabla^2 \ubar\big) \Big\|_{p,s'''}
	\\
	&\overset{\eqref{Ann_ap57}}{\lesssim} \| \nabla w \|_{p,s'''} +\mu_d(\sqrt{\tau}) \big\|\big(\frac{\nabla \ubar }{\sqrt{\tau}},\nabla^2 \ubar \big)\big\|_{p,s''}
	\\
	&\overset{\eqref{Ann_ae15}}{\lesssim} 	\barC_{p,s''',s''}(\tau) \mu_d(\sqrt{\tau})
	\big\|\big(\frac{\nabla \ubar }{\sqrt{\tau}},\nabla^2 \ubar \big)\big\|_{p,s'}.
	\end{aligned}
	\label{Ann_ae16}
      \end{align}
      Last, the third constituent of $(v_<,h_<)$ is estimated by appealing to the correctors estimate \eqref{Ann_ap57}:
      \begin{equation}
        \label{Ann_ae16_1}
        \big\|\big(\frac{\phi_{\tau,i} \partial_i \ubar }{\sqrt{T}},\phi_{\tau,i} \nabla \partial_i \ubar \big)\big\|_{p,s''''} 
        \overset{\tau \leq T}{\lesssim} \mu_d(\sqrt{\tau}) \big\|\big(\frac{\nabla \ubar }{\sqrt{\tau}},\nabla^2 \ubar \big)\big\|_{p,s'}.
      \end{equation} 
      By definition \eqref{Ann_ae12} of $(v_<,h_<)$, gathering \eqref{Ann_ae15}, \eqref{Ann_ae16}, and \eqref{Ann_ae16_1} yields
      \begin{align*}
	\big\|\big(\frac{v_<}{\sqrt{T}},h_<\big)\big\|_{p,s''''}
	&\lesssim \barC_{p,s''',s''}(\tau) \mu_d(\sqrt{\tau})
	\big\|\big(\frac{\nabla \ubar }{\sqrt{\tau}},\nabla^2 \ubar \big)\big\|_{p,s'}
	\\
	&\lesssim \barC_{p,s''',s''}(\tau) \frac{\mu_d(\sqrt{\tau})}{\sqrt{\tau}} \| \nabla u\|_{p,s'},
      \end{align*}
      where we used the annealed estimate \eqref{Ann_eee23} for the massive equation with constant coefficients \eqref{Ann_ae5} in its differentiated form:
      \begin{equation*}
        \frac{1}{\tau} \partial_i \ubar  - \nabla \cdot \abar_\tau \nabla \partial_i \ubar = \big(\frac{1}{\tau}-\frac{1}{T}\big)\partial_iu.
      \end{equation*}
      Finally, using \eqref{Ann_ee5} and recalling \eqref{Defmu} establishes \eqref{Ann_ae17}.

    \parag{Part 2: Real interpolation, from sub-optimal to optimal estimates}      
      Let $p$, $r$, $s''''$, $s'''$, $s''$, $s'$, $r''$, $r'$ be as in Part 1 (\textit{i.e.} satisfying \eqref{Ann_orderExp}).
      Assume that there exists $\theta \in (0,1)$ such that
      \begin{align}\label{Ann_ap43}
	\frac{1}{s'}\;>\frac{1}{s}:=\;\theta\frac{1}{s''''}+(1-\theta)\frac{1}{r''}.
      \end{align}
      (See Figure \ref{Figure_indices}.)
      Suppose that we sub-optimally control the massive CZ constants 
      both for high $(r',r)$ and low $(s''',s'')$ stochastic integrability, in the sense of
      \begin{align}
	\barC_{p,r',r}(\tau)
	+\barC_{p,s''',s''}(\tau)
	\le\Lambda\sqrt{\tau}^{\frac{1}{2}\theta}&\pourtout \tau\ge 1
	\label{Ann_ee7}
      \end{align}
      for some fixed constant $\Lambda\ge 1$.
      We claim the following control of the CZ constant:
      \begin{align}\label{Ann_ee8}
	\barC_{p,s',r}(T)\lesssim \Lambda^{\frac{1}{1-\theta}} &\pourtout  T \geq 1.
      \end{align}
      
      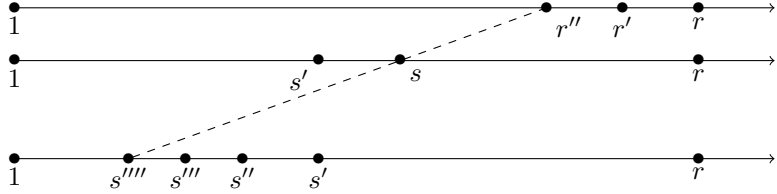
\begin{figure}[h]
	\begin{tikzpicture}
      \def\yS{0};
      \def\yT{1.3};
      \def\yR{2};
    
      \coordinate (1S) at (0,\yS);
      \coordinate (1T) at (0,\yT);
      \coordinate (1R) at (0,\yR);
      \coordinate (2S) at (10,\yS);
      \coordinate (2T) at (10,\yT);
      \coordinate (2R) at (10,\yR);
      \coordinate (S4) at (1.5,\yS);
      \coordinate (S3) at (2.25,\yS);
      \coordinate (S2) at (3,\yS);
      \coordinate (S1) at (4,\yS);
      \coordinate (S0) at (9,\yS);
      \coordinate (T1) at (4,\yT);
      \coordinate (T2) at (5.075,\yT);
      \coordinate (T0) at (9,\yT);
      \coordinate (R2) at (7,\yR);
      \coordinate (R1) at (8,\yR);
      \coordinate (R0) at (9,\yR);
      
      \draw[->] (1S)--(2S);
      \draw[->] (1T)--(2T);
      \draw[->] (1R)--(2R);
      
      \draw (1S) node {$\bullet$};
      \draw (1S) node [below] {$1$};
      \draw (1T) node {$\bullet$};
      \draw (1T) node [below] {$1$};
      \draw (1R) node {$\bullet$};
      \draw (1R) node [below] {$1$};
      
      \draw (S4) node {$\bullet$};
      \draw (S4) node [below] {$s''''$};
      \draw (S3) node {$\bullet$};
      \draw (S3) node [below] {$s'''$};
      \draw (S2) node {$\bullet$};
      \draw (S2) node [below] {$s''$};
      \draw (S1) node {$\bullet$};
      \draw (S1) node [below] {$s'$};
      \draw (S0) node {$\bullet$};
      \draw (S0) node [below] {$r$};
      
      \draw (R2) node {$\bullet$};
      \draw (R2) node [below right] {$r''$};
      \draw (R1) node {$\bullet$};
      \draw (R1) node [below] {$r'$};
      \draw (R0) node {$\bullet$};
      \draw (R0) node [below] {$r$};

      \draw (T2) node {$\bullet$};
      \draw (T2) node [below right] {$s$};
      \draw (T1) node {$\bullet$};
      \draw (T1) node [below left] {$s'$};
      \draw (T0) node {$\bullet$};
      \draw (T0) node [below] {$r$};
      
      \draw [dashed] (R2)--(S4);
      
    \end{tikzpicture}
	\caption{Exponents $p, r, s, s'''', s''', s'', s', r'', r'$. \label{Figure_indices}}
      \end{figure} 
      
      The form of the result arises from a real interpolation argument.
      We make use of the $K$-method \cite[Chap.\ 3, p.\ 38]{BerghLofstrom}, which relies on the one-parameter family of splittings  of $(u,\nabla u)$ studied in the Part 1 into a part $(v_>,h_>)$ with good stochastic integrability (slightly worse than $r'$) and a part $(v_<,h_<)$ with a bad one (slightly worse than $s'''$).
      Here the length scale $\sqrt{\tau}$ plays the role of the parameter.

      
      \parag{Part 2, Step 1: Argument for  \eqref{Ann_ee8}}
      We now fix $T\geq 1$.
      By the $K$-method, in view of \eqref{Ann_ap43}, we have
      \begin{equation}\label{Ann_interpolv}
	\begin{aligned}
	  &\big\|\big(\frac{u}{\sqrt{T}},\nabla u \big)\big\|_{p,s'} 
	  \\
	  &\quad \lesssim \sup_{\mu>0} 
	  \inf_{\scriptsize{\begin{array}{c} v_>+v_<=u \\ h_>+h_<=\nabla u \end{array}}}
	  \Big\{\mu^{-\theta} \big\|\big(\frac{v_>}{\sqrt{T}},h_> \big)\big\|_{p,r''} + \mu^{1-\theta} \big\|\big(\frac{v_<}{\sqrt{T}},h_< \big)\big\|_{p,s''''} \Big\},
	\end{aligned}
      \end{equation}
      (indeed, by \cite[Th.\ 5.1.2 \& Th.\ 5.2.1]{BerghLofstrom} the above r.~h.~s.\ corresponds to a norm on the functional space $\LL^p_{\R^d}(\LL^{s,\infty}_{\langl\cdot\rangl})$, and the Lorentz space $\LL^{s,\infty}_{\langl\cdot\rangl}$ is included in $\LL_{\langl\cdot\rangl}^{s'}$).
      Next, we will show in Part 2, Step 2 that for any $\mu>0$ and $T \geq 1$ there is a splitting $(u,\nabla u)=(v_>,h_>)+(v_<,h_<)$ such that 
      \begin{equation}\label{Ann_anycase}
	\begin{aligned} 
	  \mu^{-\theta} \big\|\big(\frac{v_>}{\sqrt{T}},h_>\big)\big\|_{p,r''} + \mu^{1-\theta} \big\|\big(\frac{v_<}{\sqrt{T}},h_<\big)\big\|_{p,s''''}
	  \lesssim \Lambda \barC_{p,s',r}^\theta(T) \big\|\big(\frac{g}{\sqrt{T}},f\big)\big\|_{p,r}.
	\end{aligned}
      \end{equation}
      Then, inserting \eqref{Ann_anycase} into \eqref{Ann_interpolv} entails
      \begin{equation*}
        \barC_{p,s',r}(T) \lesssim \Lambda \barC^{\theta}_{p,s',r}(T),
      \end{equation*}
      and since $\barC_{p,s',r}(T)$ is finite by Lemma \ref{Lem_SsOpti}, we obtain \eqref{Ann_ee8}.
      
      \parag{Part 2, Step 2: Argument for \eqref{Ann_anycase}}
      For given $\mu>0$, we define $\tau$ by
      \begin{align*}
	\mu=\frac{\sqrt{\tau}^{\frac{1}{2}}}{\barC_{p,s',r}(T)}.
      \end{align*}
      Thus we have to estimate
      \begin{equation}
        \begin{aligned}
          &\mu^{-\theta} \big\|\big(\frac{v_>}{\sqrt{T}},h_>\big)\big\|_{p,r''} + \mu^{1-\theta} \big\|\big(\frac{v_<}{\sqrt{T}},h_<\big)\big\|_{p,s''''}
	  \\
	  &\quad= \barC_{p,s',r}^\theta(T) \Big( \sqrt{\tau}^{-\frac{1}{2}\theta}\big\|\big(\frac{v_>}{\sqrt{T}},h_>\big)\big\|_{p,r''} + \frac{\sqrt{\tau}^{\frac{1}{2}(1-\theta)}}{\barC_{p,s',r}(T)} \big\|\big(\frac{v_<}{\sqrt{T}},h_<\big)\big\|_{p,s''''}\Big),
        \end{aligned}
        \label{Ann_ee18}
      \end{equation}
      and distinguish the three cases: either $\tau\leq 1$, $1 \leq \tau \leq T$, or $T \leq \tau$.
      
      If $\tau \leq 1$, we set
	$(v_>,h_>) = (0,0)$ and $(v_<,h_<)=(u,\nabla u)$
      so that by equation \eqref{Ann_ae1}
      \begin{equation*}
	\begin{aligned}
	\big\|\big(\frac{v_<}{\sqrt{T}},h_<\big)\big\|_{p,s''''} 
	\overset{\eqref{Num:03_bis}}{\leq} \big\|\big(\frac{u}{\sqrt{T}},\nabla u\big) \big\|_{p,s'}
	\overset{\eqref{Ann_ee5}}{\lesssim} \barC_{p,s',r}(T) \big\|\big(\frac{g}{\sqrt{T}},f\big)\big\|_{p,r}.
	\end{aligned}
      \end{equation*}
      Inserting this estimate into \eqref{Ann_ee18} establishes \eqref{Ann_anycase} since $\tau\leq 1 \leq \Lambda$.
      
      In the generic case $1 \leq \tau \leq T$, we define $(v_>,h_>)$ and $(v_<,h_<)$ by \eqref{Ann_ae12}.
      Therefore, \eqref{Ann_ae141} and \eqref{Ann_ae17} combined with our assumption \eqref{Ann_ee7} yield
      \begin{equation*}
	\left\{
	\begin{aligned}
	  \big\|\big(\frac{v_>}{\sqrt{T}},h_>\big)\big\|_{p,r''} 
	  &\lesssim
	  \Lambda \sqrt{\tau}^{\frac{1}{2}\theta}\big\|\big(\frac{g}{\sqrt{T}} ,f\big)\big\|_{p,r},
	  \\
	  \big\|\big(\frac{v_<}{\sqrt{T}},h_<\big)\big\|_{p,s''''} 
	  &\lesssim
	  \Lambda\sqrt{\tau}^{\frac{1}{2}(\theta-1)} \barC_{p,s',r}(T) \big\|\big(\frac{g}{\sqrt{T}} ,f\big)\big\|_{p,r}.
	\end{aligned}
	\right.
      \end{equation*}
      Inserting this into \eqref{Ann_ee18} entails \eqref{Ann_anycase}.
      
      Finally, if $\tau \geq T$ we set $(v_>,h_>)=(u,\nabla u)$ and $(v_<,h_<)=(0,0)$ to the effect of
      \begin{align*}
	\big\|\big(\frac{v_>}{\sqrt{T}},h_>\big)\big\|_{p,r''} \leq \|\big(\frac{u}{\sqrt{T}},\nabla u\big)\|_{p,r'} \lesssim \barC_{p,r',r}(T) \|\big(\frac{g}{\sqrt{T}},f\big)\|_{p,r}.
      \end{align*}
      Then, invoking \eqref{Ann_ee7} (with $\tau$ replaced by $T$) and recalling that $T \leq \tau$, we obtain
      \begin{align*}
	\big\|\big(\frac{v_>}{\sqrt{T}},h_>\big)\big\|_{p,r''}
	\lesssim \Lambda \sqrt{T}^{\frac{1}{2} \theta} \|\big(\frac{g}{\sqrt{T}},f\big)\|_{p,r} 
	\leq \Lambda \sqrt{\tau}^{\frac{1}{2} \theta} \|\big(\frac{g}{\sqrt{T}},f\big)\|_{p,r}.
      \end{align*}
      Inserting this into \eqref{Ann_ee18} and yields \eqref{Ann_anycase} also in this case.

      \parag{Part 3: Complex interpolation, from optimal estimates to suboptimal ones for higher integrability}
	We consider two triplets  of exponents $(p,r',r)$ and $(q,s',s) \in (1,\infty)\times [2,\infty)^2$ that satisfy
	\begin{align}\label{Ann_ap94_ter}
	s'\le r',\quad s\le r,\quad \et \quad
	\frac{1}{r'}-\frac{1}{r}>\frac{1}{s'}-\frac{1}{s}>0.
	\end{align}
	For any $\theta\in(0,1]$ such that
	\begin{align}\label{Ann_ap94_quad}
	\theta>\max\Big\{1-\frac{s}{r},1-\frac{s'}{r'},1-\frac{q}{p},
	\frac{\frac{1}{p}-\frac{1}{q}}{1-\frac{1}{q}}\Big\}
	\end{align}
	we claim that
	\begin{align}\label{Ann_ap90bis}
	\barC_{p,r',r}(\tau)\lesssim \barC_{q,s',s}^{1-\theta}(\tau)\sqrt{\tau}^{d\theta}
	\pourtout \tau\ge 1.
	\end{align}
	Note that the r.~h.~s.~of \eqref{Ann_ap94_quad} is strictly less than $1$ so that such a $\theta$ always exists.

	Here comes the argument: We first note that thanks to the constraint \eqref{Ann_ap94_quad} on $\theta$, the identities
	\begin{align}\label{fs01}
	\frac{1}{r}=(1-\theta)\frac{1}{s}+\theta\frac{1}{\tilde r},\quad
	\frac{1}{r'}=(1-\theta)\frac{1}{s'}+\theta\frac{1}{\tilde r'},\quad
	\frac{1}{p}=(1-\theta)\frac{1}{q}+\theta\frac{1}{\tilde p}
	\end{align}
	define a triple of exponents $(\tilde p,\tilde r',\tilde r)\in(1,\infty)\times[2,\infty)^2$.
	Indeed, in case of $\tilde r$, we rewrite the implicit \eqref{fs01} as the explicit $\frac{1}{\tilde r}$ $=\frac{1}{\theta}(\frac{1}{r}-(1-\theta)\frac{1}{s})$, and note that this expression is positive thanks to the first constraint on $\theta$ in \eqref{Ann_ap94_quad}, rewritten as $1-\theta<\frac{s}{r}$. 
	We also have $\frac{1}{\theta}(\frac{1}{r}-(1-\theta)\frac{1}{s})$ $\le\frac{1}{2}$, or equivalently, $\frac{1}{s}-\frac{1}{r}$ $\ge\theta(\frac{1}{s}-\frac{1}{2})$ because of our assumption $2\le s\le r$, see \eqref{Ann_ap94_ter}. 
	The case of $\tilde r'$ is treated analogously. 
	Finally, in the case of $\tilde p$, the lower bound $\frac{1}{\tilde p}$ $=\frac{1}{\theta}(\frac{1}{p}-(1-\theta)\frac{1}{q})>0$ follows as for the two others. 
	The upper bound $\frac{1}{\theta}(\frac{1}{p}-(1-\theta)\frac{1}{q})$ $<1$, which is equivalent to $\frac{1}{p}-\frac{1}{q}<\theta(1-\frac{1}{q})$, follows from the last of the four constraints in \eqref{Ann_ap94_quad}.

	Note that by the third item of \eqref{Ann_ap94_ter} we have $\tilde{r}'< \tilde{r}$, so that by Lemma \ref{Lem_SsOpti} there holds $\barC_{\tilde{p},\tilde{r}',\tilde{r}}(\tau) \lesssim \sqrt{\tau}^d$.
	On the other hand, by complex interpolation, there holds $\barC_{p,r',r}(\tau)$ $\le \barC_{q,s',s}^{1-\theta}(\tau)$ $ \barC_{\bar p,\tilde{r}',\tilde{r}}^{\theta}(\tau)$.
	This yields \eqref{Ann_ap90bis}.
	
	\begin{figure}[h]
	  \begin{center}
	  \begin{tikzpicture}
	    \newcommand{\rayon}{1};
	    \newcommand{\prof}{0};
	  
	    \coordinate (A0) at (0*\rayon,\prof);
	    \coordinate (A1) at (10*\rayon,\prof +0);
	    \coordinate (A2) at (1*\rayon,\prof +0);
	    \coordinate (A3) at (4*\rayon,\prof +0);
	    \coordinate (A4) at (6*\rayon,\prof +0);
	    
	    \coordinate (A11) at (2.5*\rayon,\prof -0.5);
	    \coordinate (A12) at (5*\rayon,\prof -0.5);
	    
	    \coordinate (A21) at (1*\rayon,\prof -0.25);
	    \coordinate (A22) at (4*\rayon,\prof -0.25);
	    \coordinate (A23) at (6*\rayon,\prof -0.25);
	    
	    \draw (A0)--(A1);
	    \foreach \j in {0,1,...,4}
	    {
	      \draw (A\j) node {$\bullet$};
	    }
	    \draw (A0) node[above] {$0$};
	    \draw (A1) node[above] {$1$};
	    \draw (A2) node[above] {$\frac{1}{q}$};
	    \draw (A3) node[above] {$\frac{1}{p}$};
	    \draw (A4) node[above] {$\frac{1}{\tilde{p}}$};
	    \draw (A11) node {$\theta$};
	    \draw (A12) node {$1-\theta$};
	    
	    \draw[<->] (A21)--(A22);
	    \draw[<->] (A22)--(A23);

	    \renewcommand{\rayon}{0.5};
	    \renewcommand{\prof}{-1.5};
	  
	    \coordinate (A0) at (0*\rayon,\prof);
	    \coordinate (A1) at (10*\rayon,\prof +0);
	    \coordinate (A2) at (4*\rayon,\prof +0);
	    \coordinate (A3) at (6*\rayon,\prof +0);
	    \coordinate (A4) at (9*\rayon,\prof +0);
	    
	    \coordinate (A11) at (5*\rayon,\prof -0.5);
	    \coordinate (A12) at (7.5*\rayon,\prof -0.5);
	    
	    \coordinate (A21) at (4*\rayon,\prof -0.25);
	    \coordinate (A22) at (6*\rayon,\prof -0.25);
	    \coordinate (A23) at (9*\rayon,\prof -0.25);
	    
	    \draw (A0)--(A1);
	    \foreach \j in {0,1,...,4}
	    {
	      \draw (A\j) node {$\bullet$};
	    }
	    \draw (A0) node[above] {$0$};
	    \draw (A1) node[above] {$\frac{1}{2}$};
	    \draw (A2) node[above] {$\frac{1}{\tilde{r}'}$};
	    \draw (A3) node[above] {$\frac{1}{r'}$};
	    \draw (A4) node[above] {$\frac{1}{s'}$};
	    \draw (A11) node {$1-\theta$};
	    \draw (A12) node {$\theta$};
	    
	    \draw[<->] (A21)--(A22);
	    \draw[<->] (A22)--(A23);
	    
	    \renewcommand{\rayon}{0.5};
	    \renewcommand{\prof}{-3.};
	  
	    \coordinate (A0) at (0*\rayon,\prof);
	    \coordinate (A1) at (10*\rayon,\prof +0);
	    \coordinate (A2) at (3*\rayon,\prof +0);
	    \coordinate (A3) at (5*\rayon,\prof +0);
	    \coordinate (A4) at (8*\rayon,\prof +0);
	    
	    \coordinate (A11) at (4*\rayon,\prof -0.5);
	    \coordinate (A12) at (6.5*\rayon,\prof -0.5);
	    
	    \coordinate (A21) at (3*\rayon,\prof -0.25);
	    \coordinate (A22) at (5*\rayon,\prof -0.25);
	    \coordinate (A23) at (8*\rayon,\prof -0.25);
	    
	    \draw (A0)--(A1);
	    \foreach \j in {0,1,...,4}
	    {
	      \draw (A\j) node {$\bullet$};
	    }
	    \draw (A0) node[above] {$0$};
	    \draw (A1) node[above] {$\frac{1}{2}$};
	    \draw (A2) node[above] {$\frac{1}{\tilde{r}}$};
	    \draw (A3) node[above] {$\frac{1}{r}$};
	    \draw (A4) node[above] {$\frac{1}{s}$};
	    \draw (A11) node {$1-\theta$};
	    \draw (A12) node {$\theta$};
	    
	    \draw[<->] (A21)--(A22);
	    \draw[<->] (A22)--(A23);
	  \end{tikzpicture}
	  \end{center}
	  \caption{Exponents $p, \tilde{p}, q, r, r', \tilde{r}, \tilde{r}', s, s'$ in the case $p \leq q$.}\label{Fig:Expo}
	\end{figure}
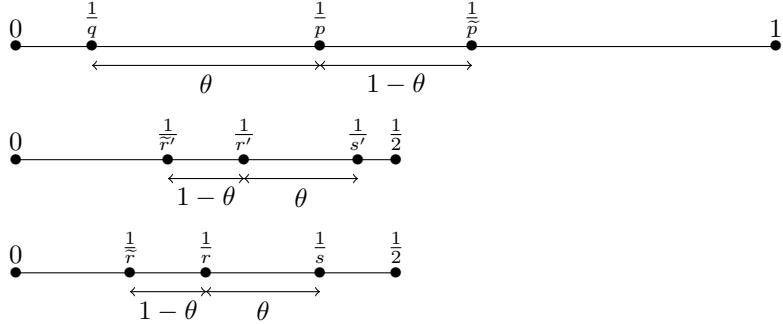

      \parag{Part 4: Conclusion, \textit{i.e.} proof of \eqref{Borne_C}}
	Let us start by rewriting \eqref{Ann_ap94_quad} in terms of $\theta\leadsto\frac{\theta}{2d}$ in the more iterable (but equivalent) form of
	\begin{align}\label{fs02}
	  \begin{array}{c}
	    (1-\frac{\theta}{2d})\frac{2}{s}<\frac{2}{r},\quad
	    (1-\frac{\theta}{2d})\frac{2}{s'}<\frac{2}{r'},\quad
	    (1-\frac{\theta}{2d})\frac{2}{q}<\frac{2}{p},\\[1ex]
	    \mbox{and}\quad
	    (1-\frac{\theta}{2d})(2-\frac{2}{q})<2-\frac{2}{p}.
	  \end{array}
	\end{align}
	In a first step, we apply Part 3 with $(q,s',s)=(2,2,2)$, feeding in the pivotal estimate \eqref{Ann_ap25}, so that \eqref{Ann_ap90bis} assumes the form
	\begin{align}\label{Ann_ap91}
	  \barC_{p,r',r}(\tau)\lesssim\sqrt{\tau}^{\frac{1}{2}\theta}\pourtout \tau\ge 1.
	\end{align}
	In view of \eqref{fs02}, such a $\theta<1$ exists provided the triplet 
	$(p,r',r)\in(1,\infty)\times[2,\infty)^2$ of exponents is sufficiently to
	$(2,2,2)$ in the sense of
	\begin{align}\label{Ann_ap92}
	  2-\big(1-\frac{1}{2d}\big)
	  =1 + \frac{1}{2d}
	  >\frac{2}{p}>1-\frac{1}{2d} \et 1 \ge \frac{2}{r'}>\frac{2}{r} >1-\frac{1}{2d},
	\end{align}
	We take this as input \eqref{Ann_ee7} yielding by \eqref{Ann_ee8} the output
	\begin{align}\label{Ann_ap93}
	\barC_{p,r',r}(\tau)\lesssim 1\pourtout \tau\ge 1
	\end{align}
	in the same range (\ref{Ann_ap92}) of exponents.

	In the second step, we feed (\ref{Ann_ap93}) with $(p,r',r)$ playing the role of $(q,s',s)$ into (\ref{Ann_ap90bis}).
	In view of \eqref{fs02}, this yields \eqref{Ann_ap91} and ultimately \eqref{Ann_ap93} in the extended range of
	\begin{align*}
	2-\big(1-\frac{1}{2d}\big)^2>\frac{2}{p} >\big(1-\frac{1}{2d}\big)^2 \et 1\geq \frac{2}{r'}>\frac{2}{r} >\big(1-\frac{1}{2d}\big)^2.
	\end{align*}
	Obviously, we may reach any arbitrary triplet $(p,r',r)$ of exponents in $(1,\infty) \times [2,\infty)^2$
	with $r'<r$ by a finite number of such steps. 
	By a duality argument and Jensen's inequality we finally obtain \eqref{Borne_C} for any triplet $(p,r',r)$ of exponents in $(1,\infty)^3$ with $1\leq r'<r$.
      \end{proof}

    \begin{proof}[Proof of Proposition \ref{PropCZ}(ii)]
      Given $f$, we consider the sequence of solutions $u_T$ to \eqref{Ann_ae1} for $g=0$.
      By Proposition \ref{PropmassCZ}, such functions have their gradients uniformly bounded in $\LL^p_{\R^d}\big(\LL^{r'}_{\langl\cdot\rangl}\big)$.
      Therefore, as $T\uparrow\infty$, there exists $\nabla u_\infty \in \LL^p_{\R^d}\big(\LL^{r'}_{\langl\cdot\rangl}\big)$ such that, for a subsequence, there holds
      \begin{align*}
        \nabla u_T \rightharpoonup \nabla u_\infty \dans \LL^{p}_{\R^d}(\LL^{r'}_{\langl \cdot \rangl}).
      \end{align*}
      By lower semi-continuity of the norm of $\LL^p_{\R^d}\big(\LL^{r'}_{\langl\cdot\rangl}\big)$, $\nabla u_\infty$ inherits \eqref{Estim13_ter} from the estimates \eqref{Ann_ee5} and \eqref{Borne_C} satisfied by $\nabla u_T$.
      Moreover, by the energy estimate, \textit{cf.} \eqref{Energ}, there holds
      \begin{align*}
        \frac{1}{T} \int  u_T^2 + \int  |\nabla u_T|^2 \lesssim \int  |f|^2,
      \end{align*}
      which implies
      \begin{align*}
        \frac{1}{T} u_T \rightarrow 0 \dans \LL^2(\R^d).
      \end{align*}
      As a consequence, the limit $\nabla u_\infty$ of the sequence $\nabla u_T$ is a weak solution to \eqref{Prop3Defv} (since $\nabla u_\infty$ is square-integrable, it is the only solution).
    \end{proof}

    \subsection{Proof of the stochastic estimates of the massive correctors}\label{Sec_LemMassCorr}
    
	The proof of Lemma \ref{LemMassCorr} follows the same scheme as the proof of Proposition \ref{Propcorr}, with minor modifications due to the presence of the massive term.
	We introduce the abbreviation
	\begin{align}\label{Num:710}
	  \Lambda_r:=\Big\langl \Big( \fint_{\Boule_1} \big| \big(\frac{\phi_T}{\sqrt{T}},\nabla \phi_T \big) \big|^2 \Big)^{r} \Big\rangl^{\frac{1}{r}}
	  \pour r \in [1,\infty).
	\end{align}
	In Step 1, we momentarily fix some deterministic function $g$ and vector field $f$ and we derive a representation formula for the Malliavin derivatives of the random variables $F$ and $F^\star$ defined by
	\begin{equation}\label{Def:F_T}
	  F:=\int \big(-\frac{1}{T}g\phi_{T,j}+f\cdot\nabla\phi_{T,j}\big)
	  \et \quad F^\star:=\int f\cdot(q_{T,j}-\langl q_{T,j}\rangl).
	\end{equation}
	(We henceforth omit the index $j$.)
	Using the spectral gap \eqref{Prop21a}, we establish in Step~2 that these satisfy the following estimate:
	\begin{align}
	  \langl |(F,F^\star)|^{2r}\rangl^\frac{1}{r}
	  \lesssim \Lambda_r \int  \big|\big(\frac{g}{\sqrt{T}}, f \big) \big|^2
	  \quad \text{provided}\; r \gg 1.
	  \label{Massive_Step1}
	\end{align}
	In Step 3, we argue that, if the stationary vector fields $u$, $f$, and $g$ are related through \eqref{Ann_ae1},
	then there exists $\theta=\theta(d,\lambda) >0$ such that, for all $R>0$, the following annealed Caccioppoli estimate holds
	\begin{equation}
	  \label{Num:730}
	  \Big\langl \Big( \fint_{\Boule_R} \big|\big( \frac{u}{\sqrt{T}},\nabla u \big)\big|^2 \Big)^r \Big\rangl
	  \lesssim
	  \Big\langl \big|\big( \frac{g}{\sqrt{T}},f \big)\big|^{2r} \Big \rangl
	  + \Big\langl  \Big| \fint_{\Boule_{\theta R}} \big( \frac{u}{\sqrt{T}},\nabla u \big) \Big|^{2r} \Big \rangl.
	\end{equation}
	Then, in Step 4, we buckle on Step 2 and Step 3 (where $\phi_T$ plays the role of $u$) to obtain
	\begin{align}
	  \label{Massive_Step3}
	  \Lambda_r \lesssim 1,
	\end{align}
	which yields that $\nabla \phi_T$ satisfies \eqref{Ann_ap56} by local regularity.
	In Step 5, we argue that \eqref{Massive_Step1} and \eqref{Massive_Step3} imply
	\begin{align}
	  \label{Ann_as14}
	  \Big\langl\Big|\int_{\Boule_1}\big(\phi_T,\sigma_T,\frac{\psi_T}{\sqrt{T}}\big)\Big|^{2r}\Big\rangl^\frac{1}{r}
	  &\lesssim \mu^2_d(\sqrt{T}),
	  \\
	  \label{Num:733}
	  \Big\langl\Big|\int_{\Boule_1}\big(\nabla \sigma_T,\frac{\nabla \psi_T}{\sqrt{T}}\big)\Big|^{2r}\Big\rangl^\frac{1}{r}
	  &\lesssim
	  1.
	\end{align}
	In Step 6, we show that also $\nabla \sigma_T$ and $\nabla \psi_T$ satisfy \eqref{Ann_ap56}.
	Last, in Step 7, we argue that \eqref{Ann_as14} and \eqref{Ann_ap56} imply \eqref{Ann_ap57}.
	
	\parag{Step 1: Representation formulas of the Malliavin derivatives}
	  The starting point is provided by the following representation formulas of the Malliavin derivatives:
	  \begin{align}
	    \label{Id_1}
	    \frac{\partial F}{\partial a}&=\nabla v\otimes(e+\nabla\phi_T),
	    \\
	    \label{Id_2}
	    \frac{\partial F^\star}{\partial a}&=(f+\nabla v^\star)\otimes(e+\nabla\phi_T),
	  \end{align}
	  where $v$ and $v^\star$ are defined through
	  \begin{align}
	    \frac{1}{T}v-\nabla\cdot a^\star\nabla v&=\frac{1}{T}g+\nabla\cdot f,
	    \label{Ann_as06}
	    \\
	    \frac{1}{T} v^\star-\nabla\cdot a^\star\nabla v^\star&= \nabla \cdot  a^\star f.
	    \nonumber
	  \end{align}
	  
	  Here comes the argument for \eqref{Id_1} (\eqref{Id_2} is justified exactly the same way).
	  Given a smooth and compactly supported infinitesimal variation $\delta a$ of $a$, we have for the generated infinitesimal variation of $F$
	  \begin{align*}
	    \delta F
	    &\overset{\eqref{Def:F_T}}{=} \int \big(-\frac{1}{T}g \delta \phi_T +f\cdot\nabla \delta \phi_T\big)
	    \overset{\eqref{Ann_as06}}{=}-\int  \frac{1}{T} v \delta \phi_T + \nabla v \cdot a \nabla \delta \phi_T.
	  \end{align*}
	  Now, differentiating the definition \eqref{Ann_as12} of $\phi_T$ entails
	  \begin{align}\label{Diff_phiT}
	    \frac{1}{T} \delta \phi_{T}-\nabla\cdot a \nabla \delta \phi_{T}=\nabla \cdot \delta a (e+\nabla \phi_T),
	  \end{align}
	  which we test against $v$, obtaining \eqref{Id_1} in form of
	  \begin{align*}
	    \delta F=\int  \nabla v \cdot \delta a (e+\nabla \phi_T).
	  \end{align*}

	\parag{Step 2: Proof of \eqref{Massive_Step1}} 
	  We now derive \eqref{Massive_Step1} from \eqref{Id_1} and \eqref{Id_2}, focusing on $F$, since the argument for $F^\star$ is similar.
	  Since $\phi_T$ is stationary, so that $\nabla \phi_T$ and $\nabla\cdot a(e+\nabla \phi_T)$ have vanishing expectations and thus also $\frac{1}{T}\phi_T$ by \eqref{Ann_as12}, the random variable $F$ is of vanishing expectation.
	  Therefore, by the spectral gap \eqref{Prop21a}, invoking a duality argument, the H\"older inequality, the stationarity of $\nabla \phi_T$, and Jensen's inequality, we obtain:
	  \begin{equation}
	    \label{Estim11}
	    \begin{aligned}
	    \langl F^{2r}\rangl^{\frac{1}{r}} 
	    \lesssim~& \sup_{\langl G^{2r^\star}\rangl \leq 1} 
	    \Big\langl \int \Big( \fint_{\Boule_1(x)} | ( G\nabla v) \otimes(e+\nabla\phi_T)| \Big)^2 \dd x \Big\rangl
	    \\
	    \leq~& \sup_{\langl G^{2r^\star}\rangl \leq 1}
	    \int  \Big\langl \Big( \fint_{\Boule_1(x)} | G\nabla v|^2 \Big)^{r^\star}
	    \Big\rangl^{\frac{1}{r^\star}}
	    \Big\langl \Big(\fint_{\Boule_1(x)}  |e+\nabla\phi_T|^2\Big)^r \Big\rangl^{\frac{1}{r}} \dd x
	    \\
	    \overset{\eqref{Num:710}}{\leq}& (\Lambda_r+1) \sup_{\langl G^{2r^\star}\rangl \leq 1}
	    \int \Big\langl \Big( \fint_{\Boule_1(x)} | G\nabla v|^2 \Big)^{r^\star}
	    \Big\rangl^{\frac{1}{r^\star}} \dd x
	    \\
	    \overset{\eqref{Num:556}}{\leq}~& (\Lambda_r+1) \sup_{\langl G^{2r^\star}\rangl \leq 1}
	    \int  \langl | G\nabla v|^{2r^\star} \rangl^{\frac{1}{r^\star}} \dd x.
	    \end{aligned}
	  \end{equation}
	  Using annealed CZ estimates on \eqref{Ann_as06} or rather on 
	  $\frac{1}{T}Gv-\nabla\cdot a^\star\nabla Gv$ $=\frac{1}{T}Gg+\nabla\cdot Gf$
	  (the perturbative proof of which, based on \eqref{Ann_eee23}, is similar to the proof of Proposition \ref{PropCZ}(i) relying on Lemma \ref{LemCstMassCZ}), provided  $r^\star-1 \ll 1$, we have
	  \begin{equation*}
	    \begin{aligned}
	      \int \langl|G\nabla v|^{2r^\star}\rangl^\frac{1}{r^\star}
	      \lesssim~& \int \Big\langl\big|G\big(\frac{g}{\sqrt{T}},f\big)\big|^{2r^\star}\Big\rangl^\frac{1}{r^\star}
	      = \langl G^{2r^\star} \rangl^{\frac{1}{r^\star}} \int \big|\big(\frac{g}{\sqrt{T}},f\big)\big|^{2}.
	    \end{aligned}
	  \end{equation*}
	  Inserting this into \eqref{Estim11} yields \eqref{Massive_Step1}.

	\parag{Step 3: Proof of \eqref{Num:730}}
	  Let $R' \leq R$ (to be fixed later).
	  We denote by $u_{R'}$  the mollification of $u$ on scale $R'$ by averaging on balls of radius $R'$ (see \eqref{Num:719}).
	  By the Caccioppoli estimate (see Lemma \ref{Lem_Caccio}), for all constants $c$, we have
	  \begin{equation*}
	    \fint_{\Boule_{R}}\big|\big(\frac{u-c}{\sqrt{T}},\nabla u\big)\big|^2\lesssim
	    \fint_{\Boule_{2R}}\big|\big(\frac{g-c}{\sqrt{T}},f,\frac{u-c}{R}\big)\big|^2.
	  \end{equation*}
	  Choosing $c=\fint_{\Boule_{2R}}u_{R'}$ and invoking \eqref{Num:739} (replacing $R \rightsquigarrow R/2$), we obtain
	  \begin{align*}
	    \fint_{\Boule_R}\big|\big(\frac{u}{\sqrt{T}},\nabla u\big)\big|^2
	    \lesssim~&
	    \fint_{\Boule_{4R}}\big|\big(\frac{g}{\sqrt{T}},f,\frac{u_{R'}}{\sqrt{T}},\nabla u_{R'}\big)\big|^2
	    +\big(\frac{R'}{R}\big)^2\fint_{\Boule_{4R}}|\nabla u|^2,
	  \end{align*}
	  of which we take the $r$-th moments:
	  \begin{align*}
	    &\Big\langl \Big(\fint_{\Boule_R}\big|\big(\frac{u}{\sqrt{T}},\nabla u\big)\big|^2\Big)^r \Big\rangl
	    \\
	    &\qquad\lesssim
	    \Big\langl \Big(\fint_{\Boule_{4R}}\big|\big(\frac{g}{\sqrt{T}},f,\frac{u_{R'}}{\sqrt{T}},\nabla u_{R'}\big)\big|^2\Big)^r \Big\rangl
	    +\big(\frac{R'}{R}\big)^{2r}\Big\langl \Big(\fint_{\Boule_{4R}}|\nabla u|^2\Big)^r \Big\rangl.
	  \end{align*}
	  By the triangle inequality and using the stationarity of $\nabla u$, which implies the estimate 
	  $\langl (\fint_{\Boule_{4R}}|\nabla u|^2)^r \rangl$ 
	  $\lesssim \langl (\fint_{\Boule_{R}}|\nabla u|^2)^r  \rangl$,
	  \textit{cf.} \eqref{Num:771}, and choosing ${R'}=\theta R$ with $\theta \ll 1$, we may absorb the second r.~h.~s.\ term into the l.~h.~s.\ :
	  \begin{align*}
	    \Big\langl \Big(\fint_{\Boule_R}\big|\big(\frac{u}{\sqrt{T}},\nabla u\big)\big|^2\Big)^r \Big\rangl
	    \lesssim~&
	    \Big\langl \Big(\fint_{\Boule_{4R}}\big|\big(\frac{g}{\sqrt{T}},f,\frac{u_{\theta R}}{\sqrt{T}},\nabla u_{\theta R}\big)\big|^2\Big)^r \Big\rangl.
	  \end{align*}
	  Finally, by Jensen's inequality and the stationarity of $u$, $f$, and $g$, we get \eqref{Num:730}.
	  
	  \parag{Step 4: Proof of \eqref{Massive_Step3} and \eqref{Ann_ap56} for $\nabla\phi_T$}
	  By the same reasoning based on the stationarity of $(\phi_T,\nabla \phi_T)$ as the one leading to \eqref{2b}, we get
	  \begin{equation*}
	    \Lambda_r
	    \lesssim
	    R^{d(1-\frac{1}{r})} \Big\langl \Big(\fint_{\Boule_{2R}} \big|\big(\frac{\phi_T}{\sqrt{T}},\nabla \phi_T\big)\big|^2 \Big)^{r}\Big\rangl^{\frac{1}{r}}
	    \pour R \gg 1.
	  \end{equation*}
	  Inserting \eqref{Num:730}, where we replace $(u,g,f) \rightsquigarrow (\phi_T, 0, a e)$, into this estimate yields
	  \begin{align}
	    \label{Estim12}
	    \Lambda_r
	      \lesssim
	      R^{d(1-\frac{1}{r})} \Big( \Big\langl  \Big| \fint_{\Boule_{\theta R}} \big( \frac{\phi_T}{\sqrt{T}},\nabla \phi_T \big) \Big|^{2r} \Big \rangl^{\frac{1}{r}}
	      +1\Big).
	  \end{align}
	  On the other hand, setting $g:=\sqrt{T}\mathds{1}_{\Boule_{R}}$ and $f:=\mathds{1}_{\Boule_{R}}e_j$ for $j \in \{1,\cdots,d\}$, we obtain from \eqref{Massive_Step1} that
	  \begin{align*}
	    \Big\langl \Big| \fint_{\Boule_{R}} \big(\frac{\phi_T}{\sqrt{T}},\nabla \phi_T \big) \Big|^{2r} \Big\rangl^{\frac{1}{r}}
	    \lesssim \Lambda_r R^{-d}.
	  \end{align*}
	  Inserting this (for $R \rightsquigarrow \theta R$) into \eqref{Estim12} yields
	  \begin{align*}
	    \Lambda_r
	    \lesssim R^{-\frac{d}{r}} \Lambda_r +R^{d(1-\frac{1}{r})} \pour R \gg 1.
	  \end{align*}
	  Hence, we may absorb the first r.~h.~s.\ term and obtain \eqref{Massive_Step3}.
	  By H\"older regularity of the coefficient field, namely by Lemma \ref{LemAppendHold} (in which we replace the fields $(u,f)$ and the exponents $(q,\alpha)$ by $(\phi_T,a\cdot e)$ and $(2,\alpha/2)$, respectively), the Hölder inequality, and Lemma \ref{LemSG}, estimate \eqref{Massive_Step3} may be upgraded to
	  \begin{equation*}
	    \begin{aligned}
	      \langl|\nabla\phi_T|^{2r}\rangl^{\frac{1}{r}}
	      &\overset{\eqref{cw18}}{\lesssim} \Big\langle  (1+ \|a \|_{\CC^{\frac{\alpha}{2}}(\Boule_1)})^{\frac{2r}{\alpha}(\frac{d}{2}+1)} \Big(1+\fint_{\Boule_1}\big| \big(\frac{\phi_T}{\sqrt{T}},\nabla \phi_T \big) \big|^2 \Big)^r \Big\rangle^{\frac{1}{r}}
	      \\
	      &\leq \big\langle (1+ \|a \|_{\CC^{\frac{\alpha}{2}}(\Boule_1)})^{\frac{4r}{\alpha}(\frac{d}{2}+1)} \big\rangle^{{\frac{\alpha}{4r(\frac{d}{2}+1)}}} \Big\langle \Big(1+\fint_{\Boule_1}\big| \big(\frac{\phi_T}{\sqrt{T}},\nabla \phi_T \big) \big|^2 \Big)^{2r} \Big\rangle^{\frac{1}{2r}}
	      \\
	      &\!\!\!\!\!\!\!\!\overset{\eqref{wg03},\eqref{Massive_Step3}}{\lesssim} 1.
	    \end{aligned}
	  \end{equation*}

	\parag{Step 5: Proof of \eqref{Ann_as14} and \eqref{Num:733}}
	  Fixing $i, j, k, l \in \{1,\cdots,d\}$, we introduce the auxiliary functions
	  \begin{align*}
	  \big(\frac{1}{T}-\Delta\big)\ubar =\mathds{1}_{\Boule_1} \et \quad \big(\frac{1}{T}-\Delta\big)\vbar =-\nabla \cdot (\mathds{1}_{\Boule_1} e_l),
	  \end{align*}
	  so that by the symmetry of $\frac{1}{T}-\Delta$ and the definitions \eqref{Ann_as12s} and \eqref{Ann_as11} of $\sigma_{T}$ and $\psi_{T}$, we obtain the representations
	  \begin{align*}
	    &\int_{\Boule_1}\phi_{T}=\int \big(\frac{1}{T}g\phi_T-f\cdot\nabla\phi_T\big)
	    &&\pour(g,f):=(\ubar ,-\nabla\ubar ),
	    \\
	    &\int_{\Boule_1}\sigma_{T,ijk}=\int  f\cdot (q_{T,i}-\langl q_{T,i}\rangl)
	    &&\pour f:=\partial_k\ubar e_j-\partial_j\ubar e_k,
	    \\
	    &\int_{\Boule_1}\psi_{T,ij}=\int  f\cdot (q_{T,i}-\langl q_{T,i}\rangl-\nabla\phi_{T,i})
	    &&\pour f:=\ubar e_j,
	    \\
	    &\int_{\Boule_1}\partial_l \sigma_{T,ijk}=\int  f\cdot (q_{T,i}-\langl q_{T,i}\rangl)
	    &&\pour f:=\partial_k\vbar e_j-\partial_j\vbar e_k,
	    \\
	    &\int_{\Boule_1}\partial_l \psi_{T,ij}=\int  f\cdot (q_{T,i}-\langl q_{T,i}\rangl-\nabla\phi_{T,i})
	    &&\pour f:=\vbar e_j.
	  \end{align*}
	  Hence (\ref{Ann_as14}) and  \eqref{Num:733} follow from \eqref{Massive_Step1} and \eqref{Massive_Step3} via the standard estimates
	  \begin{align*}
	  \int \big|\big(\frac{\ubar }{\sqrt{T}},\nabla \ubar \big)\big|^2
	  \lesssim\mu^2_d(\sqrt{T}) 
	  \et \quad 
	  \int \big|\big(\frac{\vbar }{\sqrt{T}},\nabla \vbar \big)\big|^2 \lesssim 1.
	  \end{align*}
	  
	\parag{Step 6: Proof of  \eqref{Ann_ap56} for $\nabla \sigma_T$ and $\nabla\psi_T$}
	  The argument, similar to Part 2, Step 2 of the proof of Proposition \ref{Propcorr}, is the same for $\sigma_T$ and $\psi_T$; therefore we only show \eqref{Ann_ap56} for $\nabla \sigma_T$.
	  
	  By local CZ estimate \eqref{cw15} (replacing $(p,q) \rightsquigarrow (2r,2)$) for the constant-coefficient operator $\frac{1}{T} - \Delta$ applied to the equation \eqref{Ann_as12s}, there holds
	  \begin{equation*}
	    \fint_{\Boule_R} \big|\big( \frac{\sigma_T}{\sqrt{T}}, \nabla \sigma_T \big) \big|^{2r}
	    \lesssim \fint_{\Boule_{2R}} \lt|q_T\rt|^{2r} + \Big(\fint_{\Boule_{2R}} \big|\big( \frac{\sigma_T}{\sqrt{T}}, \nabla \sigma_T \big) \big|^2\Big)^r.
	  \end{equation*}
	   By ergodicity and stationarity of $\sigma_T$ and $q_T$, when $R \uparrow \infty$, each of these spatial averages converges almost surely to the associated expectation (this is a consequence of the Birkhoff theorem).
	   Hence
	   \begin{equation}\label{Num:951}
	     \big\langle \big|\big( \frac{\sigma_T}{\sqrt{T}}, \nabla \sigma_T \big) \big|^{2r} \big\rangle
	     \lesssim \big \langle  \lt|q_T\rt|^{2r} \big \rangle
	     +\big\langle \big|\big( \frac{\sigma_T}{\sqrt{T}}, \nabla \sigma_T \big) \big|^{2} \big\rangle^r.
	   \end{equation}
	   Moreover, since the constant-coefficient operator $\frac{1}{T}-\Delta$ obviously satisfies the weighted energy estimate \eqref{Ann_ap34} (for a weight $\omega_T$ defined by \eqref{Ann_ap21} for $C \gg_d 1$), there holds
	   \begin{equation*}
	     \int \omega_T \big\langle\big|\big( \frac{\sigma_T}{\sqrt{T}}, \nabla \sigma_T \big) \big|^{2} \big \rangle \lesssim  \int \omega_T \big\langle\lt|q_T\rt|^{2}\big \rangle,
	   \end{equation*}
	   from which we deduce
	   \begin{equation*}
	     \big\langle \big|\big( \frac{\sigma_T}{\sqrt{T}}, \nabla \sigma_T \big) \big|^{2}\big\rangl \lesssim \big\langle \lt|q_T\rt|^{2}\big\rangle
	     \overset{\eqref{Def_qT},\eqref{Ann_ap56}}{\lesssim} 1.
	   \end{equation*}
	   Inserting this estimate into \eqref{Num:951} finally yields the desired estimate  \eqref{Ann_ap56} for $\nabla \sigma_T$.
	  
	\parag{Step 7: Proof of \eqref{Ann_ap57}}
	  We finally pass from \eqref{Ann_as14} to \eqref{Ann_ap57} with help of \eqref{Ann_ap56}.
	  Since the proof is similar for $\phi_T$, $\sigma_T$ and $\psi_T$, we only treat $\phi_T$.
	  Appealing to a Sobolev embedding in form of
	  \begin{align*}
	    |\phi_T(0)|^{2r} \lesssim 
	    \fint_{\Boule_1}|\nabla\phi_T|^{2r}
	    +\Big|\fint_{\Boule_1}\phi_T\Big|^{2r},
	  \end{align*}
	  which holds provided $2r>d$, and to the stationarity of $\nabla\phi_T$, we obtain the desired estimate:
	  \begin{align*}
	    \langl \phi_T^{2r}\rangl^\frac{1}{r}
	    \lesssim \langl |\nabla\phi_T|^{2r}\rangl^\frac{1}{r}
	    +\Big\langl \Big|\fint_{\Boule_1}\phi_T\Big|^{2r} \Big\rangl^{\frac{1}{r}} \overset{\eqref{Ann_ap56},\eqref{Ann_as14}}{\lesssim} \mu^2_d(\sqrt{T}).
	  \end{align*}
	  {\hfill $\qed$}

  \subsection{Alternative approach via quenched CZ estimates}\label{SecLocal} 
  
    We describe here another approach in order to obtain annealed CZ estimates; we restrict our explanation to a special case of Proposition \ref{PropCZ}(ii), namely specifying $r' \rightsquigarrow p$ in \eqref{Estim13_ter}:
    \begin{align}
      \label{Prop3E2}
      \int  \langl |\nabla u|^p \rangl
      \lesssim_{\gamma, p, r}
      \int  \langl |f|^{r} \rangl^{\frac{p}{r}} \quad\text{provided } 1< p < r < \infty.
    \end{align}
    This strategy, which is close to \cite[Prop. 6.4]{DuerinckxOtto_2019} and inspired by periodic homogenization (see the recent monograph \cite[Th.\ 4.3.1, p.\ 83]{Shen}), is quite different from the one of previous section.
    However, it involves two important concepts: 
    \begin{itemize}
      \item{the minimal radius $\rstar$, above which homogenization kicks in,}
      \item{the quenched CZ estimates (\textit{i.e.} deterministic or pathwise estimates).}
    \end{itemize}
    The very difference between the strategy exposed here and \cite[Prop. 6.4]{DuerinckxOtto_2019} is that we do not need the full strength of the Lipschitz regularity theory to get quenched estimates; on the contrary, we only rely on the weaker result Proposition \ref{PropHcon}.

    The first idea is to build a random stationary field of minimal radii $\rstar$ above which  homogenization kicks in the sense of Proposition \ref{PropHcon} (see \cite[Th. 1]{Gloria_Neukamm_Otto_2019}):
    \begin{lemma}[Minimal radius $\rstar$]
      For every $\delta >0$, there exists a random stationary field $\rstar \geq 1$ such that:
      
      (i) For all $R \geq \rstar(x)$, we have
      \begin{equation*}
	\frac{1}{R} \Big( \fint_{\Boule_R(x)} | ( \phi,\phi^\star,\sigma,\sigma^\star) -  \fint_{\Boule_R(x)} ( \phi,\phi^\star,\sigma,\sigma^\star) |^2 \Big)^{\frac{1}{2}} \leq \delta.
      \end{equation*}
      
      (ii) For all $p<\infty$ there holds
      \begin{equation*}
	\langl \rstar^p \rangl \lesssim_{\gamma, p, \delta} 1.
      \end{equation*}
      
      (iii) The function $x \mapsto \rstar(x)$ is (almost surely) $\frac{1}{2}$-Lipschitz.
    \end{lemma}
    In establishing (ii), the control of moments of $(\phi,\sigma)$ provided by Proposition \ref{Propcorr} plays a crucial role.
    
    Next, appealing to a CZ decomposition \cite[Th.\ 4.3.1, p.\ 83]{Shen}, a global \textit{quenched} $\LL^p$ estimate (see also \cite[Prop. 6.4]{DuerinckxOtto_2019}) may be obtained:
    \begin{lemma}[Quenched CZ estimates]
      For every $p \in (1,\infty)$, there exists a constant $\delta=\delta(d,\lambda,p)>0$ such that, for the associated $\rstar$ defined in Lemma \ref{Cor2}, we have:
      
      If $\nabla u$ and $f$ are square-integrable random vector fields related by
      \begin{align*}
	\nabla \cdot ( a \nabla u + f)=0,
      \end{align*}
      then the following quenched estimates hold:
      \begin{align}\label{Quench}
	\int  \Big( \fint_{\Boule_{\rstar(x)}(x)} | \nabla u|^2 \Big)^p \dd x
	\lesssim_{\gamma, p, \delta}
	\int  \Big( \fint_{\Boule_{\rstar(x)}(x)} | f|^2 \Big)^p \dd x.
      \end{align}
    \end{lemma}
    
    Finally, Proposition \ref{PropCZ}(ii) is obtained in two steps:
    \begin{itemize}
      \item{first, the quenched estimate \eqref{Quench} is upgraded to an annealed estimate thanks to the control on the minimal radius $\rstar$ (at the price of a loss in stochastic integrability),}
      \item{then, an interpolation argument between the previous estimate and \eqref{Prop3E1} establishes \eqref{Prop3E2}.
      }
    \end{itemize}

\section{Acknowledgements}

    The second author wishes to thank his long-term collaborators A.~Gloria and M.~Duerinckx, the many discussions with whom influenced this text in a way  that cannot be done justice by the bibliography. Just to mention one example: The strategy of the proof of Lemma \ref{LemMassCorr} originated from discussions with A.~Gloria and J.~Nolen. 
    F.~O.~also acknowledges the hospitality of the Institut des Math\'ematiques de Toulouse, where he gave a lecture series on Sections 1 through 5 (and Section \ref{SecLocal}) of this text.
    Both authors thank J. Sauer for his precious help concerning the theory of UMD spaces and his accurate advice for references.

  \bibliographystyle{plain}  
\def\cprime{$'$} \def\cprime{$'$} \def\cprime{$'$}

\appendix
\section{Standard regularity theory}

\subsection{Caccioppoli estimate}\label{ProofCaccioppoli}

  For the convenience of the reader, we reproduce the proof of the classical Caccioppoli estimate (see for instance \cite[Th.\ 4.4, p.\ 63]{Giaquinta} for the massless operator): \begin{lemma}\label{Lem_Caccio}
    Let $T \in [1,+\infty]$, and $R >0$.
    Assume that the coefficient field $a$ satisfies \eqref{Ellipticite}.
    If $u$ solves
    \begin{align}\label{Caccio_Eq}
      \frac{1}{T}u-\nabla \cdot a \nabla u=\frac{1}{T}g + \nabla \cdot f \dans \Boule_{R},
    \end{align}
    then, for any constant $c$, there holds
    \begin{equation}\label{Caccioppoli3}
      \fint_{\Boule_{R/2}}\big|\big(\frac{u-c}{\sqrt{T}},\nabla u\big)\big|^2\lesssim
      \fint_{\Boule_{R}}\big|\big(\frac{g-c}{\sqrt{T}},f,\frac{u-c}{R}\big)\big|^2.
    \end{equation}
  \end{lemma}
  
  \begin{proof}
    We select a smooth function $\eta$ supported in $\Boule_R$ with values in $[0,1]$ such that $\eta=1$ in $\Boule_{R/2}$ and $|\nabla \eta| \lesssim R^{-1}$.
    Note that \eqref{Caccio_Eq} is invariant under the substitution ${(u,g)\rightsquigarrow({u + c}, \break{g+c})}$.
    Hence, w.~l.~o.~g.\ we show only \eqref{Caccioppoli3} for $c=0$.    
    Testing \eqref{Caccio_Eq} with $\eta^2 u$, we get by Leibniz' rule in form of $\nabla(\eta^2 u)=\eta^2\nabla u+ 2 \eta u \nabla \eta$ that
    \begin{align*}
      &\frac{1}{T}\fint_{\Boule_R} \eta^2 u^2 + \fint_{\Boule_R} \eta^2 \nabla u \cdot a \nabla u
      \\
      &\qquad =
      -2 \fint_{\Boule_R} \eta u \nabla \eta \cdot a \nabla u
      +\frac{1}{T}\fint_{\Boule_R}\eta^2 u g 
      -\fint_{\Boule_R} \eta^2 f \cdot \nabla u - 2 \fint_{\Boule_R} \eta u f \cdot \nabla \eta.
    \end{align*}
    Using the uniform ellipticity \eqref{Ellipticite} of $a$ for the l.~h.~s.\ and the Cauchy-Schwarz inequality for the r.~h.~s.\ (recalling that $\eta \leq 1$ and $|\nabla \eta| \lesssim R^{-1}$), we obtain
    \begin{align*}
      &\fint_{\Boule_R} \big|\big( \frac{\eta u}{\sqrt{T}}, \eta \nabla u \big) \big|^2
      \\
      &\lesssim
      \Big(\fint_{\Boule_R} \big(\frac{u}{R}\big)^2\Big)^{\frac{1}{2}}
      \Big(\fint_{\Boule_R} |\eta \nabla u|^2\Big)^{\frac{1}{2}}
      +\Big(\fint_{\Boule_R} \big(\frac{g}{\sqrt{T}}\big)^2 \Big)^{\frac{1}{2}} \Big(\fint_{\Boule_R}  \big( \frac{\eta u}{\sqrt{T}}\big)^2 \Big)^{\frac{1}{2}}
      \\
      &~~~~+
      \Big(\fint_{\Boule_R} |f|^2\Big)^{\frac{1}{2}}
      \Big(\fint_{\Boule_R} |\eta \nabla u|^2\Big)^{\frac{1}{2}}
      +\Big(\fint_{\Boule_R} |f|^2\Big)^{\frac{1}{2}}
      \Big(\fint_{\Boule_R} \big(\frac{u}{R}\big)^2\Big)^{\frac{1}{2}}
      .
    \end{align*}
    By the Young inequality, we may absorb the r.~h.~s.\ integrals $\fint_{\Boule_R} |\eta \nabla u|^2$ and $\fint_{\Boule_R}  \big( \frac{\eta u}{\sqrt{T}}\big)^2$ into the l.~h.~s., obtaining \eqref{Caccioppoli3} for $c=0$ (since $\eta=1$ in $\Boule_{R/2}$).
  \end{proof}

\subsection{A Sobolev estimate}\label{ProofE}

  \begin{proof}[Proof of \eqref{wg04}]
    The basis for (\ref{wg04}) is a Sobolev embedding in the form of
    \begin{align}\label{E1}
      \sup_{\Boule_{1}(x)}\zeta^2\lesssim
      \Big(\int_{\Boule_2(x)}\sum_{n=0}|\nabla^n\zeta|^2\Big)^\frac{1}{2}
      \Big(\int_{\Boule_2(x)}\zeta^2\Big)^\frac{1}{2},
    \end{align}
    which by Cauchy-Schwarz, an identity of the type of \eqref{Num:556}, and Plancherel leads to
    \begin{align}
    \label{Estim04}
    \int  \sup_{\Boule_1(x)}\zeta^2 \dd x\lesssim
    \Big(\int(1+|k|)^{2d}|\FF \zeta|^2\dd k\Big)^\frac{1}{2}
    \Big(\int|\FF \zeta|^2\dd k\Big)^\frac{1}{2}.
    \end{align}
    Under additional conditions on the support of the Fourier transform in form of $\Supp(\FF \zeta)\subset\{1\le R|k|<2\}$, 
    this yields
    \begin{align*}
      \int \sup_{\Boule_1(x)}\zeta^2 \dd x\lesssim
      R^{2\alpha}\int(1+|k|)^{d+2\alpha}|\FF \zeta|^2\dd k
    \end{align*}
    and if $\Supp(\FF \zeta)	\subset\{|k|<1\}$, then
    \begin{align*}
      \int \sup_{\Boule_1(x)}\zeta^2 \dd x\lesssim
      \int(1+|k|)^{d+2\alpha}|\FF \zeta|^2\dd k.
    \end{align*}
    Hence (\ref{wg04}) follows from decomposing $\zeta$ in Fourier space
    accordingly (\textit{i.e.} with $R=2^{-n}$, $n=0, 1, \cdots$) 
    and using the triangle inequality on the l.~h.~s.~norm in \eqref{Estim04}.
    The sum over $n$ converges since $\alpha>0$.

    Here comes the proof of \eqref{E1}.
    Without loss of generality, we may assume that $\zeta$ has a support in $\Boule_2$ (if not, we may just multiply it by a cut-off function and appeal to the Leibniz' rule).
    Then, using the Fourier transform\footnote{see footnote \ref{fref} p.\ \thecompteur~for the normalization of the Fourier transform} and separating small and large wavelengths according to a parameter $R>0$, we obtain
    \begin{align*}
      \sup_{\Boule_1(x)}|\zeta|
      &\leq 
      \frac{1}{(2\pi)^d}\int |\FF \zeta|
      =\frac{1}{(2\pi)^d} \Big(\int_{|k| \leq R} |\FF \zeta| \dd k + \int_{|k|>R} |\FF \zeta| \dd k\Big)
      \\
      &\lesssim R^{\frac{d}{2}} \Big( \int |\FF \zeta|^2  \dd k\Big)^{\frac{1}{2}} + R^{-\frac{d}{2}} \Big( \int (1+|k|)^{2d} |\FF \zeta|^2  \dd k\Big)^{\frac{1}{2}}
      \\
      &\lesssim R^{\frac{d}{2}} \Big(\int_{\Boule_2(x)}\zeta^2\Big)^\frac{1}{2}
       + R^{-\frac{d}{2}} \Big(\int_{\Boule_2(x)}(|\nabla^d\zeta|^2+\zeta^2)\Big)^\frac{1}{2}.
    \end{align*}
    We now obtain \eqref{E1} by optimizing in $R$.
  \end{proof}

\subsection{Small scale regularity results}

\subsubsection{H\"older regularity of the Gaussian field}\label{SecProofLemloc}
  
    \begin{proof}[Proof of \eqref{wg03}]
      We first argue that it is enough to establish
      \begin{align}\label{Besov}
      \langl (g(h)-g(0))^2\rangl\lesssim|h|^{2\alpha}.
      \end{align}
      Indeed, since $g$ is Gaussian, \eqref{Besov} yields $\langl |g(h)-g(0)|^r\rangl^{\frac{1}{r}}\lesssim|h|^{\alpha}$, likewise, $\langl g^2(0) \rangl$ $ =c(0) \lesssim 1$ yields $\langl |g(0)|^r \rangl^{\frac{1}{r}} \lesssim 1$.
      Hence, by stationarity, we obtain
      \begin{align*}
	\Big\langl \int_{\Boule_2} \int_{\Boule_2} \Big(\frac{|g(x+h)-g(x)|}{|h|^{\alpha''}}\Big)^{r} \frac{\dd h}{|h|^d} \dd x + \int_{\Boule_2} |g|^r \Big\rangl
	\lesssim 1
      \end{align*}
      for all $r \in [1,\infty)$ and $\alpha''<\alpha$, which with help of a smooth cut-off function $\eta$ for $\Boule_1$ in $\Boule_2$ we rewrite as
      \begin{align*}
        \Big\langl \int_{\R^d} \int_{\Boule_2} \Big(\frac{|(\eta g)(x+h)-(\eta g)(x)|}{|h|^{\alpha''}}\Big)^{r} \frac{\dd h}{|h|^d} \dd x + \int_{\R^d} |\eta g|^r \Big\rangl
	\lesssim 1.
      \end{align*}
      This amounts to an estimate of $\eta g$ in the Besov space $\Besov^{\alpha''}_{r,r}(\R^d)$ \textit{cf.} \cite[Th.\ 2.36, p.\ 74]{Bahouri}, which embeds into a H\"older space $\CC^{\alpha'}(\R^d)$ for any $\alpha' < \alpha ''$ provided $r$ is sufficiently large (see \cite[Prop.\ 2.71, p.\ 99]{Bahouri} and \cite[p.\ 99, Examples]{Bahouri}).
      Therefore, \eqref{Besov} implies \eqref{wg03}.
      
      Estimate \eqref{Besov} can be seen as follows:
      \begin{align*}
      \langl (g(x)-g(0))^2\rangl
      &~=2c(0)-c(x)-c(-x)\\
      &~=\frac{4}{(2\pi)^d}\int \sin^2\big(\frac{k\cdot x}{2}\big)(\FF  c) \dd k
      \\
      &\stackrel{(\ref{wg01})}{\le} \frac{4}{(2\pi)^d}\int \min\Big\{\frac{|k|^2|x|^2}{4},1\Big\}|k|^{-d-2\alpha}\dd k
      \\
      &~= \frac{4}{(2\pi)^d}|x|^{2\alpha}\int \min\{|\hat k|^2,1\} |\hat k|^{-d-2\alpha}\dd\hat k,
      \end{align*}
      where the last integral converges because of $\alpha\in(0,1)$.
    \end{proof}

\subsubsection{Small-scale annealed CZ estimates for ensembles of H\"older continuous coefficients}

    We recall the norms  \eqref{Num:701} and \eqref{Ann_ap20} and state the following:
    \begin{lemma}\label{Lem_AnnealedCZ_2}
      Let the ensemble $\langl \cdot \rangl$ be defined as in Section \ref{SecAssumpGauss} and $T \geq 1$.
      Assume that the square-integrable functions $u$, $f$, and $g$ are related through
      \begin{equation}\label{Num:905}
	\frac{1}{T} u -\nabla\cdot a\nabla u=\frac{1}{T}g+\nabla \cdot f.
      \end{equation}
      (i) For $q,p\in(1,\infty)$ and $1\le r'<r<\infty$ we have
      \begin{equation}\label{CZstar_2}
      \begin{aligned}
        \big\|\big(\frac{u}{\sqrt{T}},\nabla u\big)\big\|_{p,r'}
	\lesssim_{\gamma,q,p,r',r}
	& \big\|\big(\frac{u}{\sqrt{T}},\nabla u\big)\big\|_{p,r,q}
	+\big\|\big(\frac{g}{\sqrt{T}},f\big)\big\|_{p,r}.
      \end{aligned}
      \end{equation}
      (ii) For $q\in(1,2]$ and $r>2$ we have
      \begin{equation}\label{Num:116}
	\begin{aligned}
	  &\big\|\big(\frac{u}{\sqrt{T}},\nabla u\big)\big\|_{2,2,q}
	  \lesssim_{\gamma,q,r}
	  \big\|\big(\frac{g}{\sqrt{T}},f\big)\big\|_{2,r,q}.
	\end{aligned}
      \end{equation}
    \end{lemma}    
    
    \begin{proof}
      We split the proof of Lemma \ref{Lem_AnnealedCZ_2} into two independent parts.
      
      \parag{Part 1: Proof of \eqref{CZstar_2}}
      Let $q,p, r'$ be as in the statement, and choose $r'' \in (r',r)$.
      W.~l.~o.~g.\ we may assume that $q \leq\min\{r',p\}$ (the other case $q \geq\min\{r',p\}$ can be recovered from the former one by Jensen's inequality).
      
      The proof is divided into three steps.
      In Step 1, we define a (small) random radius $\rho_\star$ below which we may apply classical regularity theory.
      In Step 2, we define a splitting of $u$ by means of which we establish \eqref{CZstar_2}, admitting the estimate that for any $F=F(a)$ and $y \in \R^d$
      \begin{equation}
	\label{Num:019}
	\big\langle |F|^{r'} \big\rangle^{\frac{p}{r'}}
	\lesssim
	\int \big\langle \mathds{1}_{\Boule_{\frac{\rho_\star(x)}{2}}(x)}(y) |F|^{r''} \big\rangle^{\frac{p}{r''}} \dd x,
      \end{equation} 
      which we derive in Step 3.
      \parag{Part 1, Step 1: Definition of the radius $\rho_\star$}
	Let $\delta \leq 1$ (to be fixed in Step 2). 
	We introduce the (local) radius below which $a$ does not vary much:
	\begin{equation}
	  \label{Def_rhostar}
	  \rho_\star(x):= \sup \Big\{ \rho \leq 1, \sup_{y,y' \in \Boule_\rho(x)} |a(y)-a(y')| \leq \delta \Big\}.
	\end{equation}
	Note that by \eqref{wg03} we have $0<\rho_\star(x) \leq 1$ $\langl\cdot\rangl$-almost-surely; furthermore, $\rho_\star$ is stationary.
	We claim that $\rho_\star$ is bounded by below in the sense of
	\begin{equation}
	  \label{Num:012}
	  \langl \rho_\star^{-s} \rangl^{\frac{1}{s}} \lesssim_{\gamma,s,\delta} 1 \pourtout s \in [1,\infty).
	\end{equation}
	
	By stationarity, we may assume that $x=0$ (we omit this variable in the expressions below), and estimate $\rho_\star$ by the easier quantity
	\begin{align*}
	  \tilde{\rho}_\star:= \min\Big\{\frac{1}{2}\Big(\frac{\delta}{\|a\|_{\CC^{\frac{\alpha}{2}}(\Boule_1)}} \Big)^{\frac{2}{\alpha}},1\Big\}.
	\end{align*}
	Indeed we have $\rho_\star \geq \tilde{\rho}_\star$ because, if $y, y' \in \Boule_{\tilde{\rho}_\star} \subset \Boule_{1}$ then
	\begin{equation*}
	  |a(y)-a(y')|
	  \leq \|a\|_{\CC^{\frac{\alpha}{2}}(\Boule_1)} |y-y'|^{\frac{\alpha}{2}}
	  \leq \|a\|_{\CC^{\frac{\alpha}{2}}(\Boule_1)} (2\tilde{\rho}_\star)^{\frac{\alpha}{2}}
	  \leq \delta.
	\end{equation*}
	Therefore \eqref{Num:012} is a consequence of \eqref{wg03}:
	\begin{equation*}
	  \begin{aligned}
	    \langl \rho_\star^{-s} \rangl^{\frac{1}{s}} 
	    \leq \langl \tilde{\rho}_\star^{-s} \rangl^{\frac{1}{s}} 
	    \lesssim_{\gamma,s,\delta} \big\langl \|a\|_{\CC^{\frac{\alpha}{2}}(\Boule_1)}^{\frac{2s}{\alpha}} \big\rangl^{\frac{1}{s}}
	    \overset{\eqref{wg03}}{\lesssim}_{\gamma,s,\delta} 1.
	  \end{aligned}
	\end{equation*}	
	
      \parag{Part 1, Step 2: Splitting and argument for \eqref{CZstar_2}}
	We introduce
	\begin{equation*}
	  \begin{aligned}
	    &\Boule_\star(x):=\Boule_{\rho_\star(x)}(x),
	    \quad \tilde{g}_x:=\mathds{1}_{\Boule_1(x)}g,
	    \quad\tilde{f}_x:= \mathds{1}_{\Boule_1(x)}f,
	    \\
	    &\text{and} \qquad
	    \tilde{a}_x(y):=
	    \begin{cases}
	      a(y) & \text{if } y \in \Boule_\star(x),
	      \\
	      a(x) & \text{otherwise}.
	    \end{cases}
	  \end{aligned}
	\end{equation*}
	Accordingly, we freeze the variable $x$ and define $\tilde{u}_x$ as the solution to
	\begin{equation*}
	  \frac{1}{T} \tilde{u}_x -\nabla \cdot \tilde{a}_x \nabla \tilde{u}_x=\frac{1}{T}\tilde{g}_x+ \nabla \cdot \tilde{f}_x,
	\end{equation*}
	where the operator $\nabla$ acts on the variable $y$.
	This combines with \eqref{Num:905} to
	\begin{equation}\label{Num:013}
	  \frac{1}{T}(\tilde{u}_x-u) - \nabla \cdot a \nabla (\tilde{u}_x-u)=0 \dans \Boule_\star(x).
	\end{equation}
	
	First, we estimate $\tilde{u}_x$.
	Since by definition \eqref{Def_rhostar} of $\rho_\star$, we have $\|\tilde{a}_x-a(x)\|_{\LL^\infty} \leq \delta$, by a perturbative argument \textit{\`a la Meyers} as in the proof Proposition \ref{PropCZ}(i), there holds for  $\delta \ll_{d,\lambda,p,q,r''} 1$
	\begin{align}
	  \label{Num:011}
	  \int \big| \big(\frac{\tilde{u}_x}{\sqrt{T}}, \nabla \tilde{u}_x\big) \big|^q 
	  &\lesssim \int_{\Boule_1(x)}  \big|\big(\frac{g}{\sqrt{T}},f\big)\big|^q,
	  \\
	  \label{Num:010}
	  \int \big\langle \big|\big(\frac{\tilde{u}_x}{\sqrt{T}},\nabla \tilde{u}_x\big) \big|^{r''} \big\rangle^{\frac{p}{{r''}}}
	  &\lesssim \int_{\Boule_1(x)} \big\langle \big|\big(\frac{g}{\sqrt{T}},f\big) \big|^{r''} \big\rangle^{\frac{p}{{r''}}}.
	\end{align}
	
	Second, we estimate $\nabla u -\nabla \tilde{u}_x$.
	Using Lemma \ref{LemAppendHold}, more precisely a rescaled version of \eqref{Num:031_bis} for $\alpha \rightsquigarrow \frac{\alpha}{2}$, on \eqref{Num:013}, the triangle inequality and \eqref{Num:011}, there also holds
	\begin{equation*}
	\begin{aligned}
	  &\sup_{\frac{1}{2} \Boule_\star(x)} \Big|\Big(\frac{u-\tilde{u}_x}{\sqrt{T}},\nabla u-\nabla \tilde{u}_x\Big)\Big|
	  \\
	  &\qquad\overset{\eqref{Num:031_bis}}{\lesssim}
	  \big(1+\rho_\star^{\frac{\alpha}{2}}(x) \|a\|_{\CC^{\frac{\alpha}{2}}(\Boule_\star(x))}\big)^{\frac{2}{\alpha}(\frac{d}{q}+1)}
	  \Big(\fint_{\Boule_\star(x)} \Big|\Big(\frac{u-\tilde{u}_x}{\sqrt{T}},\nabla u-\nabla \tilde{u}_x\Big)\Big|^q \Big)^{\frac{1}{q}}
	  \\
	  &\quad\overset{\eqref{Def_rhostar},\eqref{Num:011}}{\lesssim}
	  \big(1+\|a\|_{\CC^{\frac{\alpha}{2}}(\Boule_\star(x))}\big)^{\frac{2d}{q\alpha}} \rho_\star^{-\frac{d}{q}}(x) \Big(\fint_{\Boule_1(x)} \big|\big(\frac{g}{\sqrt{T}},f,\frac{u}{\sqrt{T}},\nabla u\big)\big|^q \Big)^{\frac{1}{q}}.
	\end{aligned}
	\end{equation*}
	Taking the $\LL^{r''}_{\langl\cdot\rangl}$-norm of this estimate and using the H\"older inequality followed by \eqref{wg03} and \eqref{Num:012} yields
	\begin{equation}\label{Num:914}
	\begin{aligned}
	  &\Big\langle \sup_{\frac{1}{2} \Boule_\star(x)} \Big|\Big(\frac{u-\tilde{u}_x}{\sqrt{T}},\nabla u-\nabla \tilde{u}_x\Big)\Big|^{r''} \Big\rangle^{\frac{1}{{r''}}}
	  \\
	  &\qquad\qquad\lesssim
	  \Big\langle\Big(\fint_{\Boule_1(x)} \big|\big(\frac{g}{\sqrt{T}},f,\frac{u}{\sqrt{T}},\nabla u\big)\big|^q \Big)^{\frac{r}{q}} \Big \rangle^{\frac{1}{{r}}}.
	\end{aligned}
	\end{equation}
	
	Finally, momentarily fixing a $y\in\R^d$, which we suppress in most of the notation, by \eqref{Num:019} (for $F:=(\frac{u(y)}{\sqrt{T}},\nabla u(y))$) followed by the triangle inequality, there holds
	\begin{align*}
	  \big\langle \big|\big(\frac{u}{\sqrt{T}},\nabla u\big)\big|^{r'} \big\rangle^{\frac{p}{r'}}
	  &\lesssim
	  \int\big\langle \mathds{1}_{\frac{1}{2}\Boule_\star(x)} \big|\big(\frac{\tilde{u}_x}{\sqrt{T}},\nabla \tilde{u}_x\big)\big|^{r''} \big\rangle^{\frac{p}{r''}} \dd x
	  \\
	  &\qquad+\int\big\langle \mathds{1}_{\frac{1}{2}\Boule_\star(x)} \big|\big(\frac{u-\tilde{u}_x}{\sqrt{T}},\nabla u-\nabla \tilde{u}_x\big)\big|^{r''} \big\rangle^{\frac{p}{r''}} \dd x.
	\end{align*}
	Integrating this in the $y$ variable (recall that $\rho_\star \leq 1$), and using \eqref{Num:010}, \eqref{Num:914}, and the definitions of the norms \eqref{Num:701} and \eqref{Ann_ap20} yields
	\begin{align*}
	  \int \big\langle \big|\big(\frac{u}{\sqrt{T}},\nabla u\big)\big|^{r'} \big\rangle^{\frac{p}{r'}}
	  &\lesssim
	  \int\int \big\langle \big|\big(\frac{\tilde{u}_x}{\sqrt{T}},\nabla \tilde{u}_x\big)\big|^{r''} \big\rangle^{\frac{p}{r''}}  \dd x
	  \\
	  &\qquad+\int\int\big\langle \mathds{1}_{\frac{1}{2}\Boule_\star(x)} \big|\big(\frac{u-\tilde{u}_x}{\sqrt{T}},\nabla u-\nabla \tilde{u}_x\big)\big|^{r''} \big\rangle^{\frac{p}{r''}}  \dd x
	  \\
	  &\!\!\!\!\!\!\!\!\overset{\eqref{Num:010}, \eqref{Num:914}}{\lesssim}
	  \int \int_{\Boule_1(x)} \big\langle \big|\big(\frac{g}{\sqrt{T}},f\big) \big|^{r''} \big\rangle^{\frac{p}{{r''}}} \dd x
	  \\
	  &\qquad+
	  \int\Big\langle\Big(\fint_{\Boule_1(x)} \big|\big(\frac{g}{\sqrt{T}},f,\frac{u}{\sqrt{T}},\nabla u\big)\big|^q \Big)^{\frac{r}{q}} \Big \rangle^{\frac{p}{{r}}}\dd x.
	\end{align*}
	By definition of the norms \eqref{Num:701}, this assumes the form
	\begin{align*}
	  \big\| \big(\frac{u}{\sqrt{T}},\nabla u\big) \big\|_{p,r'}
	  \overset{\eqref{Num:556}, \eqref{Ann_ap20}}{\lesssim} \big\| \big(\frac{g}{\sqrt{T}},f\big) \big\|_{p,r''} 
	  +  \big\|\big(\frac{g}{\sqrt{T}},f,\frac{u}{\sqrt{T}},\nabla u\big)\big\|_{p,r,q}.
	\end{align*}
	By the estimates \eqref{Num:03_bis} and \eqref{Num:706} (recall that $r''\leq r$ and $q \leq \min\{p,r\}$), we have
	\begin{equation*}
	  \big\| \big(\frac{g}{\sqrt{T}},f\big) \big\|_{p,r''} + \big\| \big(\frac{g}{\sqrt{T}},f\big) \big\|_{p,r,q} \lesssim \big\| \big(\frac{g}{\sqrt{T}},f\big) \big\|_{p,r}.
	\end{equation*}
	This yields the desired estimate \eqref{CZstar_2}.
	
      \parag{Part 1, Step 3: Argument for \eqref{Num:019}}
	By the H\"older inequality in space (based on $\rho_\star \leq 1$) followed by Jensen's inequality, there holds
	\begin{equation}\label{Num:021}
	  \begin{aligned}
	    \Big(\int \big\langl \mathds{1}_{\frac{1}{2}\Boule_\star(x)}(y) |F|^{r''} \big\rangl^{\frac{p}{{r''}}}\dd x \Big)^{\frac{1}{p}}
	    &\gtrsim
	    \int \big\langl \mathds{1}_{\frac{1}{2}\Boule_\star(x)}(y) |F|^{r''} \big\rangl^{\frac{1}{{r''}}}\dd x
	    \\
	    &\geq
	    \Big\langl \Big( \int \mathds{1}_{\frac{1}{2}\Boule_\star(x)}(y)\dd x |F|\Big)^{r''} \Big\rangl^{\frac{1}{{r''}}}.
	  \end{aligned}
	\end{equation}
	Note that if $x \in \frac{1}{3}\Boule_\star(y)$ then $\Boule_{\frac{2}{3}\rho_\star(y)}(x) \subset \Boule_{\star}(y)$.
	Indeed, if $z \in \Boule_{\frac{2}{3}\rho_\star(y)}(x)$, then we have by the triangle inequality $|z-y| \leq |z-x| + |x-y|< \rho_\star(y)$.
	Therefore, by definition \eqref{Def_rhostar} of $\rho_\star$, this
	implies that $\rho_\star(x) \geq \frac{2}{3}\rho_\star(y)$ and in particular $y \in \frac{1}{2}\Boule_\star(x)$.
	As a consequence, there holds
	\begin{equation*}
	    \int \mathds{1}_{\frac{1}{2}\Boule_\star(x)}(y)\dd x 
	    \geq \int_{\frac{1}{4}\Boule_\star(y)}\dd x 
	    \gtrsim \rho_\star(y)^d.
	\end{equation*}
	Inserting this estimate into \eqref{Num:021}, we get
	\begin{equation}\label{Num:022}
	\begin{aligned}
	  \langl  \rho_\star(y)^{{r''}d} |F|^{r''} \rangl^{\frac{1}{{r''}}}
	  \lesssim
	  \Big(\int \big\langl \mathds{1}_{\frac{1}{2}\Boule_\star(x)}(y) |F|^{r''} \big\rangl^{\frac{p}{{r''}}}\dd x \Big)^{\frac{1}{p}}.
	\end{aligned}
	\end{equation}
	Finally, by the H\"older inequality we obtain
	\begin{equation*}
	    \langl |F|^{r'} \rangl^{\frac{1}{r'}}
	    =
	    \langl (\rho_\star(y)^{d} |F|)^{r'} 
	    \rho_\star(y)^{-d r'} \rangl^{\frac{1}{r'}}
	    \leq \langl \rho_\star(y)^{{r''}d} |F|^{{r''}} \rangl^{\frac{1}{{r''}}}  \big\langl \rho_\star(y)^{-d\frac{r'r''}{{r''}-r'}}\big\rangl^{\frac{{r''}-r'}{{r''} r'}},
	\end{equation*}
	whence, by \eqref{Num:012} (for $s \rightsquigarrow d\frac{r'r''}{{r''}-r'}$),
	\begin{equation*}
	    \langl |F|^{r'} \rangl^{\frac{1}{r'}}
	    \lesssim \big\langl \rho_\star(y)^{r'' d} |F|^{{r''}} \big\rangl^{\frac{1}{{r''}}}.
	\end{equation*}
	Inserting \eqref{Num:022} into the above estimate, we obtain \eqref{Num:019}.

      \parag{Part 2: Proof of \eqref{Num:116}} 
	By duality, it is sufficient to establish
	for $p\in[2,\infty)$ and $r<2$ that
	\begin{align}\label{cw20}
	\big\|\big(\frac{u}{\sqrt{T}},\nabla u\big)\big\|_{2,r,p}\lesssim\big\|\big(\frac{g}{\sqrt{T}},f\big)\big\|_{2,2,p}.
	\end{align}
	In order to establish this, we appeal to Lemma \ref{LemAppendHold}, more precisely (\ref{cw15}) for $q=2$, 
	with $\alpha$ replaced by some $0<\alpha'<\alpha$ (which we now fix),
	and with the origin replaced by a general point $x\in\mathbb{R}^d$:
	\begin{align*}
	\lefteqn{\big\|\big(\frac{u}{\sqrt{T}},\nabla u\big)\big\|_{\LL^p(\Boule_1(x))}}\nonumber\\
	&\qquad\lesssim
	\|a\|_{\CC^{\alpha'}(\Boule_2(x))}^{\frac{1}{\alpha'}(\frac{d}{2}-\frac{d}{p})}
	\big\|\big(\frac{u}{\sqrt{T}},\nabla u\big)\big\|_{\LL^2(\Boule_2(x))}
	+\big\|\big(\frac{g}{\sqrt{T}},f\big)\big\|_{\LL^p(\Boule_2(x))}.
	\end{align*}
	We first take the $\LL^r_{\langle\cdot\rangle}$-norm; by \eqref{wg03} in Lemma \ref{LemSG}
	(note that the shift by $x$ is irrelevant by stationarity) in conjunction with 
	H\"older's inequality (recall $r<2$) for the first r.~h.~s.~term 
	(and Jensen's inequality on the second) we obtain
	\begin{align*}
	\lefteqn{\big\|\big\|\big(\frac{u}{\sqrt{T}},\nabla u\big)\big\|_{\LL^p(\Boule_1(x))}\big\|_{\LL^r_{\langle\cdot\rangle}}}
	\nonumber\\
	&\qquad\lesssim
	\big\|\big\|\big(\frac{u}{\sqrt{T}},\nabla u\big)\big\|_{\LL^2(\Boule_2(x))}\big\|_{\LL^2_{\langle\cdot\rangle}}
	+\big\|\big\|\big(\frac{g}{\sqrt{T}},f\big)\big\|_{\LL^p(\Boule_2(x))}\big\|_{\LL^2_{\langle\cdot\rangle}}.
	\end{align*}
	We next take the $\LL^2(\mathbb{R}^d)$-norm (in $x$) of this: On the l.~h.~s.~we use the definition \eqref{Ann_ap20} and on both r.~h.~s.~terms we use \eqref{Ann_ap23}, to the effect of
	\begin{align*}
	\big\|\big(\frac{u}{\sqrt{T}},\nabla u\big)\big\|_{2,r,p}\lesssim
	\big\|\big(\frac{u}{\sqrt{T}},\nabla u\big)\big\|_{2,2,2}
	+\big\|\big(\frac{g}{\sqrt{T}},f\big)\big\|_{2,2,p}.
	\end{align*}
	It remains to appeal to the energy estimate \eqref{Ann_ap97}, H\"older's inequality in the inner space variable, in conjunction with a version of \eqref{Num:556} to obtain \eqref{cw20}:
	\begin{align*}
	\big\|\big(\frac{u}{\sqrt{T}},\nabla u\big)\big\|_{2,2,2} &\stackrel{\eqref{Num:556}}{=}
	\big\|\big(\frac{u}{\sqrt{T}},\nabla u\big)\big\|_{2,2}\\
	&\stackrel{\eqref{Ann_ap97}}{\lesssim}
	\big\|\big(\frac{g}{\sqrt{T}},f\big)\big\|_{2,2}\stackrel{\eqref{Num:556}}{=}
	\big\|\big(\frac{g}{\sqrt{T}},f\big)\big\|_{2,2,2}\lesssim\big\|\big(\frac{g}{\sqrt{T}},f\big)\big\|_{2,2,p}.
	\end{align*}
    \end{proof}

\subsubsection{Local estimates for H\"older continuous coefficients}

  The interest of the local CZ and Schauder estimates below is the explicit, polynomial dependence on the H\"older norm of the coefficient field $a$. A secondary aspect is the uniformity in the massive parameter $T$. The ingredients are classical; we give an argument for the convenience of the reader.

  \begin{lemma}\label{LemAppendHold}
    Let $\alpha \in (0,1)$, and let $a \in \CC^{\alpha}(\Boule_1)$ satisfy~\eqref{Ellipticite}. 
    For $T\ge 1$,  we assume that the functions $u$, $g$, and the vector field $f$ satisfy the relation
    \begin{align}\label{cw09}
    \frac{1}{T}u-\nabla\cdot a\nabla u=\frac{1}{T}g+\nabla\cdot f\quad\mbox{in}\;\Boule_2.
    \end{align}
    For $1<q\le p<\infty$ we claim that
    \begin{equation}\label{cw15}
      \begin{aligned}
        \lefteqn{\big\|\big(\frac{u}{\sqrt{T}},\nabla u\big)\big\|_{\LL^p(\Boule_1)}}
        \\
	&\qquad\lesssim_{d,\lambda,\alpha,p,q}
	\|a\|_{\CC^\alpha(\Boule_2)}^{\frac{1}{\alpha}(\frac{d}{q}-\frac{d}{p})}
	\big\|\big(\frac{u}{\sqrt{T}},\nabla u\big)\big\|_{\LL^q(\Boule_2)}
	+\big\|\big(\frac{g}{\sqrt{T}},f\big)\big\|_{\LL^p(\Boule_2)},    
      \end{aligned}
    \end{equation}
    and, provided $g\equiv0$,
    \begin{equation}\label{cw18}
      \begin{aligned}
	\lefteqn{\big\|\big(\frac{u}{\sqrt{T}},\nabla u\big)\big\|_{\CC^{\alpha}(\Boule_1)}}
	\\
	&\qquad\lesssim_{d,\lambda,\alpha,q}
	\|a\|_{\CC^\alpha(\Boule_2)}^{\frac{1}{\alpha}(\frac{d}{q}+1)}
	\big\|\big(\frac{u}{\sqrt{T}},\nabla u\big)\big\|_{\LL^q(\Boule_{2})}
	+\|a\|_{\CC^\alpha(\Boule_2)}^{\frac{1}{\alpha}-1}\|f\|_{\CC^{\alpha}(\Boule_{2})}.
      \end{aligned}
    \end{equation}
  \end{lemma}

    In case of $g=0$ and $f=0$,
    since $T \geq 1$, \eqref{cw18} implies
    \begin{equation}\label{Num:031_bis}
      \begin{aligned}
        \big\|\big(\frac{u}{\sqrt{T}},\nabla u\big)\big\|_{\LL^\infty(\Boule_{1})} 
        \lesssim_{d,\lambda,\alpha,q}~& \|a \|_{\CC^{\alpha}(\Boule_2)}^{\frac{1}{\alpha}(\frac{d}{q}+1)} \big\|\big(\frac{u}{\sqrt{T}},\nabla u\big)\big\|_{\LL^q(\Boule_2)}.
      \end{aligned}
    \end{equation}

\begin{proof}
  The strategy of the proof is the following: First, assuming the additional condition $[a]_{\CC^{\alpha}(\Boule_2)}\le 1$, we prove in Step 1 and 2 that  \eqref{cw15} and \eqref{cw18} hold.
  In Step 3, we make use of a rescaling in order to apply the previous result and finally get \eqref{cw15} and \eqref{cw18} in full generality.
  
  \parag{Step 1: Proof of statement under an additional assumption} Provided 
    \begin{align}\label{cw08}
      [a]_{\CC^{\alpha}(\Boule_2)}\le 1
    \end{align}
    we claim
    \begin{align}\label{cw01}
      \big\|\big(\frac{u}{\sqrt{T}},\nabla u\big)\big\|_{\LL^p(\Boule_1)}\lesssim
      \big\|\big(\frac{u}{\sqrt{T}},\nabla u\big)\big\|_{\LL^q(\Boule_2)}+ \big\|\big(\frac{g}{\sqrt{T}},f\big)\big\|_{\LL^p(\Boule_2)},
    \end{align}
    and, in case of $g\equiv0$,
    \begin{align}\label{cw02}
      \big\|\big(\frac{u}{\sqrt{T}},\nabla u\big)\big\|_{\CC^{\alpha}(\Boule_1)}\lesssim
      \big\|\big(\frac{u}{\sqrt{T}},\nabla u\big)\big\|_{\LL^q(\Boule_2)}+ \|f\|_{\CC^{\alpha}(\Boule_2)}.
    \end{align}

    Here comes the argument: We split $u=v+w$ with help of the Dirichlet problem 
    \begin{align}\label{cw10}
      \frac{1}{T}v-\nabla\cdot a\nabla v=\frac{1}{T}g+\nabla\cdot f\;\;\mbox{in}\;\Boule_2,\quad
      v=0\;\;\mbox{on}\;\partial \Boule_2,
    \end{align}
    so that in view of (\ref{cw09}), $w$ satisfies
    \begin{align}\label{cw07}
      \frac{1}{T}w-\nabla\cdot a\nabla w=0\quad\mbox{in}\;\Boule_2.
    \end{align}
    By standard CZ and Schauder theory for (\ref{cw10})  (see  \cite[Th.\ 7.2 p.\ 140]{Giaquinta}, and \cite[Chap.\ 8 p.\ 177-218]{GT} or \cite[Chap.\ 5 p.\ 75-96]{Giaquinta}, respectively), which is not affected by the presence of the massive term with $T\ge 1$ (as for instance can be seen by expressing the solution operator in terms of the semi-group and appealing to the parabolic version of the theory), 
    we have in case of $g\equiv 0$,
    \begin{align}\label{cw11}
      \|\nabla v\|_{\LL^p(\Boule_2)}&\lesssim\|f\|_{\LL^p(\Boule_2)}\quad\mbox{and}\quad
      \|\nabla v\|_{\CC^{\alpha}(\Boule_2)}\lesssim\|f\|_{\CC^{\alpha}(\Boule_2)}.
    \end{align}
    Note that thanks to our assumption (\ref{cw08}),
    the implicit constant depends only on $d,\lambda,p,\alpha$.
    In the presence of $g$, we solve the auxiliary Dirichlet problem $-\triangle v'=\frac{1}{T}g$ on $\Boule_2$, so that $f':=-\nabla v'$ satisfies $\nabla\cdot f'=\frac{1}{T}g$.
    By constant-coefficient CZ theory, we have $\|\nabla f'\|_{\LL^p(\Boule_2)}\lesssim\|\frac{1}{T}g\|_{\LL^p(\Boule_2)}$.
    In conjunction with the Poincar\'e estimate for $v$, and the Poincar\'e-Wirtinger estimate for $f'$ (which only matters up to a constant), we may upgrade the first item in (\ref{cw11}) to
    \begin{align*}
      \lefteqn{\|(\frac{v}{\sqrt{T}},\nabla v)\|_{\LL^p(\Boule_2)}
      \stackrel{T\ge 1}{\le}\|(v,\nabla v)\|_{\LL^p(\Boule_2)}\lesssim\|\nabla v\|_{\LL^p(\Boule_2)}}
      \\
      &\qquad\lesssim\|(f',f)\|_{\LL^p(\Boule_2)}\lesssim\|(\frac{g}{T},f)\|_{\LL^p(\Boule_2)}
      \stackrel{T\ge 1}{\le}\big\|\big(\frac{g}{\sqrt{T}},f\big)\big\|_{\LL^p(\Boule_2)}.
    \end{align*}
    Since $\|\nabla v\|_{\CC^0(\Boule_2)}$ controls the Lipschitz constant of $v$,   the second item in (\ref{cw11}) obviously upgrades to
    \begin{align*}
      \|(\frac{v}{\sqrt{T}},\nabla v)\|_{\CC^{\alpha}(\Boule_2)}&\lesssim\|f\|_{\CC^{\alpha}(\Boule_2)}.
    \end{align*}
    It remains to appeal to (\ref{cw06}) in Step 2 to obtain (\ref{cw01}) and (\ref{cw02})
    by the triangle inequality.
  
  \parag{Step 2: Inner regularity theory} Suppose that $w$ satisfies (\ref{cw07}).
    Then under the assumption (\ref{cw08}) we claim that for any $p,q\in(1,\infty)$
    \begin{align}\label{cw06}
      \big\|\big(\frac{w}{\sqrt{T}},\nabla w\big)\big\|_{\LL^p(\Boule_1)}\lesssim
      \big\|\big(\frac{w}{\sqrt{T}},\nabla w\big)\big\|_{\CC^\alpha(\Boule_1)}\lesssim
      \big\|\big(\frac{w}{\sqrt{T}},\nabla w\big)\big\|_{\LL^q(\Boule_2)}.
    \end{align}
    The first estimate is trivial. 
    For the second estimate, the presence of the massive term in \eqref{cw07} requires a classical bootstrap argument that we outline for the convenience of the user. We split $w=w'+(w-w')$ once more according to
    \begin{align}\label{cw05}
      -\nabla\cdot a\nabla w'=-\frac{1}{T}w\;\;\mbox{in}\;\Boule_2,\quad
      w'=0\;\;\mbox{on}\;\partial \Boule_2.
    \end{align}
    We have by CZ theory that
    $\|\nabla w'\|_{\LL^{q'}(\Boule_2)}\lesssim\|\frac{w}{T}\|_{\LL^q(\Boule_2)}$, where
    the prime on an exponent indicates the improvement coming from Sobolev's embedding,
    that is,
    \begin{align*}
      q'=\frac{dq}{d-q}
    \end{align*}
    (which we only use as long as $q<d$).
    For the second contribution $w-w'$, which by (\ref{cw07}) and (\ref{cw05}) satisfies the homogeneous equation $-\nabla\cdot a\nabla(w-w')=0$ in $\Boule_2$, we have by CZ-based inner regularity theory $\|\nabla(w-w')\|_{\LL^{q'}(\Boule_1)}\lesssim \|\nabla(w-w')\|_{\LL^q(\Boule_2)}$.
    Appealing to the triangle inequality and H\"older's inequality in form of $\|\nabla w'\|_{\LL^q(\Boule_2)}$ $\lesssim\|\nabla w'\|_{\LL^{q'}(\Boule_2)}$, both estimates add up to $\|\nabla w\|_{\LL^{q'}(\Boule_1)}\lesssim\|\frac{w}{T}\|_{\LL^q(\Boule_2)}+\|\nabla w\|_{\LL^q(\Boule_2)}$.
    In combination with the Poincar\'e estimate $\|w\|_{\LL^{q'}(\Boule_1)}\lesssim \|\nabla w\|_{\LL^{q'}(\Boule_1)}+\|w\|_{\LL^q(\Boule_1)}$, and using $T\ge 1$, this yields the iterable form
    \begin{align*}
      \big\|\big(\frac{w}{\sqrt{T}},\nabla w\big)\big\|_{\LL^{q'}(\Boule_1)}\lesssim 
      \big\|\big(\frac{w}{\sqrt{T}},\nabla w\big)\big\|_{\LL^q(\Boule_2)}.
    \end{align*}
    We now consider an exponent $p>d$ and set
    \begin{align*}
      \alpha'=1-\frac{d}{p}
    \end{align*}
    (which we use only for $\alpha'\le\alpha$).
    In this situation, in the splitting introduced in (\ref{cw05}), we may pass from
    CZ-theory to Schauder theory to obtain $\|\nabla w'\|_{\CC^{\alpha'}(\Boule_2)}$ $\lesssim\|\frac{w}{T}\|_{\LL^p(\Boule_2)}$ and
    $\|\nabla(w-w')\|_{\CC^{\alpha'}(\Boule_1)}\lesssim\|\nabla(w-w')\|_{\LL^q(\Boule_2)}$.
    This combines to $\|\nabla w\|_{\CC^{\alpha'}(\Boule_1)}$
    $\lesssim\|\frac{w}{T}\|_{\LL^p(\Boule_r)}+\|\nabla w\|_{\LL^q(\Boule_r)}$, and as above
    upgrades to
    \begin{align*}
      \big\|\big(\frac{w}{\sqrt{T}},\nabla w\big)\big\|_{\CC^{\alpha'}(\Boule_1)}\lesssim
      \big\|\big(\frac{w}{\sqrt{T}},\nabla w\big)\big\|_{\LL^q(\Boule_2)}.
    \end{align*}
    In one final step of Schauder theory, applied to both contributions of the splitting (\ref{cw05}), 
    we obtain
    \begin{align*}
      \big\|\big(\frac{w}{\sqrt{T}},\nabla w\big)\big\|_{\CC^{\alpha}(\Boule_1)}\lesssim
      \big\|\big(\frac{w}{\sqrt{T}},\nabla w\big)\big\|_{\CC^{\alpha'}(\Boule_2)}.
    \end{align*}
    An inspection of the above arguments shows that the pair of balls $(\Boule_1,\Boule_2)$ may be replaced by $(\Boule_r,\Boule_R)$ with
    $1\le r<R\le 2$, at the expense of a constant that depends on $R-r$.
    Hence a finite iteration indeed yields (\ref{cw06}).

  \parag{Step 3: Removing the additional assumption (\ref{cw18})}
    For any radius $0<R\le 1$ and point $z\in\mathbb{R}^d$, the change of variables $x= R\hat x+z$, 
    alongside with $\hat T=R^{-2}T\ge 1$, shows the statement of Step 1 generalizes to:
    Provided
    \begin{align}\label{cw12}
    R^\alpha[a]_{\CC^{\alpha}(\Boule_{2R}(z))}\le 1
    \end{align}
    we have
    \begin{equation}\label{cw13}
    \begin{aligned}
    &\big\|\big(\frac{u}{\sqrt{T}},\nabla u\big)\big\|_{\LL^p(\Boule_R(z))}
    \\
    &\qquad\qquad\lesssim
    R^{\frac{d}{p}-\frac{d}{q}}\big\|\big(\frac{u}{\sqrt{T}},\nabla u\big)\big\|_{\LL^q(\Boule_{2R}(z))}
    +\|(\frac{1}{\sqrt{T}},f)\|_{\LL^p(\Boule_{2R}(z))},  
    \end{aligned}
    \end{equation}
    and, in case of $g\equiv0$,
    \begin{equation}\label{cw17}
      \begin{aligned}
        \big\|\big(\frac{u}{\sqrt{T}},\nabla u\big)\big\|_{\CC^{\alpha}(\Boule_R(z))}
        \lesssim
	R^{-\alpha-\frac{d}{q}}\big\|\big(\frac{u}{\sqrt{T}},\nabla u\big)\big\|_{\LL^q(\Boule_{2R}(z))}
	+\|f\|_{\CC^{\alpha}(\Boule_{2R}(z))}.
      \end{aligned}
    \end{equation}

    Selecting
    $R\le\frac{1}{2}$ to be a small fraction of $\|a\|_{\CC^\alpha(\Boule_2)}^{-\frac{1}{\alpha}}\lesssim \lambda^{\frac{1}{\alpha}}$ (recall \eqref{Ellipticite}),
    we may cover $\Boule_1$ by the union of $N$ balls $\Boule_R(z_n)$, $n=1,\cdots,N$, with
    \begin{align}\label{cw14}
      R^{-1}\sim\|a\|_{\CC^\alpha(\Boule_2)}^{\frac{1}{\alpha}}
    \end{align}
    in such a way that (\ref{cw12}) is satisfied for every $z=z_n\in \Boule_1$.
    Taking the $\ell^p$-norm of (\ref{cw13}) (with $z=z_n$) in $n=1,\cdots,N$, we obtain 
    \begin{align}\label{cw16}
      \big\|\big(\frac{u}{\sqrt{T}},\nabla u\big)\big\|_{\LL^p(\Boule_1)}
      \lesssim
      R^{\frac{d}{p}-\frac{d}{q}}\big\|\big(\frac{u}{\sqrt{T}},\nabla u\big)\big\|_{\LL^q(\Boule_2)}
      +\big\|\big(\frac{g}{\sqrt{T}},f\big)\big\|_{\LL^p(\Boule_2)},
    \end{align}
    where on the first r.~h.~s.~term we used that (recall $p\ge q$) 
    \begin{align*}
      \big(\sum_{n=1}^N\|f\|_{\LL^q(\Boule_{2R}(z_n))}^p\big)^\frac{1}{p}
      \le\big(\sum_{n=1}^N\|f\|_{\LL^q(\Boule_{2R}(z_n))}^q\big)^\frac{1}{q}
      \lesssim \|f\|_{\LL^q(\Boule_2)}.
    \end{align*}
    Now (\ref{cw15}) follows from inserting (\ref{cw14}) into (\ref{cw16}).

    We now turn to the Schauder estimate \eqref{cw18}.
    Using a telescoping sum we learn the norm relation
    \begin{align*}
      \|u\|_{\CC^\alpha(\Boule_1)}\lesssim R^{\alpha-1}\max_{n=1,\cdots,N}\|u\|_{\CC^\alpha(\Boule_R(z_n))}.
    \end{align*}
    Hence taking the maximum of (\ref{cw17}) (with $z=z_n$) in $n=1,\cdots,N$ yields
    \begin{align*}
      \big\|\big(\frac{u}{\sqrt{T}},\nabla u\big)\big\|_{\CC^{\alpha}(\Boule_1)}
      \lesssim
      R^{-1-\frac{d}{q}}\big\|\big(\frac{u}{\sqrt{T}},\nabla u\big)\big\|_{\LL^q(\Boule_{2})}
      +R^{\alpha-1}\|f\|_{\CC^{\alpha}(\Boule_{2})}.
    \end{align*}
    Finally, (\ref{cw18}) follows from inserting (\ref{cw14}) into the above estimate.
\end{proof}
\end{document}

%% file: Commandes.tex
\newcommand{\R}{\mathbb{R}}

\newcommand{\Z}{\mathbb{Z}}

\newcommand{\LL}{{\rm{L}}}
\newcommand{\CC}{{\rm{C}}}
\newcommand{\HH}{{\rm{H}}}

\newcommand{\Besov}{{\rm B}}

\newcommand{\lt}{}
\newcommand{\rt}{}
\def\[{[\![}
\def\]{]\!]}

\newcommand{\dd}{{\rm{d}}}

\newcommand{\Supp}{{\rm{supp}}}

\newcommand{\ee}{{\rm{e}}}
\newcommand{\ii}{{\rm{i}}}
\newcommand{\Id}{{\rm Id}}

\newcommand{\FF}{\mathcal{F}}

\newcommand{\Boule}{{\rm B}}

\def\Xint#1{\mathchoice
  {\XXint\displaystyle\textstyle{#1}}%
  {\XXint\textstyle\scriptstyle{#1}}%
  {\XXint\scriptstyle\scriptscriptstyle{#1}}%
  {\XXint\scriptscriptstyle\scriptscriptstyle{#1}}%
  \!\int}
\def\XXint#1#2#3{{\setbox0=\hbox{$#1{#2#3}{\int}$}
    \vcenter{\hbox{$#2#3$}}\kern-.5\wd0}}

\def\fint{\Xint-}

\renewcommand{\epsilon}{\varepsilon}
\renewcommand{\tilde}{\widetilde}
\newcommand{\langl}{\langle}
\newcommand{\rangl}{\rangle}

%% file: CommandesANG.tex
\newcommand{\et}{\quad\text{and }}
\newcommand{\dans}{\quad\text{in }}
\newcommand{\si}{\quad\text{if }}

\newcommand{\pour}{\quad\text{for }}
\newcommand{\pourtout}{\quad\text{for all }}

%% file: Commandes_adhoc.tex
\newcommand{\ubar}{\bar{u}}
\newcommand{\vbar}{\bar{v}}

\newcommand{\abar}{\bar{a}}

\newcommand{\rstar}{r_\star}

\newcommand{\QU}{\mathcal{Q}}

\newcommand{\barC}{{C}}

\newcommand{\chiu}{u}
\newcommand{\Cmu}{c}